\documentclass{article}

\usepackage{PRIMEarxiv}
\usepackage[utf8]{inputenc} 
\usepackage[T1]{fontenc}    
\usepackage{hyperref}       
\usepackage{url}            
\usepackage{booktabs}       
\usepackage{amsfonts}       
\usepackage{nicefrac}       
\usepackage{microtype}      
\usepackage{lipsum}
\usepackage{fancyhdr}       
\usepackage{graphicx}       
\graphicspath{{media/}}     
\usepackage{subcaption}
\usepackage{xfrac}

\usepackage{booktabs} 

\usepackage{graphicx}
\usepackage{amsmath}
\usepackage{amsthm}
\usepackage[version=4]{mhchem}
\usepackage{siunitx}
\usepackage{longtable,tabularx}
\usepackage{dsfont}
\usepackage{calc}
\usepackage{color}
\usepackage[ruled,vlined]{algorithm2e}
\usepackage{verbatim}
\usepackage{bm}

\newcommand{\norm}[1]{\left\lVert#1\right\rVert}
\newcommand{\abs}[1]{\left|#1\right|}
\newcommand{\menge}[1]{\left\lbrace#1\right\rbrace}

\newcommand{\intD}{\;\mathrm{d}}
\newcommand{\rom}[1]{\uppercase\expandafter{\romannumeral #1\relax}}

\newtheorem{theorem}{Theorem}
\newtheorem{lemma}{Lemma}
\newtheorem{corollary}{Corollary}
\newtheorem{definition}{Definition}

\newcommand{\inner}[1]{\left< #1 \right>}

\newcommand{\fu}{\mathbf{u}}
\newcommand{\ofu}{\overline{{\mathbf{u}}}}
\newcommand{\ofus}{{\overline{\mathbf{u}}_{\#}}}
\newcommand{\fv}{\mathbf{v}}
\newcommand{\fw}{\mathbf{w}}

\newcommand{\fx}{\mathbf{x}}
\newcommand{\fm}{\mathbf{m}}
\newcommand{\fg}{\mathbf{g}}

\newcommand{\balpha}{{\boldsymbol{\alpha}}}
\newcommand{\bbeta}{{\boldsymbol{\beta}}}

\newcommand{\Tr}{\top}

\title{Structure-preserving neural networks for the regularized entropy-based closure of the Boltzmann moment system
\thanks{
Notice:  This manuscript has been authored, in part, by UT-Battelle, LLC, under contract DE-AC05-00OR22725 with the US Department of Energy (DOE). The US government retains and the publisher, by accepting the article for publication, acknowledges that the US government retains a nonexclusive, paid-up, irrevocable, worldwide license to publish or reproduce the published form of this manuscript, or allow others to do so, for US government purposes. DOE will provide public access to these results of federally sponsored research in accordance with the DOE Public Access Plan (http://energy.gov/downloads/doe-public-access-plan).
}

\author{
  Steffen Schotthöfer \\
  Computer Science and Mathematics Division \\
  Oak Ridge National Laboratory \\
  Oak Ridge, TN 37831 USA\\
  \texttt{schotthofers@ornl.gov} \\
   \And
  M.~Paul Laiu\\\
  Computer Science and Mathematics Division \\
  Oak Ridge National Laboratory \\
  Oak Ridge, TN 37831 USA\\
  \texttt{laiump@ornl.gov} \\
  \And
  Martin Frank\\\
  Faculty of Mathematics \\
  KIT \\
  Karlsruhe, Germany\\
  \texttt{martin.frank@kit.edu} \\
   \And
  Cory D. Hauck\\\
  Computer Science and Mathematics Division \\
  Oak Ridge National Laboratory, and
Department of Mathematics (Joint Faculty),  University of Tennessee \\
  Oak Ridge, TN 37831 USA\\
  \texttt{hauckc@ornl.gov} \\
}}

\begin{document}
\maketitle

\begin{abstract}
The main challenge of large-scale numerical simulation of radiation transport is the high memory and computation time requirements of discretization methods for kinetic equations. In this work, we derive and investigate a neural network-based  approximation to the entropy closure method to accurately compute the solution of the multi-dimensional moment system with a low memory footprint and competitive computational time. We extend 
methods developed for the standard entropy-based closure 
to the context of regularized entropy-based closures. The main idea is to interpret structure-preserving neural network approximations of the regularized entropy closure as a two-stage approximation to the original entropy closure. We conduct a numerical analysis of this approximation and investigate optimal parameter choices. Our numerical experiments demonstrate that the method has a much lower memory footprint than traditional methods with competitive computation times and simulation accuracy. The code and all trained networks are provided on GitHub\footnote{\url{https://github.com/ScSteffen/neuralEntropyClosures}}$^,$\footnote{\url{https://github.com/CSMMLab/KiT-RT}}. 
\end{abstract}

\keywords{Kinetic Theory \and Moment Methods \and Entropy Closure \and Regularized Optimization \and Neural Networks}

\section{Introduction}
Numerically solving kinetic equations is at the heart of many physics-based applications, including neutron transport~\cite{NeutronTransport}, radiation transport~\cite{Chahine1987FoundationsOR}, semiconductors~\cite{Markowich1990SemiconductorE} and rarefied gas dynamics~\cite{BoltzmannApplications}.
A kinetic equation is a high-dimensional integro-differential equation formulated over a position-momentum phase space that can be as large as six dimensions. The high dimensionality of the phase space presents a severe computational challenge for large-scale numerical simulations. In addition, kinetic equations possess intricate structural properties: Galilean invariance, conserved quantities, and entropy dissipation (which implies hyperbolicity and an H-Theorem) \cite{Levermore1996MomentCH}.
It is key to the stability of a numerical method to explicitly preserve as many structural properties as possible.

Several methods have been proposed to solve kinetic equations, which are generally divided into nodal and modal methods. Nodal methods (also known as discrete velocity or discrete ordinates methods) evaluate the velocity space at specific points, yielding a system of equations only coupled by a quadrature approximation for the integral scattering operator. While computationally efficient, these methods suffer from numerical artifacts, which are called ray effects \cite{lathrop1968ray} when few discrete velocities are used, and on the other hand have a very large memory footprint when many velocities are used.

Modal methods (also known as moment methods) eliminate the velocity dependence of the phase space by computing the moment hierarchy of the kinetic equation. Due to the intrinsic structure of the advection term the moment system requires a closure~\cite{Levermore}. Modal methods differ in the closure model for the kinetic density in the transport term. The classical P$_N$ closure reconstructs the kinetic density linearly from the moment basis, which is computationally very fast. Its main drawbacks are spurious numerical artifacts such as large oscillations of the solution in the presence of highly anisotropic particle distributions, which may even result in negative solutions in the particle density ~\cite{GarretHauck}.  Filtering procedures may mitigate these oscillations \cite{MCCLARREN20105597}, but ensuring positivity requires a sophisticated limiting strategy \cite{laiu2016positive,laiu2019positive,laiu2019positivity,hauck2010positive}.

A moment closure, which preserves almost all important physical and mathematical properties of the Boltzmann equation (the one exception being the superposition principle, cf.\ Section \ref{sec_momMethods}) is constructed by solving a convex, constrained optimization problem based on the principle of (mathematical) entropy minimization \cite{Levermore}. Convexity is the mathematical key, because it implies entropy dissipation. The drawback is that this so-called M$_N$ closure is far more expensive to compute, since an optimization problem has to be solved at each closure evaluation. Furthermore, this optimization problem is ill-conditioned in physically relevant situations, while gradient and Hessian evaluations usually require expensive quadratures.
Making solving entropy-based closures practically feasible has been studied in a series of papers: First experiments were performed in \cite{hauck2011entropy}; the structural properties of the optimization problem were investigated in \cite{AlldredgeHauckTits}; in \cite{KRISTOPHERGARRETT2015573}, it was demonstrated that entropy-based closures can be scaled onto large HPC systems; the ill-posedness of the optimization problem was addressed by applying regularization techniques \cite{AlldredgeFrankHauck}; the regularization introduces a numerical error that can be controlled and balanced with the spatio-temporal discretization error \cite{Alldredge_2023}). 

Recently, neural networks have been used to to relieve the computational burden of the optimization problem. One key idea is to use input-convex neural networks to approximate the entropy functional, and thereby close the system \cite{porteous2021datadriven}. Further, in \cite{pmlr_v162_schotthofer22a}, data sampling strategies and approximation errors of the neural network-based surrogate models have been investigated. Although successfully applied to low-order entropy closures, these neural network-based closures face severe challenges when approximating high-order closures in higher spatial dimensions~\cite{pmlr_v162_schotthofer22a}. Other approaches use a direct neural network-based modeling of the closure of the moment system. Examples include encoder-decoder networks \cite{Han21983}, Galilean invariant machine learning methods \cite{huang2020learning,weinan_Integrating_Machine_learning}, U-nets that consider global information of the solution field \cite{li_grad_global}, and learning the spatial gradient of the highest-order moment \cite{huang2021machine1, christlieb2024hyperbolic}. Contrary to entropy-based closures, the approach in \cite{huang2021machine1} requires additional work to ensure hyperbolicity \cite{huang2021machine2} and correct characteristic speeds \cite{huang2023machine3}.  It also requires training with kinetic data, which can be expensive to obtain.

This work presents a structure-preserving neural network-based surrogate model combining the regularized formulation of the entropy-based closure \cite{AlldredgeFrankHauck} with structure-preserving neural network-based closure approximations \cite{porteous2021datadriven}. In particular, we explore the difficulties associated with the non-regularized entropy closure in terms of neural network approximation and data sampling. To address these issues, we develop a normalized, regularized framework that generates input convex neural network approximations to the regularized entropy closure, which itself serves as a convex approximation to the non-regularized entropy closure.
We conduct an error analysis of this multi-step approximation, which includes neural network approximation error, regularization error, and scaling errors of the partially regularized closure. Lastly, we perform extensive numerical experiments comparing different neural network approximations for various higher-order entropy closures of different regularization levels in synthetic tests, as well as established simulation benchmarks in the field of radiation transport. The trained network graphs are incorporated into the open-source high-performance solver KiT-RT~\cite{Kitrt_paper}.

\section{Entropy-based moment closures}\label{sec_kt}
\subsection{Kinetic equations}\label{sec_kinetic_theory}
We consider a linear kinetic equation
\begin{equation}  \label{eq_boltzmann}
	\partial_t f+\fv \cdot \nabla_{\mathbf{x}} f = Q(f),
\end{equation}
which describes the evolution of a kinetic distribution function $f(t,\mathbf x, {\fv})$ in a many-particle system with phase space dependence  space $\mathbf x\in \mathbf X\subset\mathbb{R}^3$, and velocity ${\fv}\in  \mathbb{S}^2=\menge{{\fv}\in\mathbb{R}^3: \norm{{\fv}}=1}$ . 
Throughout this paper, we denote the Euclidean norm by $\norm{\cdot}$. The left-hand side of Eq.~\eqref{eq_boltzmann} models particle advection, and the right-hand side models interaction  with the background medium of other particles.  
We use the linear collision operator in this work, which reads
\begin{equation}
	Q(f)(\mathbf{v})=\int_{\mathbb{S}^2} k(\mathbf{x}, \fv\cdot\fv_*) \left[ f({\fv}_*)-f({\fv})\right] d{\fv}_*,
\end{equation}
where the collision kernel $k$ models the strength of collisions at different velocities. 
We consider in this paper a space-dependent, isotropic collision kernel, denoted by $\sigma_{\textup{s}}(\mathbf{x})$, however, the results can be extended to arbitrary collision kernels using the methods of~\cite{AlldredgeFrankHauck}. 
Further, we abbreviate the integral over the velocity space by
\begin{equation}
\inner{\cdot} = \int_{\mathbb{S}^2}\cdot \intD \mathbf{v}.
\end{equation}
Equation~\eqref{eq_boltzmann} becomes a well-posed problem with suitable boundary and initial conditions.

The linear kinetic equation possesses some key structural properties, which are intricately related to the physical processes and its mathematical existence and uniqueness theory~\cite{AlldredgeFrankHauck, Levermore1996MomentCH}. First, the time evolution of the solution is invariant in range, i.e. if $f(0,\mathbf{x},\mathbf{v})\in B$, then $f(t,\mathbf{x},\mathbf{v})\in B$ for all $t>0$. In this work, we have $B=[0,\infty)$. Particularly this implies the non-negativity of $f$, if the initial condition $f_0$ is non-negative.
Second, if $\varphi$ is a collision invariant, i.e.
\begin{align}\label{eq_collision_invariant}
    	\inner{\varphi Q(g)} = 0,\quad \forall g\in\mathrm{Dom}(Q),
\end{align}
then the equation
\begin{align}
    	\partial_t\inner{\varphi f} + \nabla_{\mathbf x}\cdot\inner{{\fv}\varphi f} = 0
\end{align}
is a local conservation law. Third, for each fixed direction $\mathbf{v}$, the advection operator, i.e. the left-hand side term of Eq.~\eqref{eq_boltzmann}, is hyperbolic in space and time. Fourth, for $D\subseteq\mathbb{R}_+$, any twice continuously differentiable, strictly convex function $\eta: D\rightarrow\mathbb{R}$ is called kinetic entropy density. It yields the local entropy dissipation law
\begin{align}\label{eq_kinetic_entropy_dissipation}
	\partial_t\inner{\eta(f)} + \nabla_{\mathbf x}\cdot\inner{{\fv}\eta(f)}= \inner{\eta'(f)Q(f)} \leq 0.
\end{align}
Thus the particle system always strives for a state of minimal mathematical (i.e. maximal physical) entropy. The domain $D$ of the entropy density is often consistent with the physical bounds $B$ of the problem and we assume $D=B=[0,\infty)$ in the following.
Next, the solution $f$ fulfills the H-theorem that characterizes equilibria of the linear kinetic equation.
Lastly,  the solution $f$ will be invariant under rotations in $(\fx,\fv)$ and translations in $\fx$, as long as such transformations are allowed by the kernel $k$.

\subsection{Entropy-based closures}\label{sec_momMethods}
Moment methods encode the velocity dependence of the linear kinetic equation by multiplication of $f$ with a vector of velocity-dependent basis functions ${\fm}(\mathbf{v}):\mathbb{S}^2\rightarrow\mathbb{R}^{n+1}$ that consist of polynomials up to total degree $N$%
\footnote{In general $n+1 = (N+1)^2$.  For special geometries which allow reduction to two spatial dimensions, the number of non-zero moments is $\frac12(N+1)(N+2)$ \cite{GarretHauck}, and for one-dimensional slab geometries, it is $N+1$ \cite{lewis1984computational}.}
and subsequent integration over $\mathbb{S}^2$. {The basis $\fm$ is decomposed into polynomial blocks  $\fm =[m_0,\,\fm_1^\top,\dots,\fm_N^\top]^\top$.}
The solution of the resulting moment equation is the moment vector $\mathbf u_f: [0, T]\times \mathbf{X}\rightarrow\mathbb{R}^{n+1}$, which is defined as
\begin{align}\label{eq_momentDef}
 \mathbf u_f(t,\mathbf x)=\inner{{\fm}({\fv}) f(t,\mathbf{x},{\fv})}.
 \end{align}
Common choices for the basis functions are monomials or spherical harmonics, depending on the application. {In the following, we consider spherical harmonics basis functions.} Typically, they include the collision invariants defined in Eq.~\eqref{eq_collision_invariant}.
The moment vector satisfies the system of transport equations
\begin{align}\label{eq_momentRTa}
\partial_t \mathbf u_f(t,\mathbf{x}) + \nabla_\mathbf{x}\cdot\inner{\mathbf{v} \otimes{\fm}(\mathbf{v})f}= \inner{{\fm}(\mathbf{v})Q(f)},
\end{align}
which is called the moment system.
By construction, the advection and collision operators still depend on $f$, and thus the moment system is not closed. Moment methods aim to find a meaningful closure for this system.
Since the kinetic equation dissipates entropy and fulfills a local entropy dissipation law, a suitable closure can be built by reconstructing an ansatz density $f_\fu$ from moment $\fu$ such that $f_\fu$ is the one with minimal entropy among all functions that satisfy the moment condition $\mathbf u=\inner{{\fm} g}$. We denote the elements of $\fu=[ u_0, \dots, u_n]^\top$ and in polynomial degree notation as $\fu=[u_0,\fu_1^\top,\dots,\fu_N^\top]^\top$.
The entropy-based closure, or M$_N$ method, can be formulated as a constrained optimization problem for a given vector of moments $\mathbf u  \in \mathbb{R}^{n+1}$, i.e., the ansatz $f_\fu$ solves
\begin{align}\label{eq_entropyOCP}
\min_{g\in F_{{\fm}}} \inner{\eta(g)}\quad  \text{ s.t. } \mathbf u=\inner{{\fm} g},
\end{align}
where $F_{{\fm}}=\menge{g\in \text{Dom}(Q): \text{Range}(g)\subseteq D \text{ and} \inner{{\fm} g}<\infty}$.
A necessary condition for the existence of minimizers in \eqref{eq_entropyOCP} is that $\fu$ is an element of the \textit{realizable set}
\begin{align}
	\mathcal{R}=\menge{\mathbf u: \inner{{\fm}g}=\mathbf u,\, g\in F_{{\fm}}}.
\end{align} 
When a minimizer $f_\fu$ exists%
\footnote{Even if $\fu\in\mathcal{R}$, there may not exist a solution to \eqref{eq_entropyOCP} \cite{Junk1999,Hauck2008ConvexDA} in general. In the setting of this work, there is always a solution.}
for Eq.~\eqref{eq_entropyOCP}, it is unique and takes the form
\begin{align}\label{eq_entropyRecosntruction}
	{f}_\fu({\fv}) = \eta'_*(\balpha_\fu\cdot {\fm}({\fv})),
\end{align}
where the Lagrange multiplier  $ \balpha_{\fu}\in\mathbb{R}^{n+1} $ is the solution of the convex dual problem\footnote{These dual problems are often formulated as concave maximization problems in the literature~\cite{boyd2004convex}. In this paper, we choose to formulate them as minimization problems.}
\begin{align}\label{eq_entropyDualOCP}
	\balpha_\fu =  \underset{\balpha\in\mathbb{R}^{n+1}}{\text{argmin}}
	\,\phi(\balpha;\fu)\quad\text{with}\quad
  	\phi(\balpha;\fu):= \inner{\eta_*(\balpha\cdot {\fm})} - \balpha\cdot \fu 
\end{align}
and $\eta_*$ is the Legendre dual of $\eta$.
The moment system \eqref{eq_momentRTa} is then closed by replacing $f$ with the minimizer $f_\fu$ of Eq.~\eqref{eq_entropyOCP}, which leads to
\begin{align}\label{eq_momentRT}
	\partial_t \fu(t,\mathbf{x}) + \nabla_\mathbf{x}\cdot\inner{\mathbf{v} \otimes{\fm}(\mathbf{v})f_\fu}= \inner{{\fm}(\mathbf{v})Q(f_\fu)}.
\end{align}
In this case, the minimum of \eqref{eq_entropyOCP},
\begin{equation}
    h(\mathbf u):=\inner{\eta(f_{\mathbf u})},
\end{equation}
is a convex entropy for the closed moment system \eqref{eq_momentRT} (see, e.g., \cite{AlldredgeFrankHauck}). 
By strong duality, the minimum of \eqref{eq_entropyOCP} equals the negative of the minimum of  \eqref{eq_entropyDualOCP}. Thus, an equivalent expression for $h$ is given by \begin{align}\label{eq_entropyFunctionalH}
h(\fu) = -\phi(\balpha_\fu;\fu)=  \balpha_\fu\cdot \fu - \inner{\eta_*(  \balpha_\fu\cdot {\fm})},
\end{align}
and it is well-known \cite{AlldredgeFrankHauck} that $h$ is twice differentiable and convex.
The optimality and uniqueness of the primal solution $f_\fu$ in  Eq.~\eqref{eq_entropyRecosntruction} is then confirmed by using the strong duality of \eqref{eq_entropyOCP} and \eqref{eq_entropyDualOCP}.
In particular, the first-order optimality condition of  Eq.~\eqref{eq_entropyDualOCP} leads to
\begin{align}\label{eq_recons_u}
	\fu=\inner{{\fm} \eta'_*(\balpha_\fu\cdot {\fm})},
\end{align}
which
yields the inverse of the solution map $\fu\mapsto\balpha_\fu$ of the dual optimization problem.
Furthermore,  the derivative of $h$ recovers the optimal Lagrange multipliers of Eq.~\eqref{eq_entropyDualOCP}, i.e.,
\begin{align}\label{eq_derivH}
\nabla_\fu h (\fu) =  \balpha_\fu,
\end{align}
and the Hessian of the dual objective function $\phi$ with respect to $\balpha$ is given by
\begin{align}\label{eq_hessian}
	H(\balpha)= \inner{{\fm}\otimes {\fm}\,\eta_*^{\prime\prime}(\balpha\cdot {\fm})}.
\end{align}
This entropy-based closure also conserves the structural properties of the linear kinetic equation~\cite{Levermore, AlldredgeFrankHauck} stated in Section~\ref{sec_kinetic_theory}.
In the following, we consider the closure corresponding to the Maxwell-Boltzmann entropy with $D=[0,\infty)$,
\begin{align}\label{eq_MB_entropy}
	\eta(f)=f\log(f)-f,\quad \eta'(f) = \log(f),\quad \eta_*(z) = \exp(z),\quad \eta_*'(z)=\exp(z)\:.
\end{align}

\subsection{Approximation of entropy-based closure with neural networks}\label{sec_challenges}
The entropy-based closure faces practical challenges.
In particular, solving the convex optimization problem of Eq.~\eqref{eq_entropyDualOCP} at every time step in each grid cell of a numerical simulation comes at a high computational cost, especially near the boundary of ${\mathcal{R}}$, where the condition number of the Hessian in the optimization \eqref{eq_hessian} diverges~\cite{KRISTOPHERGARRETT2015573, AlldredgeHauckTits}.
This issue motivates the neural network-based entropy closures proposed in~\cite{porteous2021datadriven, pmlr_v162_schotthofer22a}, where a convex neural network is used to approximate the entropy function $h$ and to construct an approximate ansatz to close the moment system \eqref{eq_momentRTa}.
In previous work~\cite{pmlr_v162_schotthofer22a}, $h$ is approximated on 
\begin{align}
\label{eq:Rbar}
\overline{\mathcal{R}} &= \menge{\fu\in\mathcal{R}: u_0=1}\subset\mathbb{R}^{n+1}
\end{align}
by an input convex neural network N$_\theta$ with parameters $\theta$, i.e., $\forall \overline{\fu}\in\overline{\mathcal{R}}$,
\begin{align}
\label{eq:nn}
	\text{N}_\theta(\overline{\fu})\approx h(\overline{\fu}) \quad\text{ and } \quad \nabla_\fu\text{N}_\theta (\overline{\fu}) \approx \balpha_{\ofu}.
\end{align}
This neural network-based closure reduces the computational effort by several orders of magnitude from that of an iterative optimizer \cite{pmlr_v162_schotthofer22a}.

Even though the neural network-based closure reduces the computation cost, it still suffers from the numerical difficulties near the boundary of the realizable set.
To facilitate the discussion, we introduce 
the reduced, normalized realizable set 
\begin{align}\label{eq_reduced_normalizer_realizable_set}
\widetilde{\mathcal{R}} &= \menge{\mathbf{w}\in\mathbb{R}^{n}:[1,{\mathbf{w}}^\Tr]^\Tr\in\overline{\mathcal{R}}}\subset\mathbb{R}^{n}\:.
\end{align}
Note that both $\overline{\mathcal{R}}$, defined in Eq.~\eqref{eq:Rbar}, and $\widetilde{\mathcal{R}}$ are bounded sets.
\begin{definition}\label{def_fruncation_operator}
We define the normalization operator
\begin{align}\label{eq_normalization_operator}
 \overline{(\cdot)}\,\colon\mathbb{R}^{{n+1}}\to\mathbb{R}^{n+1}\: \text{ defined such that }\: \fu \mapsto \overline{\fu}	:= \tfrac{\fu}{u_0}
\end{align}
and the ``fruncation'' operator 
\begin{align}\label{eq_reduction_operator}
   	(\cdot)_\#\colon \mathbb{R}^{n+1}\to\mathbb{R}^{n}\:\text{ defined such that }\:  \fu \mapsto \fu_\# := [\fu_1^\top,\dots,\fu_N^\top]^\top \in \mathbb{R}^{n}\: .
\end{align}
Further, repeated application of the fruncation operator $(\cdot)_\#$ to the moment vector removes moments associated with the (currently) lowest degree basis functions, e.g. $({(\fu)_{\#}})_{\#}=[\fu_2^\top,\dots,\fu_N^\top]^\top$. 
\end{definition}
    Therefore, $\fu(=[u_0,(\fu_\#)^\top]^\top)\in\mathcal{R}$ if and only if
$\overline{\fu}(=\frac{\fu}{u_0})\in\overline{\mathcal{R}}$ and $\ofus \in \widetilde{\mathcal{R}}$.

It is known \cite{AlldredgeFrankHauck,AlldredgeHauckTits,GarretHauck, pmlr_v162_schotthofer22a} that when $\overline{\fu}$ approaches the boundary $\partial\overline{\mathcal{R}}$, the associated multiplier $\balpha_{\ofu}$ becomes unbounded. Furthermore, the condition number of the Hessian for the associated optimization problem grows, and at   $\partial\overline{\mathcal{R}}$ the problem is singular. For example, in a reduced, one-dimensional slab geometry, the kinetic density $f$ associated to $\overline{\fu}\in\partial\overline{\mathcal{R}}$ is a sum of finitely many $\delta$-functions~\cite{Curto_recursiveness,Monreal_210538}, i.e.,
\begin{align}\label{eq_delta_sum}
	f({\fv}) = \sum_{i=1}^l c_i\delta(p_i({\fv})),
\end{align}
where $\delta$ is a Dirac distribution, $c_i>0$ and each $p_i$ is a rational function depending on $\overline{\fu}$. 
In the following theorem, we describe the limit behavior of $h$ as $\overline{\fu}$ approaches $\partial\overline{\mathcal{R}}$.
Its proof is given in \ref{sec_appendix_kt}.
\begin{theorem}\label{theo_diverging_entropy} Given a slab geometry and Maxwell-Boltzmann kinetic entropy \eqref{eq_MB_entropy}, the moment entropy function $h$ diverges to infinity as 
$\overline{\fu}\rightarrow\partial\overline{\mathcal{R}}$.
\end{theorem}

The divergent behavior of $h$, as well as the unbounded growth of $\balpha_{\ofu}$ heavily
affects accuracy of the neural network approximation \eqref{eq:nn} near $\partial\overline{\mathcal{R}}$ and causes numerical issues already during training.
Due to numerical overflow, it is infeasible to sample $h$ or $\balpha_{\ofu}$ near or at $\partial\overline{\mathcal{R}}$ as training data, especially when training on GPUs in single precision.
Although a bound for the neural network approximation error in the interior of $\overline{\mathcal{R}}$ has been established in \cite{pmlr_v162_schotthofer22a}, there is no control over the approximation error for moments near or at $\partial\overline{\mathcal{R}}$.
Although a neural network-based approximation to the entropy functional is technically defined for $\mathbf u\not\in\mathcal{R}$, the extrapolation error is expected to be prohibitively large outside of the realizable set.%

\section{Regularized entropy-based moment closures}\label{sec_reg_moment}

To overcome the challenge of approximating the entropy closure near the boundary of the realizable set, we adopt a regularized version of the entropy closure that was introduced in \cite{AlldredgeFrankHauck}.  We adopt this regularization strategy to improve the neural network approximation. We review the regularized entropy-based closure proposed in \cite{AlldredgeFrankHauck} in Section~\ref{sec_full_regularized}, propose a partially regularized extension of the entropy-based closure in Section~\ref{sec_reg_normlized}, and analyze the structural properties of the partially regularized closure in Section~\ref{sec_part_reg_prop}.

\subsection{Fully regularized entropy-based closure}
\label{sec_full_regularized}

In \cite{AlldredgeFrankHauck}, the minimization problem \eqref{eq_entropyOCP} is replaced by a regularized problem
\begin{align}\label{eq_entropyOCP_reg}
	\min_{g\in F_{{\fm}}} \inner{\eta(g)} + \frac{1}{2\Gamma}\norm{\inner{{\fm}( {\fv})g}-\mathbf u}^2,
\end{align}
with regularization parameter $\Gamma>0$.  Unlike the non-regularized problem, the regularized entropy closure \eqref{eq_entropyOCP_reg} is feasible for all $\fu\in\mathbb{R}^{n}$. The minimizer $f^\Gamma_\fu$ still has the form given in  Eq.~\eqref{eq_entropyRecosntruction}:
\begin{align}\label{eq_entropyRecosntruction_reg}
	f^\Gamma_\fu = \eta'_*(\balpha_{\fu}^\Gamma\cdot {\fm})\:,
\end{align}
but $\balpha_{\fu}^\Gamma$ is the solution to a regularized dual problem
\begin{align}\label{eq_entropyDualOCP_reg}
	\balpha_{\fu}^\Gamma =  \underset{\balpha\in\mathbb{R}^{n+1}}{\text{argmin}} \menge{   \inner{\eta_*(\balpha\cdot {\fm})} - \balpha\cdot \mathbf u +\frac{\Gamma}{2}\norm{\balpha}^2}\:.
\end{align}
Finally, the Hessian of the dual objective function in Eq.~\eqref{eq_entropyDualOCP_reg} is given by
\begin{align}
   H^\Gamma(\balpha)= \inner{{\fm} \otimes{\fm}\eta_*''(\balpha\cdot {\fm})} + \Gamma I \:.
\end{align}
Thus the  condition number of $H^{\Gamma}(\balpha)$ is bounded from above by $1+\Gamma^{-1}\lambda_{\max}$, where $\lambda_{\max}$ is the largest eigenvalue of  $ H(\balpha)$ defined in  Eq.~\eqref{eq_hessian} \cite{AlldredgeFrankHauck}, whereas the condition number of the non-regularized problem is unbounded near $\partial\mathcal{R}$.
The fully regularized entropy closure inherits most of the structural properties of the standard closure; a notable exception is the linear superposition principle, and especially invariance under rescaling of $f$ \cite{AlldredgeFrankHauck}.
Since this scaling invariance is a key property that allows for constructing the neural network approximations introduced in Section~\ref{sec_challenges} on the bounded set $\widetilde{\mathcal{R}}$, we propose in the following section a partially regularized entropy-based closure that preserves the scaling invariance. 

\subsection{Partially regularized entropy-based closure}\label{sec_reg_normlized}
The fully regularized closure in  Eq.~\eqref{eq_entropyOCP_reg} regularizes the entire moment vector $\fu$.
In this work, we consider a partially regularized entropy-based closure that regularizes only the moments in $\fu_\#$ (see  Eq.~\eqref{eq_reduction_operator}) while maintaining the zeroth moment $u_0$.
To define the partially regularized entropy-based closure, we first introduce the decomposition of the velocity basis
\begin{align}
	{\fm}({\fv})=[m_0({\fv}),\mathbf{m}_\#({\fv})^\Tr]^\Tr\in\mathbb{R}^{n+1}
\end{align}
using the fruncation operator $\#$ defined in  Eq.~\eqref{eq_reduction_operator}.
The partially regularized closure is then given by\footnote{The factor $u_0$ in the regularization terms in  Eq.~\eqref{eq_entropyOCP_part_reg} and  Eq.~\eqref{eq_entropyDualOCP_part_reg_objectF} was not included in the fully or partially regularized closures considered in \cite{AlldredgeFrankHauck}. We include the factor $u_0$ in this paper to impose the scaling invariance of the closure and maintain the convexity of the entropy approximation. However, this factor introduces a small consistency error between the gradient of the entropy function and the Lagrange multiplier, as shown in Eq.~\eqref{eq_grad_h_thm} in Theorem~\ref{thm_grad_h}.}
\begin{align}\label{eq_entropyOCP_part_reg}
\min_{g\in F_{{\fm}}} \inner{\eta(g)} + \frac{1}{2 u_0 \gamma}\norm{\inner{{\fm}_\# g} - {\fu}_\#}^2\quad  \text{ s.t. }  u_0=\inner{m_0\,g}\:.
\end{align}
\begin{definition}\label{def_entropy}
    The entropy function $h^\gamma$ is defined such that, for given moments $\fu$ with $u_0>0$, $h^\gamma(\fu)$ takes the optimal objective function value of  Eq.~\eqref{eq_entropyOCP_part_reg}. The gradient of $h^\gamma$ is defined as 
    \begin{align}
    \mathbf{g}_{\fu}^\gamma :=  \nabla_{\fu} {h}^\gamma(\fu).
    \end{align}
\end{definition}
Similar to the fully regularized moment entropy considered in \cite{AlldredgeFrankHauck}, the partially regularized moment entropy $h^\gamma$ falls back to the standard, unregularized entropy $h$ as $\gamma\to 0$.
However, the partially regularized closure is only defined for ${\fu}\in\mathbb{R}^{n+1}$ with $u_0>0$. This condition is still less restrictive than the realizability condition $\fu\in\mathcal{R}$ required by the standard entropy-based closure and is much easier to enforce in a numerical solver.

If a solution to Eq.~\eqref{eq_entropyOCP_part_reg} exists, it is again of the same form as in Eq.~\eqref{eq_entropyRecosntruction_reg} \cite{AlldredgeFrankHauck},
except that the Lagrange multiplier $ \balpha_{{\fu}}^{\gamma}\in\mathbb{R}^{n+1}$ corresponding to ${{\fu}}$ is now given by
\begin{equation}\label{eq_entropyDualOCP_part_reg}
	\balpha_{{\fu}}^{\gamma} =  \underset{\balpha\in\mathbb{R}^{n+1}}{\text{argmin }}
	\phi^\gamma(\balpha;{{\fu}})
\end{equation}
with the partially regularized, dual objective function
\begin{equation}
	\phi^\gamma(\balpha;{{\fu}}) =\inner{\exp( \balpha\cdot {\fm})}- \balpha\cdot {{\fu}} + \frac{u_0\gamma}{2}\norm{\balpha_\#}^2.\label{eq_entropyDualOCP_part_reg_objectF}
\end{equation}
It is straightforward to verify that $\phi^\gamma$ is twice differentiable and strictly convex in $\balpha$.
By the strong duality of  Eqs.~\eqref{eq_entropyOCP_part_reg}--\eqref{eq_entropyDualOCP_part_reg} (see~\cite{Decarreau92} for the proof of strong duality), the partially regularized entropy function $h^\gamma\colon\mathbb{R}^{n+1}\to\mathbb{R}$ satisfies 
\begin{equation}\label{eq_entropy_phi_def}
    h^\gamma(\fu)= - \phi^\gamma(\balpha_\fu^\gamma;{{\fu}})
\end{equation}

In the following theorem, we show that, unlike the standard entropy-based closure case considered in  Eq.~\eqref{eq_derivH}, the multiplier $\balpha_{{\fu}}^{\gamma}$ and the gradient of $h^{\gamma}$ at $\fu$ for the partially regularized closure are not the same. 
\begin{theorem}\label{thm_grad_h}
Let $\gamma>0$ and let $\balpha_{{\fu}}^\gamma$ be as defined in  Eq.~\eqref{eq_entropyDualOCP_part_reg}, then the gradient of the entropy, see Definition~\ref{def_entropy} is given by
\begin{equation}\label{eq_grad_h_thm}
  \mathbf{g}_{\fu}^\gamma = \balpha_{\fu}^\gamma - \frac{\gamma}{2}\big[\|(\balpha_{\fu}^\gamma)_{\#}\|^2,\mathbf{0}^{\Tr}\big]^{\Tr}.
\end{equation}
\end{theorem}
\begin{proof}
    See \ref{sec_thm_grad_h_pf}
\end{proof}
This discrepancy between $\balpha_{\fu}^\gamma$ and $\nabla_{\fu}h^{\gamma}$ results in two potential choices of ansatz in the partially regularized closure, 
\begin{equation}\label{eq_ansatzes}
    	\widetilde{f}^\gamma_\fu = \eta'_*(\balpha_{\fu}^\gamma\cdot {\fm})\quad\text{and}\quad
f^\gamma_\fu = \eta'_*(\mathbf{g}_{\fu}^\gamma\cdot {\fm}),
\end{equation}
where $\mathbf{g}_{\fu}^\gamma$ denotes the gradient of the partially regularized entropy function as defined in Eq.~\eqref{eq_grad_h_thm}.
Here $\widetilde{f}^\gamma_\fu$ is the minimizer to the primal problem \eqref{eq_entropyOCP_part_reg} and 
\begin{align}\label{eq_rel_f_f_tilde}
f^\gamma_\fu = \exp\big(-\frac{\gamma}{2}\|(\balpha_{\fu}^\gamma)_{\#}\|^2\big)\,\widetilde{f}^\gamma_\fu.
\end{align} 
In this paper, we use $f^\gamma_\fu$, rather than $\widetilde{f}^\gamma_\fu$, in the closed moment system \eqref{eq_momentRT}. This choice is justified in Section~\ref{sec_part_reg_prop} by showing that the desirable structural properties of the moment system are preserved when $f^\gamma_\fu$ is used. 
In the remainder of this section, we focus on the procedures for obtaining $\balpha_{\fu}^\gamma$, which can then be used to compute $\mathbf{g}_{\fu}^\gamma$ as shown in  Eq.~\eqref{eq_grad_h_thm}.

In \cite{pmlr_v162_schotthofer22a,porteous2021datadriven}, it was shown that, in the standard entropy-based closure, the multiplier $\balpha_{{\fu}}$ can be obtained from $\balpha_{\ofu}$, which allows for the use of efficient sampling strategies in constructing neural network approximations on $\widetilde{\mathcal{R}}\subset\mathbb{R}^{n}$ rather than $\mathcal{R}\subset\mathbb{R}^{n+1}$. 
In the case of this partially regularized closure, the partially regularized ansatz $f^\gamma_\fu$ can be obtained following an analogous strategy, which we summarize in the steps below.
\begin{enumerate}
    \item Given $\fu\in\mathbb{R}^{n+1}$ with $u_0>0$, compute $\ofus\in\mathbb{R}^n$.
    \item Solve a reduced optimization problem for $\bbeta_{\ofus}^\gamma$ as stated in Eq.~\eqref{eq_reduced_reg_closure}.
    \item Compute $\balpha_{\ofu}^\gamma $ from $\bbeta_{\ofus}^\gamma$, using Eq.~\eqref{eq_alpha_reconstructor} with the definition of $\vartheta$ in Eq.~\eqref{eq_ref_alpha_reduction}.
    \item Compute $\balpha_{\fu}^\gamma$ from $\balpha_{\ofu}^\gamma$, using Eq.~\eqref{eq_scaled_alpha}.
    \item Compute $\mathbf{g}_{\fu}^\gamma$ from $\balpha_{\fu}^\gamma$, using Eq.~\eqref{eq_grad_h_thm} and then compute $f^\gamma_\fu$ using Eq.~\eqref{eq_ansatzes}.
\end{enumerate}
In the following analysis, we show that the steps above indeed lead to the partially regularized ansatz $f^\gamma_\fu$.
We start by showing an intrinsic property of $\balpha_{\ofu}^\gamma$.

\begin{lemma}\label{lem_alpha_0}
Assume $m_0(\mathbf{v})$ is constant and positive. Given a normalized moment $\overline{\fu}$, let $\balpha_{\ofu}^\gamma$ be the corresponding multiplier defined in  Eq.~\eqref{eq_entropyDualOCP_part_reg}. Then
\begin{align}\label{eq_alpha_0}
	{\alpha}_{\overline{\fu},0}^\gamma =\vartheta\left(\left(\balpha_{\ofu}^\gamma\right)_\#\right),
\end{align}
where the function $\vartheta\colon\mathbb{R}^{n}\to\mathbb{R}$ is defined as
\begin{align}\label{eq_ref_alpha_reduction}
	\vartheta(\bbeta)=-\frac{1}{m_0}\big(\log(m_0) +\log( \inner{\exp\left(\boldsymbol{\beta}\cdot{\fm}_\#\right)})\big)\:.
\end{align}
\end{lemma}

\begin{proof}
    See \ref{sec_lem_alpha_0_pf}.
\end{proof}
Note that this relation is specific for multipliers associated with the partially regularized problem \eqref{eq_entropyDualOCP_part_reg} and does not hold in general for multipliers given by the fully regularized problem~\eqref{eq_entropyDualOCP_reg}.

\begin{definition}\label{def_reduced_values} We define the reduced objective function
    $\hat{\phi}^\gamma(\,\cdot\,;\mathbf{w})\colon \mathbb{R}^{n}\to\mathbb{R}$ as
\begin{align}\label{eq_def_reduced_obj_func}
\hat{\phi}^\gamma\left(\bbeta;\mathbf w\right):=\phi^\gamma\left(\left[\vartheta(\bbeta),\bbeta^\Tr\right]^\Tr;\left[1,\mathbf{w}^\Tr\right]^\Tr\right)\quad\text{for all }\, \mathbf w\in\mathbb{R}^{n}
\end{align}
and $\phi^\gamma$ defined in Eq.~\eqref{eq_entropyDualOCP_part_reg_objectF}. Further, we define the minimizer $\bbeta_{\ofus}^\gamma$ of the reduced optimization problem 
\begin{align}\label{eq_reduced_reg_closure}
	\bbeta_{\ofus}^\gamma:=  \underset{\boldsymbol{\beta}\in\mathbb{R}^{n}}{\textup{argmin}} \,\hat{\phi}^\gamma(\bbeta;\ofus)\:.
\end{align}
Also, we define the reduced entropy $\hat{h}^\gamma:\mathbb{R}^n\rightarrow\mathbb{R}$ by 
\begin{align}\label{eq_reduced_entropy}
\hat{h}^\gamma(\fw):= -\hat{\phi}^\gamma({\bbeta}_{\fw}^\gamma;\fw).
\end{align}
\end{definition}
We next show in the following Lemmas that $\balpha_{\ofu}^\gamma$ can be obtained by solving a minimization problem in $\mathbb{R}^{n}$ for the objective function $\hat{\phi}^\gamma\left(\bbeta;\overline{\fu}_{\#}\right)$.
\begin{lemma}\label{lem_reduced_dual_obj}
Assume $m_0(\mathbf{v})$ is constant and positive. Then, $\hat{\phi}^\gamma$ of Eq.~\ref{eq_def_reduced_obj_func} is given by
\begin{align}\label{eq_reduced_obj_reformulation}
	\hat{\phi}^\gamma(\bbeta;\mathbf w)= \frac{1}{m_0} +\frac{1}{m_0}\left(\log(m_0) +\log\left( \inner{\exp\left(\boldsymbol{\beta}\cdot{\fm}_\#\right)}\right)\right) - \bbeta\cdot \mathbf w + \frac{\gamma}{2} \norm{\bbeta}^2,
\end{align}
which is strictly convex and twice differentiable with respect to $\boldsymbol{\beta}$.
\end{lemma}
\begin{proof}
    See \ref{sec_lem_reduced_dual_obj_pf}.
\end{proof}

\begin{lemma}\label{lem_reduced_dual_min}
Let $\ofus\in\mathbb{R}^{n}$ be the fruncation of a normalized moment $\overline{\fu}\in\mathbb{R}^{n+1}$. 
Let $\gamma>0$ and
\begin{align}
	\bbeta_{\ofus}^\gamma=  \underset{\boldsymbol{\beta}\in\mathbb{R}^{n}}{\textup{argmin}} \,\hat{\phi}^\gamma(\bbeta;\ofus)\:,
\end{align}
see Eq.~\eqref{eq_reduced_reg_closure} of Definition~\ref{def_reduced_values}.
Then 
\begin{equation}\label{eq_alpha_reconstructor}
    \balpha_{\ofu}^\gamma =
[\vartheta(\bbeta_{\ofus}^\gamma),(\bbeta_{\ofus}^\gamma)^\Tr]^\Tr
\end{equation}
solves the minimization problem in Eq.~\eqref{eq_entropyDualOCP_part_reg} at $\overline{\fu}$, 
and 
$\phi^\gamma(\balpha_{\ofu}^\gamma;{\ofu})= \hat{\phi}^\gamma({\bbeta}_{\ofus}^\gamma;\ofus).$
\end{lemma}
\begin{proof}
    See \ref{sec_lem_reduced_dual_min_pf}
\end{proof}
Lemma~\ref{lem_reduced_dual_min} implies that, for normalized moments $\overline{\fu}$, the multiplier $\balpha_{\ofu}^\gamma$ in Eq.~\eqref{eq_entropyDualOCP_part_reg} can be obtained by solving the reduced problem in Eq.~\eqref{eq_reduced_reg_closure}.
By setting $\gamma>0$, the partial regularization improves the condition number of the Hessian in  Eq.~\eqref{eq_reduced_reg_closure}, as shown in the following lemma. 
\begin{lemma}\label{lem_hessian_cond}
Let $\gamma>0$ and $\hat{H}^\gamma(\bbeta)$ denote the Hessian of $\hat{\phi}^\gamma$ with respect to $\bbeta$, then the condition number of $\hat{H}^\gamma(\bbeta)$ is bounded from above by $1+\gamma^{-1}\hat{\lambda}_{\max}$, where $\hat{\lambda}_{\max}$ is the maximum eigenvalue of $\hat{H}^{\gamma=0}(\bbeta)$.
\end{lemma}
\begin{proof}
    See \ref{sec_lem_hessian_cond_pf}.
\end{proof}
With these Lemmas, the following Theorem shows that (i) $\balpha_{{\fu}}^\gamma$ can be obtained from the solution to the reduced problem \eqref{eq_reduced_reg_closure}, (ii) $h^\gamma\colon\mathbb{R}^{n+1}\to\mathbb{R}$ can be written as an extension of a convex function on $\mathbb{R}^{n}$ and (iii) the gradient of $h^\gamma$ can be computed from $\balpha_{{\fu}}^\gamma$. Finally, using Eq.~\eqref{eq_ansatzes} the moment systemin Eq.~\eqref{eq_momentRTa} can be closed. 

\begin{theorem}\label{thm_reg_rescale}
Let $\gamma>0$ and $\fu=[u_0,(\fu_{\#})^\Tr]^\Tr\in\mathbb{R}^{n+1}$ with $u_0>0$. Let $\balpha_{{\fu}}^\gamma$ be as defined in  Eq.~\eqref{eq_entropyDualOCP_part_reg}. 
 Then, 
	\begin{align}\label{eq_scaled_alpha}
	\balpha_\fu^\gamma = 
 \balpha_{\ofu}^\gamma + \left[\frac{\log u_0}{m_0},0^\top\right]^\top 
\end{align}
where $ \balpha_{\ofu}^\gamma$ is given by Eq.~\eqref{eq_alpha_reconstructor}.
Also, the function $\hat{h}^\gamma$ defined in Eq.~\eqref{eq_reduced_entropy} of Definition~\ref{def_reduced_values} is strictly convex and
\begin{equation}\label{eq_grad_h_hat_thm}
    \nabla_{\fw} \hat{h}^\gamma(\fw)={\bbeta}_{\fw}^\gamma.
\end{equation}
Moreover, $h^\gamma({\ofu}) = \hat{h}^\gamma(\ofus)$ and
\begin{equation}\label{eq_full_entropy}
h^\gamma(\fu)=u_0\hat{h}^\gamma(\ofus)+ \frac{u_0}{m_0}\log \,u_0\:.
\end{equation}
\end{theorem}
\begin{proof}
    See \ref{sec_thm_reg_rescale_pf}.
\end{proof}
The relation in  Eq.~\eqref{eq_full_entropy} allows us to prove that $h^\gamma$ is strictly convex.
\begin{corollary}\label{corr_h_convexity}
    For any $\gamma>0$, $h^\gamma$ is strictly convex on the set $\{\fu=[u_0,(\fu_{\#})^\Tr]^\Tr\in\mathbb{R}^{n+1} \colon u_0>0\}$.
\end{corollary}
\begin{proof}
    The proof follows from \cite[Theorem~3.1]{porteous2021datadriven} once the strict convexity of $\hat{h}^\gamma$ is established, which is done in Theorem~\ref{thm_reg_rescale}.
\end{proof}

\subsection{Properties of the partially regularized closure}
\label{sec_part_reg_prop}

In this section, we analyze the moment reconstruction error with the ansatz $f^\gamma_\fu$ and show that using $f^\gamma_\fu$ in  Eq.~\eqref{eq_ansatzes} to close the moment system preserves the desirable structural properties.

\subsubsection{Moment reconstruction error}
Closing the moment system \eqref{eq_momentRTa} using the partially regularized ansatz $f^\gamma_\fu$ rather than the standard entropy-based ansatz $f_\fu$ (see  Eq.~\eqref{eq_ansatzes} and  Eq.~\eqref{eq_entropyRecosntruction} for the definitions) results in discrepancies in the flux term, i.e., $\langle\mathbf{v}\otimes\mathbf{m}f^\gamma_{\fu}\rangle$ vs. $\langle\mathbf{v}\otimes\mathbf{m}f_{\fu}\rangle$ and the collision term, i.e., $\langle\mathbf{m}Q(f^\gamma_{\fu})\rangle$ vs.  $\langle\mathbf{m}Q(f_{\fu})\rangle$.
To facilitate the discussion, we denote, for a given moment $\fu$, 
\begin{equation}
 \fu^\gamma =\langle\mathbf{m} f^\gamma_{\fu}\rangle\quad\text{and}\quad
 \widetilde{\fu}^\gamma =\langle\mathbf{m} \widetilde{f}^\gamma_{\fu}\rangle
\end{equation}
as the moments associated to the ans\"atze $f^\gamma_{\fu}$ and $\widetilde{f}^\gamma_{\fu}$ defined in Eq.~\eqref{eq_ansatzes}. 
Since the standard entropy-based closure recovers the moment then (up to optimization tolerance), $\fu = \langle\mathbf{m} f_{\fu}\rangle$.

We provide in Theorem~\ref{theo_reg_ansatz_error} an upper bound on the moment reconstruction error $\|\fu^\gamma - \fu\|$, which can be used to estimate the discrepancies in the flux term under the assumption of Lipschitz continuity~\cite{Alldredge_2023}.
In Theorem~\ref{theo_reg_ansatz_error}, we separate the error from partial regularization $\norm{\widetilde{\fu}^\gamma-\fu}$ and the error from the entropy gradient and multiplier discrepancy, $\norm{\fu^\gamma-\widetilde{\fu}^\gamma}$, shown in Theorem~\ref{thm_grad_h}.

\begin{definition}
    Given any positive finite constant, let $B_M:=\menge{\boldsymbol{\beta}\in\mathbb{R}^{n}:\norm{\boldsymbol{\beta}}<M}$ be the ball of radius $M$ in the Euclidean norm on $\mathbb{R}^n$.
\end{definition}

\begin{theorem}\label{theo_reg_ansatz_error}
Let $\gamma>0$, $\fu=[u_0,(\fu_{\#})^\Tr]^\Tr\in\mathbb{R}^{n+1}$ with $u_0>0$, and let $\boldsymbol{\beta}_{\ofus}^\gamma$ be as defined in  Eq.~\eqref{eq_reduced_reg_closure}.
Then 
\begin{equation}
    \widetilde{\fu}^\gamma = \fu - u_0 \gamma \, [0,(\boldsymbol{\beta}^{\gamma}_{\ofus})^\Tr]^\Tr.
\end{equation}
Further, suppose $\boldsymbol{\beta}_{\ofus}^\gamma\in B_M$, then
\begin{equation}\label{eq_regularization_error}
	\norm{\widetilde{\fu}^\gamma-\fu} =
	  u_0\gamma\,\|(\boldsymbol{\beta}^{\gamma}_{\ofus}) \| \leq u_0\gamma M,
\end{equation}
and
\begin{equation}
\label{eq:ugamma-diffs}
	\norm{\fu^\gamma-\fu} 	
	\leq  u_0\gamma M + \big(1-\exp(-\frac{\gamma}{2}M^2)\big)\|\widetilde{\fu}^\gamma\|
	\leq u_0\left(\gamma M + (n+1) \left(\frac{n+1}{\abs{\mathbb{S}^2}}\right)^{1/2}\left(1 - \exp\left(-\frac{\gamma}{2}M^2\right)\right)\right),
\end{equation}
where $\abs{\mathbb{S}^2}$ is the measure of the unit sphere.
\end{theorem}
\begin{proof}
    See \ref{sec_theo_reg_ansatz_error_pf}.
\end{proof}

The last error estimate is rather conservative since typically $\norm{\widetilde{\fu}^\gamma}\ll u_0\,(n+1)\left(\frac{n+1}{\abs{\mathbb{S}^2}}\right)^{1/2}$. In the numerical results (cf.\ Table~\ref{tab_ICNN_hohlraum}) we found that often the gross numerical error is dominated by the regularization error given in Eq.~\eqref{eq_regularization_error}.

\subsubsection{Entropy dissipation and hyperbolicity}
We show that by closing the moment system \eqref{eq_momentRTa} using the ansatz $f_{\fu}^\gamma$ given in  Eq.~\eqref{eq_ansatzes}, key structural properties, such as invariance of range, conservation, entropy dissipation, and hyperbolicity, are preserved in the closed moment system.
The invariance of range property follows directly from the fact that $\text{Range}(f_{\fu}^\gamma)=[0,\infty)$, which is identical to the range for the kinetic density $D=[0,\infty)$ considered in this paper.
As in the case analyzed in \cite{AlldredgeFrankHauck}, the partial regularization here does not affect the collision invariants of the collision operator $Q(\cdot)$ and thus preserves the conservation property.
In the following Theorem~\ref{theo_entropy_dissipation}, we show that the moment system closed with ansatz $f_{\fu}^\gamma$ is entropy dissipative and hyperbolic due to strict convexity of $h^\gamma$, which is proven in Corollary~\ref{corr_h_convexity}.

\begin{theorem}\label{theo_entropy_dissipation}
The partially regularized entropy function $h^\gamma$ is an entropy for the moment system closed using the ansatz $f_{\fu}^\gamma$. The associated entropy flux is given by
\begin{align}
\mathbf{j}^\gamma(\fu)=\inner{\mathbf{v}\,\eta(f^\gamma_{\fu})}.
\end{align}
With $h^\gamma$ and $\mathbf{j}^\gamma$, the entropy dissipation law
\begin{align}
\partial_t h^\gamma(\fu) + \nabla_{\mathbf{x}}\cdot\mathbf{j}^\gamma(\fu)\leq 0
\end{align}
holds.
Furthermore, the closed moment system
\begin{align}\label{eq_hyper_moment_sys}
\partial_t\fu + \nabla_{\mathbf{x}}\cdot\inner{\mathbf{v}\otimes\mathbf{m}f^\gamma_{\fu}} = \inner{\mathbf{m}Q(f^\gamma_{\fu})}
\end{align}
is symmetrizable hyperbolic. 
\end{theorem}
\begin{proof}
	See \ref{sec_theo_entropy_dissipation_pf}.
\end{proof}
The entropy dissipation and hyperbolicity results in Theorem~\ref{theo_entropy_dissipation} are specific for moment systems closed with ansatz $f_{\fu}^\gamma$, which is determined by the gradient of the underlying entropy function $h^\gamma$.
When the moment system is closed using the ansatz $\widetilde{f}_{\fu}^\gamma$ (see  Eq.~\eqref{eq_ansatzes}), the results in Theorem~\ref{theo_entropy_dissipation} no longer hold, due to the inconsistency between $\nabla_{\fu} h^\gamma$ and the partially regularized multiplier $\balpha_{\fu}^\gamma$.

\section{Neural network-based partially regularized entropy-based closure}
\label{sec_NN_part_reg}
In this section, we propose an approximation strategy for the partially regularized moment entropy function $h^\gamma$, the optimal value of \eqref{eq_entropyOCP_part_reg}, using neural networks.  Applying neural networks on the partially regularized problem with normalized moments \eqref{eq_reduced_reg_closure} improves both the training efficiency and the accuracy of the approximation.


 \begin{figure}[t]
	\begin{algorithm}[H]
    	\label{alg_network_training}
    	\input{assets/algorithms/alg_network_training.tex}
	\end{algorithm}
	\vspace{-1em}
\end{figure}
\begin{figure}[t]
	\begin{algorithm}[H]
    	\label{alg_network_inference_nr}
    	\input{assets/algorithms/alg_network_inference_nr.tex}
	\end{algorithm}
	\vspace{-1em}
\end{figure}

\subsection{Neural network approximation for the normalized, partially regularized entropy}
We extend the structure-preserving neural network-based entropy closure presented in~\cite{pmlr_v162_schotthofer22a} to the partially regularized entropy closure introduced in Section~\ref{sec_reg_normlized}.
As discussed at the end of Section~\ref{sec_reg_normlized}, the results of Theorem~\ref{thm_reg_rescale} allow us to construct strictly convex approximations to $\hat{h}^{\gamma}\colon\mathbb{R}^{n}\to\mathbb{R}$ defined therein and then extend the approximation following the formula given in  Eq.~\eqref{eq_full_entropy} to the full space $\{\fu=[u_0,(\fu_{\#})^\Tr]^\Tr\in\mathbb{R}^{n+1} \colon u_0>0\}$. The extended approximation then serves as an approximation to the partially regularized entropy function $h^{\gamma}$.

To this end, we use the input convex neural network (ICNN) \cite{Amos2017InputCN} to construct approximations $\hat{h}^p$ to the strictly convex function $\hat{h}^\gamma$, i.e.,
 \begin{align}\label{eq_alg_helper1}
	\hat{h}^p(\ofus):=\text{N}_{\theta^*}(\ofus)\approx \hat{h}^\gamma(\ofus)
 \quad\text{and}\quad
 \bbeta_{\ofus}^p := \nabla_{\fw}\hat{h}^p(\ofus),
 \end{align}
  where $p=(\gamma,\theta^*)$ with $\theta$ denoting the network parameters of an ICNN $\text{N}_{\theta}(\ofus)$ and $\theta^{*}$ the trained parameter values. The extension of $\hat{h}^p$, defined on $\mathbb{R}^n$, to ${h}^p$, defined on $\mathbb{R}^{n+1}$, is defined as in Eq.~\eqref{eq_full_entropy}:
  \begin{equation}\label{eq_extend_approximation}
      h^p(\fu) := u_0 \hat{h}^p(\ofus) +  \frac{u_0}{m_0}\log \,u_0\:.
  \end{equation}
  This gives a strictly convex approximation to the partially regularized entropy $h^\gamma$. The strict convexity of $h^p$ follows from the strict convexity of $\hat{h}^p$ and Corollary~\ref{corr_h_convexity}.
  When closing the moment system \eqref{eq_momentRT} with the approximate ansatz
  \begin{equation}
      f_\fu^p := \eta_*^\prime(\mathbf{g}_\fu^p\cdot\mathbf{m}) \quad\text{with}\quad
      \mathbf{g}_\fu^p := \nabla_{\fu} h^{p}(\fu)\:,
  \end{equation}
  the strict convexity of $h^{p}$ guarantees the entropy dissipation and hyperbolicity properties of the closed moment system; the proof follows exactly the proof of Theorem~\ref{theo_entropy_dissipation}.
  Here the gradient $\mathbf{g}_\fu^p$ can be computed from the ICNN approximation $\hat{h}^p$ and its gradient $\nabla_{\fw}\hat{h}^p$. Specifically\footnote{Here $\nabla_\ofus h(\ofu)=[\partial_{\overline{u}_1}h(\ofu),\dots,\partial_{\overline{u}_n}h(\ofu)]^\top\in\mathbb{R}^n$ is the vector of partial derivatives with respect to the last $n$ arguments of $h$ and by $\nabla_\fw\hat{h}(\fw)\in\mathbb{R}^n$ the full gradient of $\hat{h}$.},
  \begin{equation}
      \mathbf{g}_\fu^p = \nabla_{\fu} h^{p}(\fu) = [\partial_{u_0}h^{p}(\fu), \nabla_{\ofus} h^{p}(\fu)^\top  ]^\top,
  \end{equation}
  where it follows from  Eq.~\eqref{eq_extend_approximation} that
  \begin{equation}\label{eq_alpha_0_approx}
      \partial_{u_0}h^{p}(\fu) = \hat{h}^p(\ofus) - (\ofus)\cdot\nabla_{\fw}\hat{h}^p(\ofus) + \frac{\log u_0 + 1}{m_0} 
      \quad\text{and}\quad \nabla_{{\fu}_{\#}} h^{p}(\fu) = u_0\nabla_{\fw}\hat{h}^p(\ofus)
      \equiv u_0 \, \bbeta_{\ofus}^p.
  \end{equation}

This strategy just outlined above allows training of the network to be performed only on fruncated normalized moments in $\mathbb{R}^n$.
To define the training loss, we introduce the formal moment reconstruction maps for the regularized and non-regularized, normalized closures,
\begin{align}\label{eq_u_from_alpha}
\boldsymbol\psi^\gamma(\boldsymbol{\beta})=\inner{\mathbf{m}_\#\exp\left(\left[\vartheta({\bbeta}),{\bbeta}^\Tr\right]^\Tr\cdot\mathbf{m}\right)} + \gamma{\bbeta},\quad \text{and}\quad
 \boldsymbol\psi({\bbeta})= \boldsymbol\psi^{\gamma=0}({\bbeta}),
\end{align}
where $\vartheta$ is defined in Eq.~\eqref{eq_ref_alpha_reduction}, and by definition, $\ofus=\boldsymbol\psi^\gamma({\bbeta}_{\ofus}^\gamma)$.
 The corresponding training error for a data-set $X_{\text{Train}} =\menge{\overline{\fu}_{\#, i},{\bbeta}_{\overline{\fu}_{\#, i}}^\gamma, \hat{h}^\gamma(\ofu_{\#,i})}_{i=1}^T$ is given by
\begin{align}\label{eq_loss_icnn}
	\mathcal{L}(\theta;X_{\text{Train}}) =  \frac{1}{T}\sum_{i=1}^{T}
	\left|{\hat{h}^\gamma(\ofu_{\#,i}) - \hat{h}^p(\ofu_{\#,i})}\right|^2 +\norm{ {\bbeta}_{\overline{\fu}_{\#,i}}^\gamma-{\bbeta}_{\overline{\fu}_{\#,i}}^p}^2 +
	\norm{{\ofus}_{,i} -  \boldsymbol\psi^\gamma({\bbeta}_{\overline{\fu}_{\#,i}}^p)}^2,
\end{align}
where $T$ denotes the size of the training set.
For details of the network architecture, we refer to Section~\ref{sec_num_training}. The training workflow for finding $\theta^*$ is summarized in Algorithm~\ref{alg_network_training}, and inference within a kinetic solver is described in Algorithm~\ref{alg_network_inference_nr}.

\subsection{Neural network approximation error and data sampling}
In addition to the errors analyzed in Theorem~\ref{theo_reg_ansatz_error} a neural network-based approximation of $h^\gamma$ introduces an approximation error defined by the inference accuracy of the used neural network. It directly follows for the neural network-based reconstruction $\fu^p$ that
\begin{align}
    \norm{\fu^p -\fu} \leq \norm{\fu^p - \fu^\gamma} + \norm{\fu^\gamma - \fu},
\end{align}
where the second term of the right-hand side is specified in Theorem~\ref{theo_reg_ansatz_error}. In the following, we analyze the first term via the approximation error of the associated approximated gradient $\fg^p_{\ofus}$ under the assumption of global convexity of the neural network approximation $h^p$, i.e. an approximation given by Algorithm~\ref{alg_network_inference_nr}. Numerical results in Section~\ref{sec_num_training} demonstrate that the neural network approximation and regularization error are related, since a higher regularization yields smaller neural network approximation errors while increasing the regularization error and vice versa. 
Due to the convexity of the neural network-based approximations of the entropy closure and the sampled moments, it is feasible to estimate the maximum interpolation error of the neural network in the convex hull of the training data, i.e.
\begin{align}\label{eq_max_error_sampling}
   	\max_{\ofus\in\mathcal{C}\left(\menge{\overline{\fu}_{\#,i}}_{i=1}^T\right)}\norm{\fg^\gamma_{\ofus} -\fg^p_{\ofus}},
\end{align}
where $\mathcal{C}\left(\menge{\overline{\fu}_{\#,i}}_{i=1}^T\right)$ is the convex hull of the training data moments $\ofus$.
For a given interpolation error tolerance one can employ ansatz~\eqref{eq_ansatzes}, to derive the error in $\fu^p$. For details, we refer to~\cite[Section 4.1]{pmlr_v162_schotthofer22a}. The authors of~\cite{pmlr_v162_schotthofer22a} further provide a sampling strategy to minimize the maximal interpolation error~\eqref{eq_max_error_sampling}, which we adapt to the partially regularized setting of this work. The strategy is to sample the bounded set
\begin{align}
\label{eq:BMtau}
	B_{M,\tau}^\gamma=\menge{{\bbeta}\in\mathbb{R}^{n}:\norm{{\bbeta}}<M\cup\lambda_{\text{min}}^\gamma>\tau}.
\end{align}
uniformly~\cite{pmlr_v162_schotthofer22a, old_entropy}\footnote{Other works~\cite{XIAO2023112278, lee2024structurepreserving} adapt Gaussian sampling for non-regularized closures, which stabilizes neural network training, at the expense of a strongly increased interpolation error.}, where $\lambda_{\text{min}}^\gamma$ is the smallest eigenvalue of the Hessian $\hat{H}_n^\gamma\left({\bbeta}\right)$ and proportional to the condition number of the reduced, partially regularized closure problem~\eqref{eq_reduced_reg_closure}, and $M$ is an additional norm boundary for the Lagrange multiplier. The training data moments and entropy values are then sampled from $B_{M,\tau}^\gamma$, see Algorithm~\ref{alg_sampling_entropy}.

Furthermore, analysis of the regularization error in Theorem~\ref{theo_reg_ansatz_error} shows that regularized moments and non-regularized moments, generated by the same Lagrange multiplier with norm bound $\norm{{\bbeta}}<M$, have the distance
\begin{align}\label{eq_psi}
	\norm{\boldsymbol\psi({\bbeta})-\boldsymbol\psi^\gamma({\bbeta})} \leq \gamma M,
\end{align}
where $\boldsymbol{\psi}$ is given by Eq.~\eqref{eq_u_from_alpha}.

Lastly, the regularized entropy closure functional $\hat{h}^\gamma$, $\gamma>0$, is finite at the boundary of the non-regularized realizable set $\partial\overline{\mathcal{R}}$, which mitigates the problem of diverging entropy values described by Theorem~\ref{theo_diverging_entropy}.
 Since $\hat{h}^\gamma$ is convex in $\ofus$, the data sampling strategy of~\cite{pmlr_v162_schotthofer22a} can be extended beyond $\partial\overline{\mathcal{R}}$ of the non-regularized closure. 
 
Figure~\ref{fig_realizable_set_gammas} displays the boundary of the convex hull of reconstructed moments  $\boldsymbol\psi^\gamma({\bbeta})$ of the M$_2$ closure in one dimension, where ${\bbeta}\in\menge{{\bbeta}\in\mathbb{R}^{n}:\norm{{\bbeta}}<M}$ and $\gamma\in\menge{0,1\mathrm{e}{-1},1\mathrm{e}{-2},1\mathrm{e}{-3}}$. 
Figure~\ref{fig_entropy_gammas} gives the corresponding result for $\hat{h}^\gamma$ for the M$_1$ closure in one dimension.
 \begin{figure}[t]
	\begin{algorithm}[H]
    	\label{alg_sampling_entropy}
    	\input{assets/algorithms/alg_sampling_entropy.tex}
	\end{algorithm}
	\vspace{-1em}
\end{figure}
\begin{figure}
\centering
\begin{subfigure}[t]{.485\textwidth}
 	\raggedright
	\includegraphics[width=\textwidth]{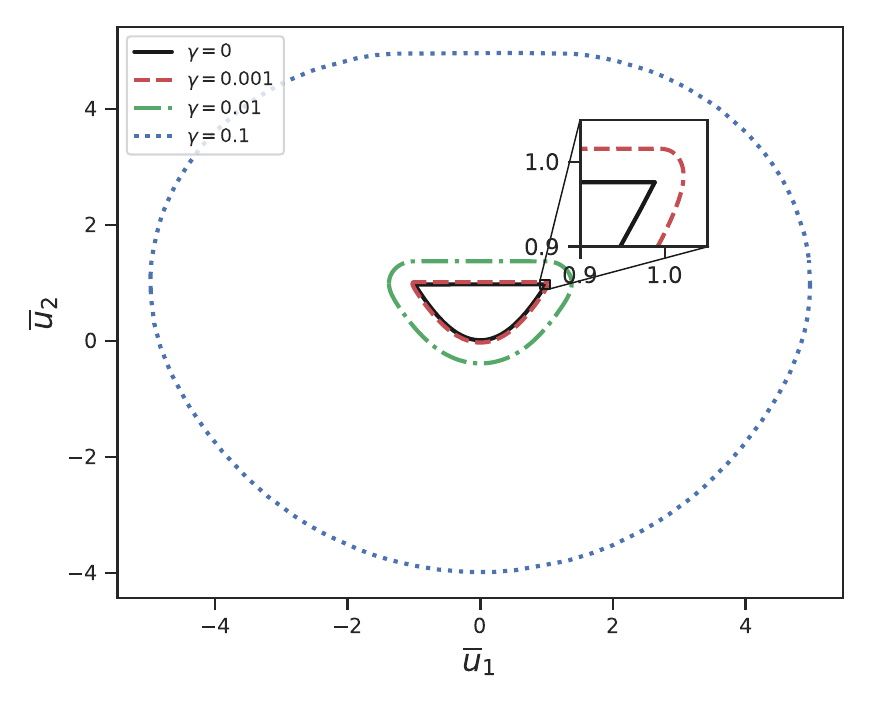}
	\caption{Normalized moments $\ofus=\boldsymbol\psi^{\gamma}({\bbeta})$ sampled for different values of $\gamma$ for the one-dimensional M$_2$ closure. The colored contour denotes the convex hull boundary of the sampled moments for each $\gamma$. The black contour is very close to the boundary of the reduced, normalized realizable set $\widetilde{\mathcal{R}}$, defined in \eqref{eq_reduced_normalizer_realizable_set}, and as $\gamma$ increases, the convex hull of the sampled moments expands.}
	\label{fig_realizable_set_gammas}
\end{subfigure}
\hfill
\begin{subfigure}[t]{.485\textwidth}
 	\raggedleft
	\includegraphics[width=\textwidth]{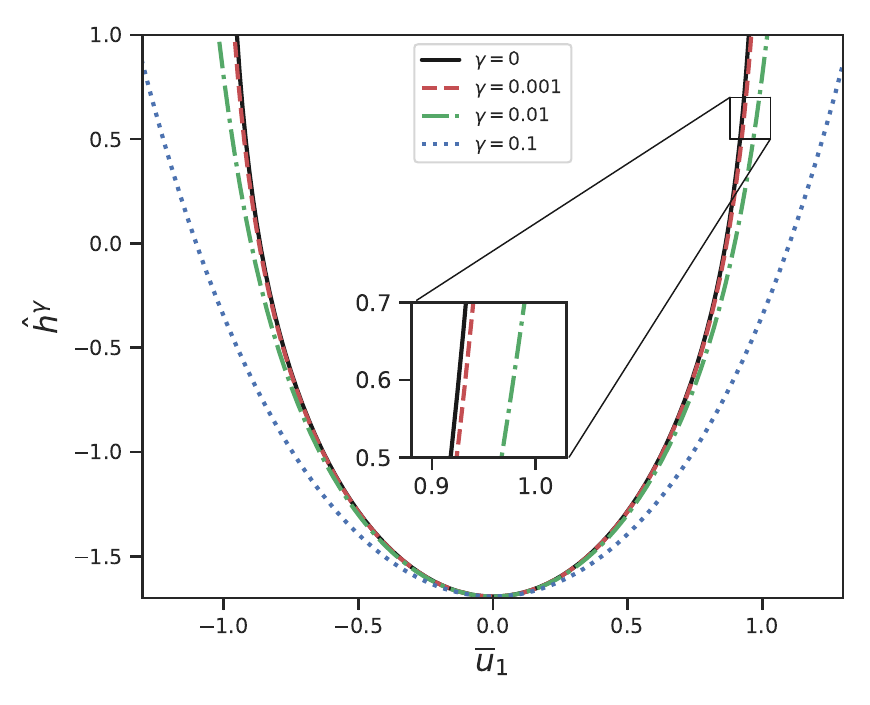}
	\caption{Regularized entropy function values $\hat{h}^{\gamma}(\ofus)$ for the one-dimensional M$_1$ closure at different values of $\gamma$. As $\gamma$ increases, the slope of $\hat{h}^{\gamma}$ becomes less steep.}  
	\label{fig_entropy_gammas}
\end{subfigure}
\caption{The sets of sampled moments $\ofus$ and the regularized entropy functions $\hat{h}^{\gamma}(\ofus)$, see Theorem~\ref{thm_reg_rescale}, for different regularization parameters $\gamma$. The moments are computed from multipliers ${\bbeta}$, i.e., $\ofus=\boldsymbol{\psi}^\gamma({\bbeta})$, where $\boldsymbol{\psi}^\gamma$ is defined in \eqref{eq_u_from_alpha} and ${\bbeta}$ is sampled from the set $\menge{{\bbeta}\in\mathbb{R}^{n}:\norm{\boldsymbol{\beta}}<M }$.}
\end{figure}

Figure~\ref{fig_alpha_u_distributions} illustrates the sampling distributions of $\ofus$ and $\bbeta_{\ofus}^\gamma$ for different $\gamma$ in the one dimensional M$_2$ closure. For $\gamma>0$, we see in Fig.~\ref{fig_alpha_gammas_3}), Fig.~\ref{fig_alpha_gammas_2}) and Fig.~\ref{fig_alpha_gammas_1}) that within the norm boundary $\norm{\bbeta_{\ofus}^\gamma}\leq M=40$, all Lagrange multipliers fulfill the eigenvalue threshold $\tau=0.01$.
 The data generator is part of the KiT-RT framework~\cite{Kitrt_paper}.

\begin{figure}
\begin{subfigure}[t]{.245\textwidth}
 	\centering
	\includegraphics[width=\textwidth]{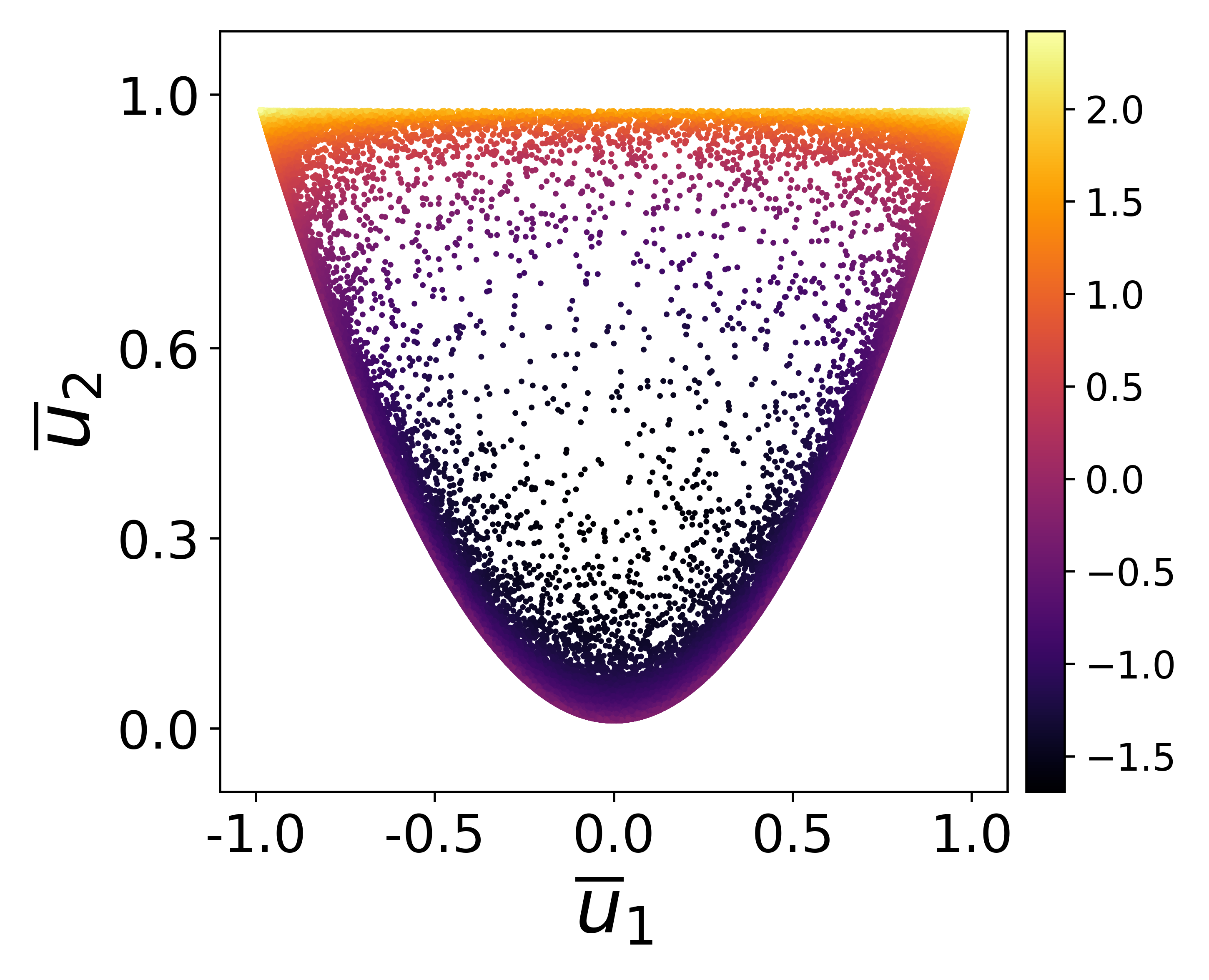}
	\caption{ $\gamma=0$} \label{fig_realizable_set_gammas_0}
\end{subfigure}
\begin{subfigure}[t]{.245\textwidth}
 	\centering
	\includegraphics[width=\textwidth]{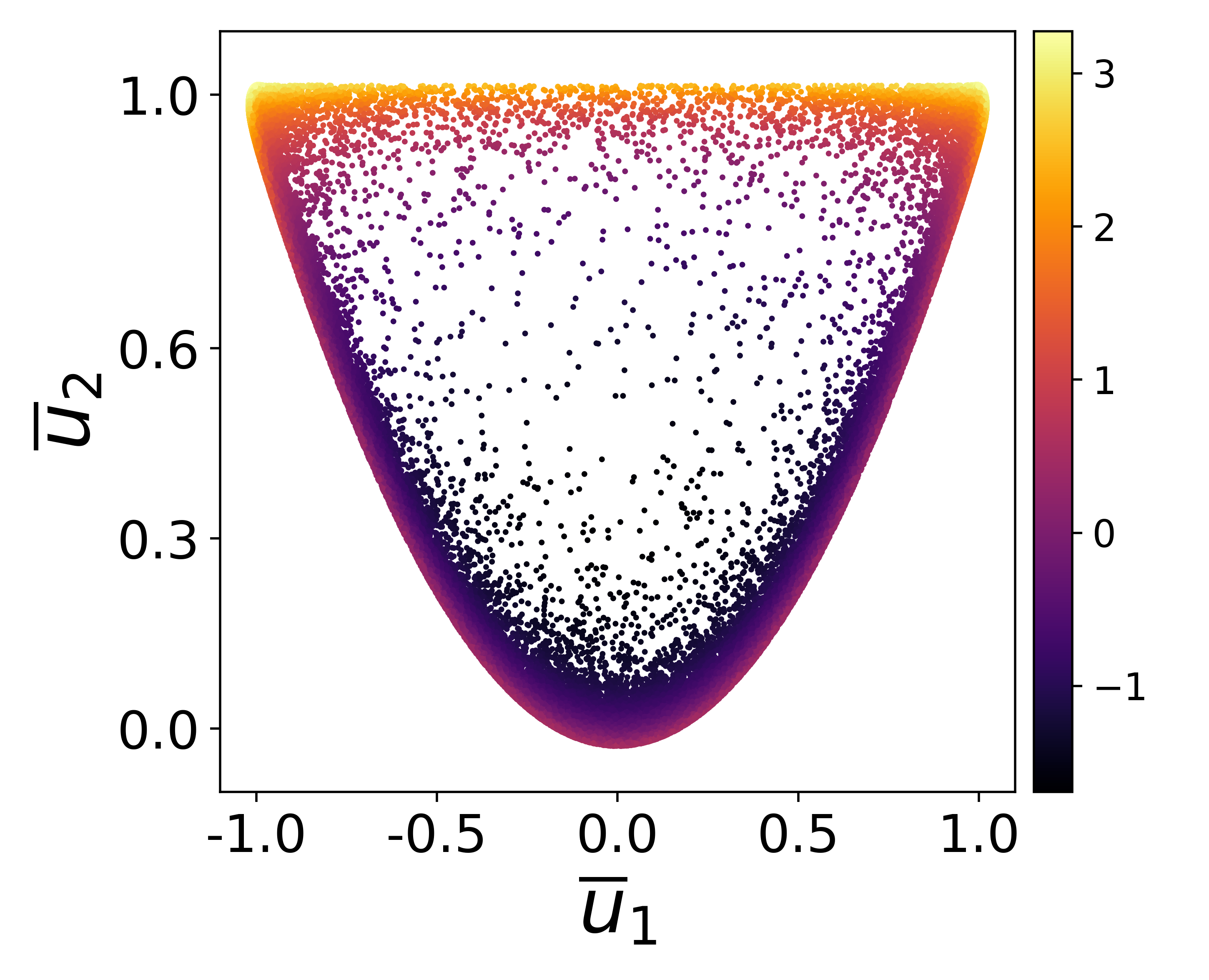}
	\caption{ $\gamma=1\mathrm{e}{-3}$}\label{fig_realizable_set_gammas_3}
\end{subfigure}
\begin{subfigure}[t]{.245\textwidth}
 	\centering
	\includegraphics[width=\textwidth]{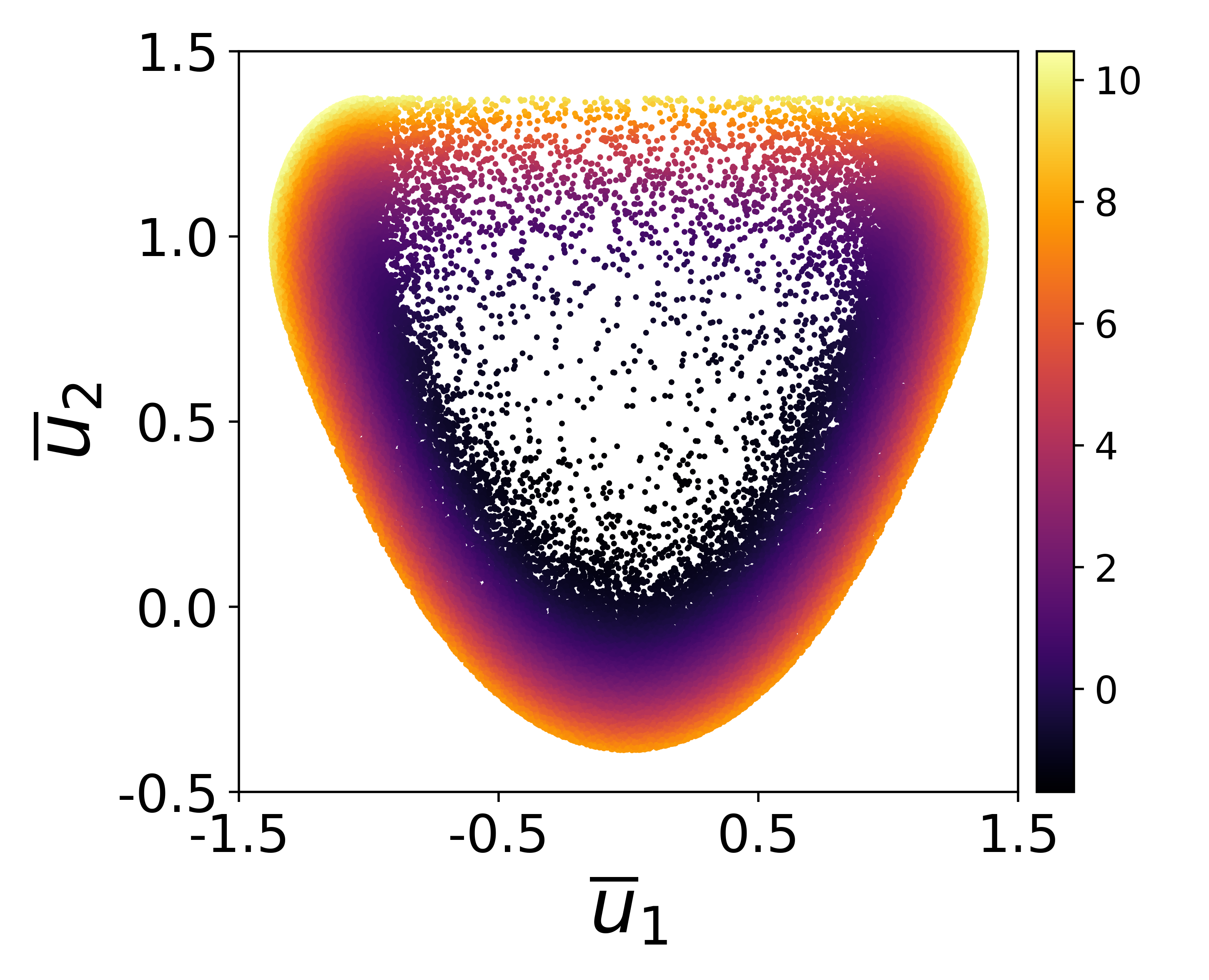}
	\caption{ $\gamma=1\mathrm{e}{-2}$}\label{fig_realizable_set_gammas_2}
\end{subfigure}
\begin{subfigure}[t]{.245\textwidth}
 	\centering
	\includegraphics[width=\textwidth]{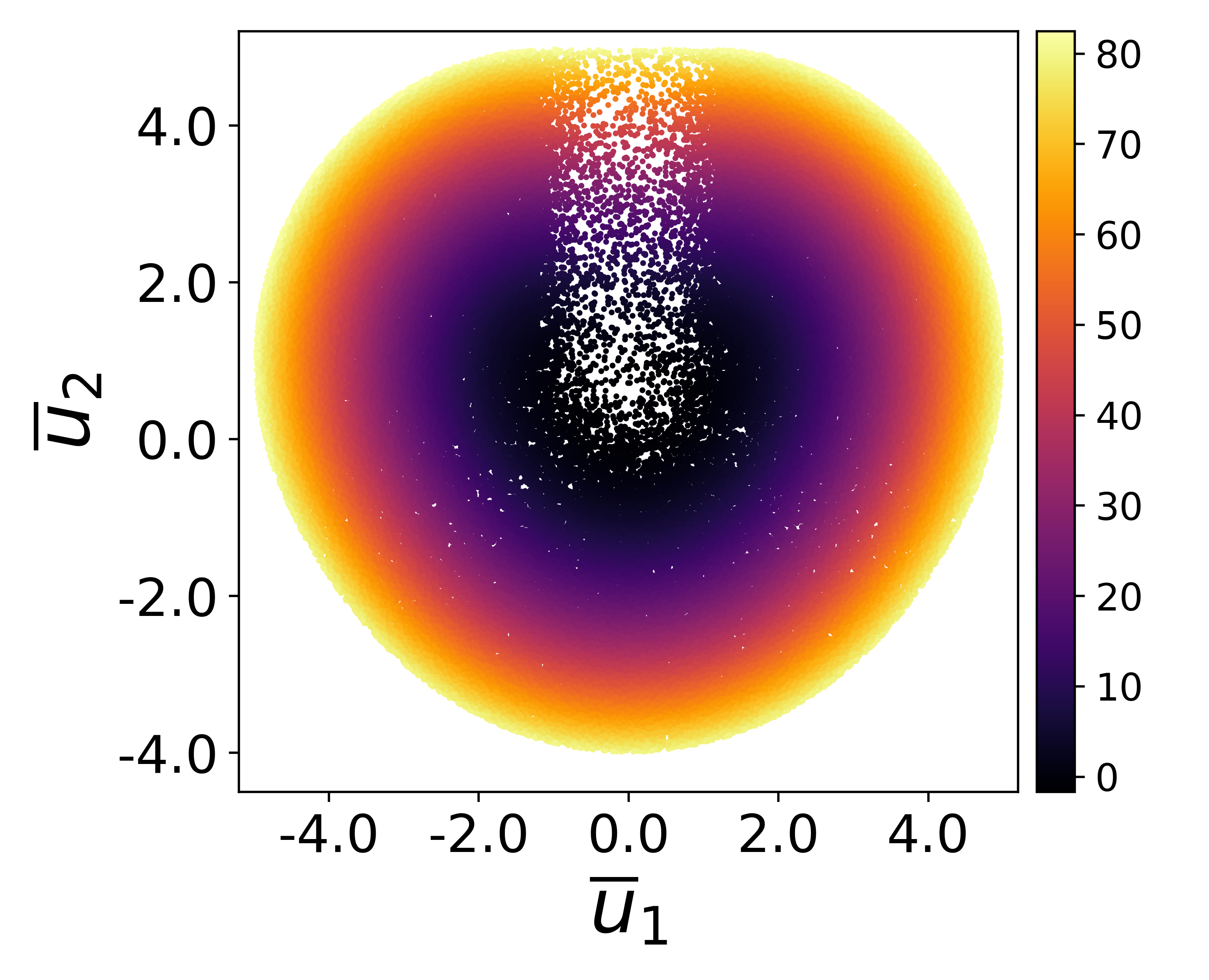}
	\caption{ $\gamma=1\mathrm{e}{-1}$}\label{fig_realizable_set_gammas_1}
\end{subfigure}
\begin{subfigure}[t]{.245\textwidth}
 	\centering
	\includegraphics[width=\textwidth]{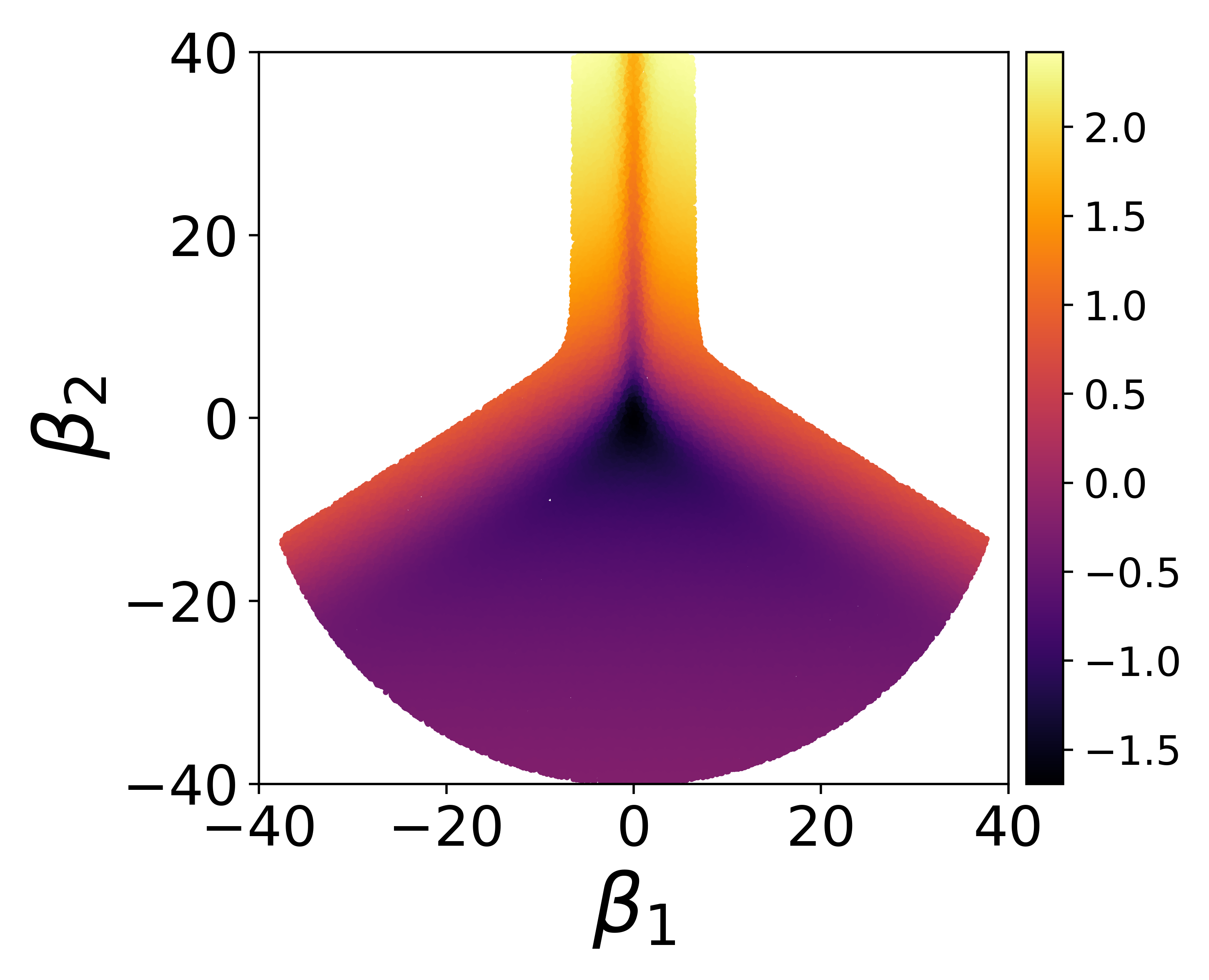}
	\caption{ $\gamma=0$}\label{fig_alpha_gammas_0}
\end{subfigure}
\begin{subfigure}[t]{.245\textwidth}
 	\centering
	\includegraphics[width=\textwidth]{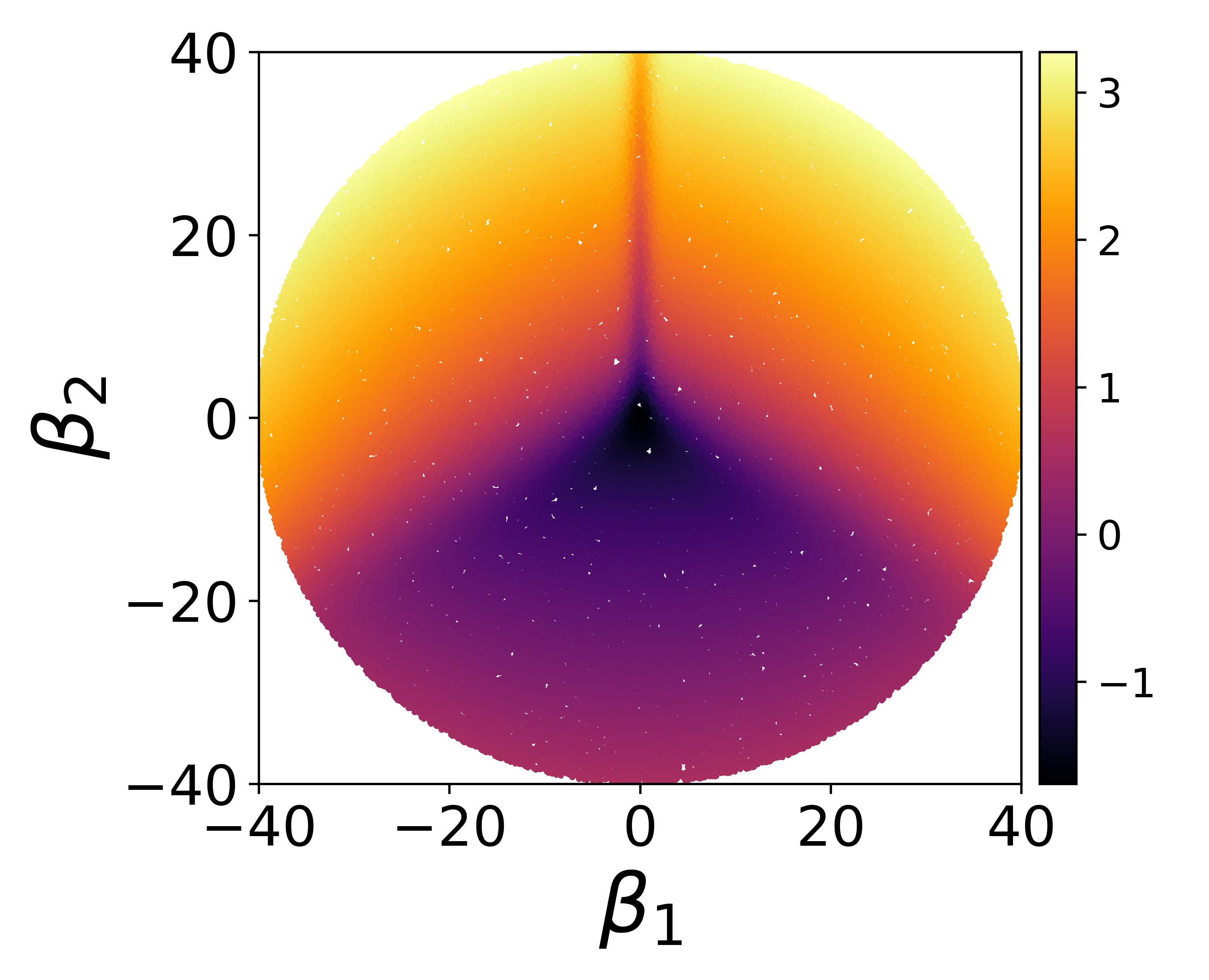}
	\caption{ $\gamma=1\mathrm{e}{-3}$}\label{fig_alpha_gammas_3}
\end{subfigure}
\begin{subfigure}[t]{.245\textwidth}
 	\centering
	\includegraphics[width=\textwidth]{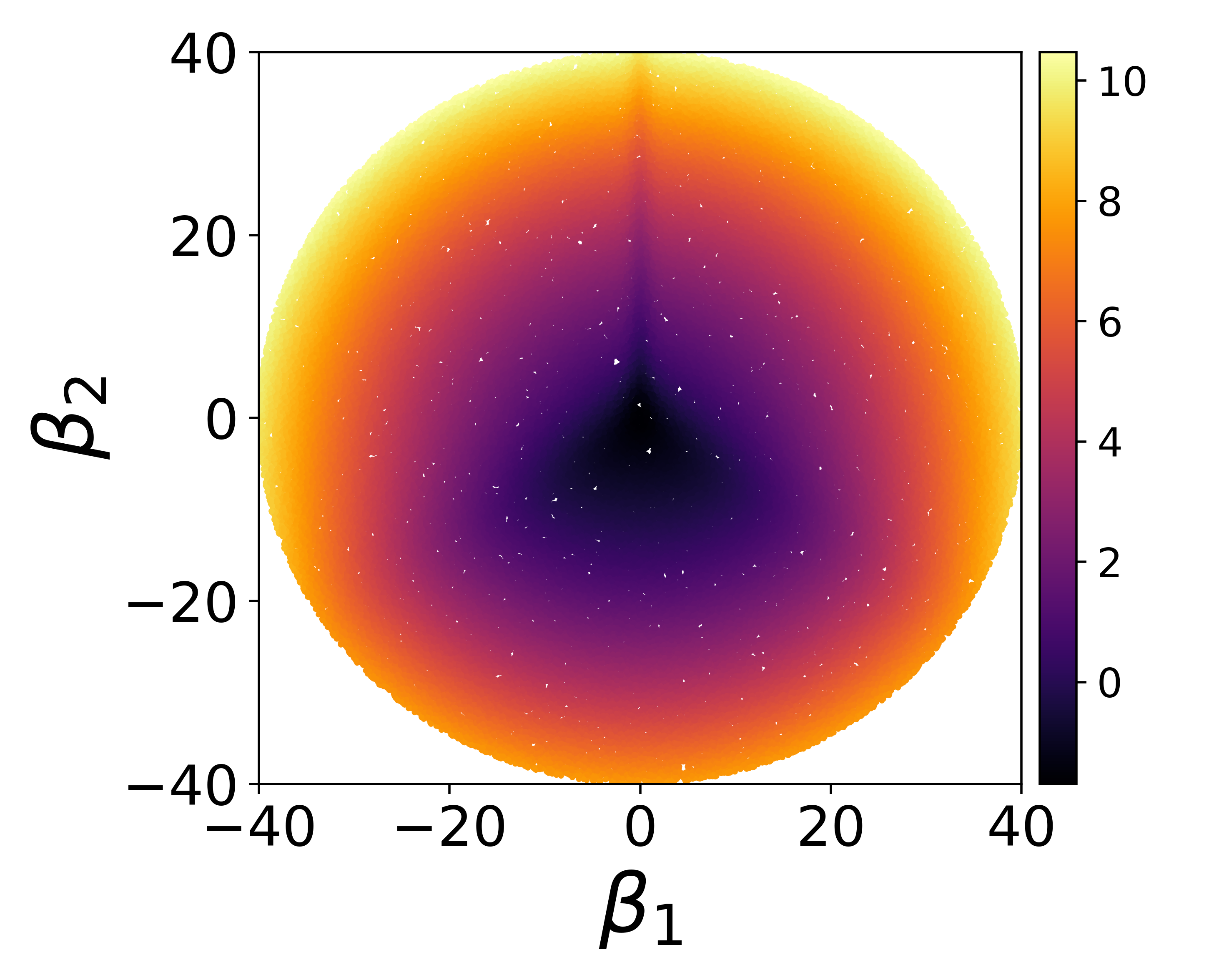}
	\caption{ $\gamma=1\mathrm{e}{-2}$}\label{fig_alpha_gammas_2}
\end{subfigure}
\begin{subfigure}[t]{.245\textwidth}
 	\centering
	\includegraphics[width=\textwidth]{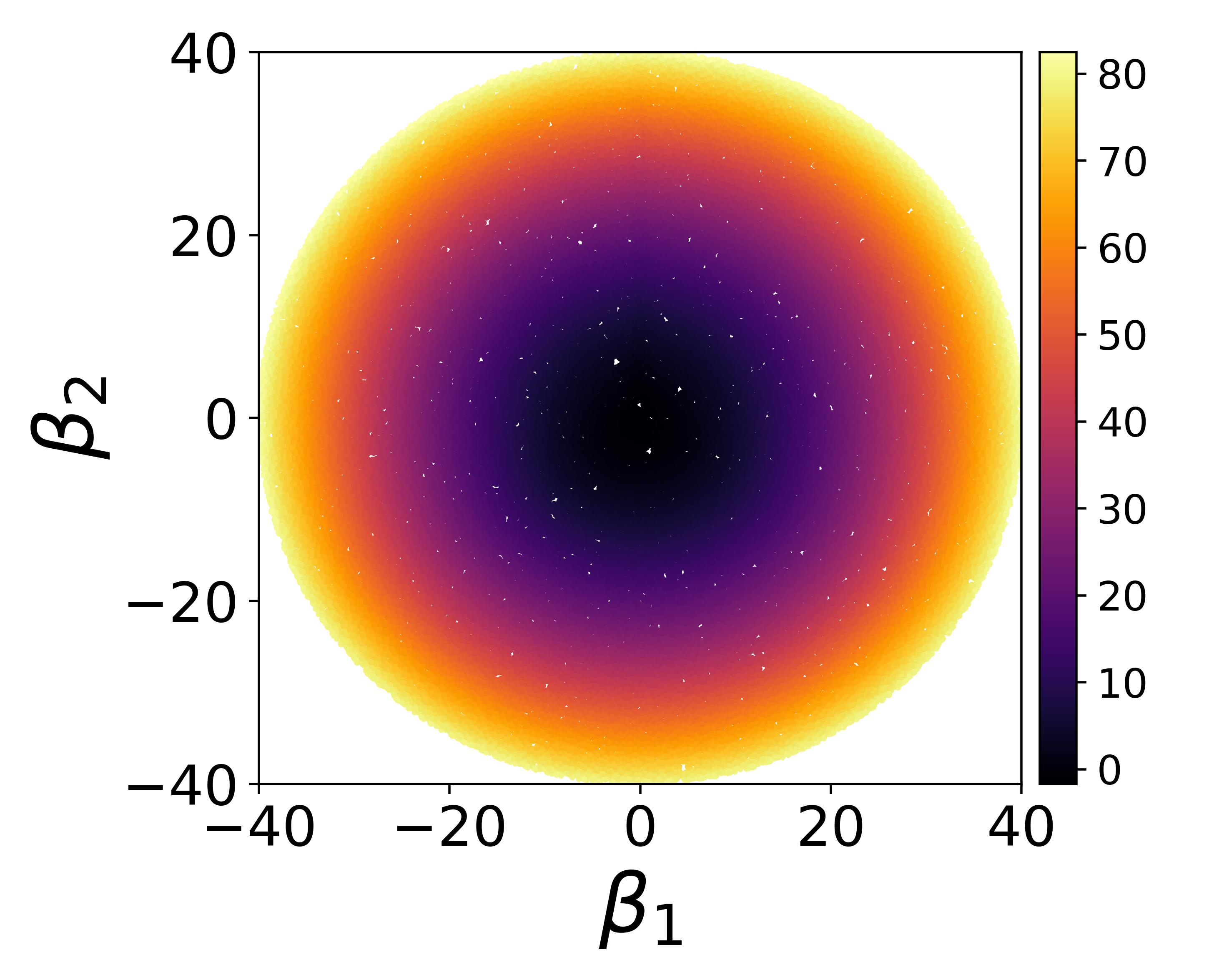}
	\caption{ $\gamma=1\mathrm{e}{-1}$}\label{fig_alpha_gammas_1}
\end{subfigure}
\caption{ Scatter plots for sampled values of $\bbeta$ (bottom row) and the corresponding values of $\ofus=\boldsymbol{\psi}^\gamma({\bbeta})$ (top row), where $\boldsymbol{\psi}^\gamma$ is defined in \eqref{eq_u_from_alpha}. The value of $\hat{h}^\gamma$ is represented in each plot by a heatmap.  Less regularization leads to steeper slopes of $\hat{h}^\gamma$ and thus higher sampling densities in regions where $\norm{\ofu}$ is large. The sampling strategy  is described in Algorithm \ref{alg_sampling_entropy} using the set $B_{M=40,\tau=0.01}^\gamma$, defined in \eqref{eq:BMtau}, for the rejection criteria.} 
\label{fig_alpha_u_distributions}
\end{figure}

\section{Numerical solver}
The entropy-based closure is combined with a kinetic scheme~\cite{GarretHauck, AlldredgeHauckTits, Kitrt_paper}. Thus, the moment system of Eq.~\eqref{eq_momentRTa} is discretized using a finite volume scheme in space,  
\begin{equation}
   \partial_t \mathbf u_i=-\frac{1}{ A_i} \sum_{j\in N(i)} \mathbf{F}_{\rm{up}} (\mathbf u_j,\mathbf u_i)+  \mathbf{G}(t,\mathbf x_i,\mathbf u_i),
\end{equation}
where $\mathbf{G}(t,\mathbf x_i,\mathbf u_i)$ is the discretization of collision and source terms, and $\mathbf{F}_{\rm{up}} (\fu_j,\fu_i)$ is the kinetic flux discretized with an upwind scheme.  The kinetic flux is computed for all faces of neighboring cells $N(i)$ of the current cell $i$ of an unstructured mesh. The cell area of the current cell is denoted by $A_i$.
We arrive at the kinetic numerical flux
\begin{align}
	\begin{aligned}
 	\mathbf{F}_{\rm{up}}(\mathbf u_j,\mathbf u_i) =\inner{{\fm} \mathbf{v}\cdot \mathbf n_{i,j} \left[{f_{\fu_i}} \mathcal{H}\left( \mathbf{v}\cdot \mathbf n_{i,j}\right) + f_{\mathbf u_j}  \left( 1- \mathcal{H}\left( \mathbf{v}\cdot \mathbf n_{i,j}\right) \right) \right]},
\end{aligned}
\end{align}
where $\mathbf{n}_{i,j}$ is the outward facing normal of the face between cells $i$ and $j$ scaled with the cell length and $\mathcal{H}$ denotes the Heaviside step function.
The int egral is evaluated with a tensorized Gauss-Legendre quadrature.
We employ a second-order slope reconstruction for unstructured grids with the Barth-Jespersen flux limiter in \cite{barthJespersen} to obtain a second-order accurate spatial discretization. The reconstruction of the kinetic density $f_{\fu}$ uses a generic ansatz.

The second-order Heun's scheme is used for temporal discretization at time step $k$ in grid cell $i$: 
\begin{align}
\begin{aligned}
   \mathbf u_i^{*}&=
	\mathbf u_i^{k}-\frac{\Delta t}{ A_i} \sum_{j\in N(i)} \mathbf{F}_{\rm{up}} (\mathbf u_j^k,\mathbf u_i^k)+ \frac{\Delta t}{ A_i} \mathbf{G}(\mathbf u_i^k),\\
	\mathbf u_i^{**}&=
	\mathbf u_i^{*}-\frac{\Delta t}{ A_i} \sum_{j\in N(i)} \mathbf{F}_{\rm{up}}(\mathbf u_j^*,\mathbf u_i^*) + \frac{\Delta t}{ A_i} \mathbf{G}(\mathbf u_i^*), \\
 	\mathbf u_i^{k+1}& = \frac{1}{2}\left(\mathbf u_i^{k} +  \mathbf u_i^{**}\right).
\end{aligned}
\end{align}
The kinetic scheme is implemented in the open-source radiative transport package KiT-RT~\cite{Kitrt_paper} and available on GitHub\footnote{\url{https://github.com/CSMMLab/KiT-RT}}.

\section{Numerical results}\label{sec_numerical_results}
We discuss numerical results for the proposed neural network-based, partially regularized entropy closures. We first discuss the network training and evaluation before we inspect the simulation performance in the  hohlraum~\cite{hohlraumHauckMclarren} test case. 

\subsection{Neural network training}\label{sec_num_training}
In this section, we evaluate the training performance of the neural network approximations for different regularization levels and moment orders.
We consider test cases in two spatial dimensions so we discard basis elements of $\fm(\fv)$ that correspond to the third spatial dimension. For each closure order $N$, we sample the $B_{M,\tau}^\gamma$, defined in Eq.~\eqref{eq:BMtau}, see Table~\ref{tab_sampling}. $M$ and is dependent on $N$ and  chosen such that the exponential in Eq.~\eqref{eq_ansatzes} does not cause a numerical overflow in single precision accuracy. Then $B_{M,\tau}^\gamma$ is sampled for  $\gamma=10^{-1},10^{-2},10^{-3}$  using the same $M$ and $\tau$.  
We partition the sampled data in $90\%$ training and $10\%$ test data.

\begin{table}[t]
\centering
\caption{Overview of the sampling parameters for the training data generation of $T$ samples from the set $B_{M,\tau}^\gamma$ defined in Eq.~\eqref{eq:BMtau}. The same parameters are used for all values of $\gamma$ used in the numerical results section. } \label{tab_sampling}
\begin{tabular}{@{\extracolsep{1pt}}lcccccc}
\toprule   
 closure &  $M$ & $\tau$ & $T$ \\
     	\midrule
M$_1$ &$40$&$1\rm{e}{-4}$ & $1\rm{e}{6}$\\
M$_2$ &$20$&$1\rm{e}{-4}$ & $1\rm{e}{6}$\\
M$_3$ &$12$&$1\rm{e}{-4}$ & $1\rm{e}{6}$\\
M$_4$ &$8$&$1\rm{e}{-4}$  & $1\rm{e}{6}$\\
\bottomrule
\end{tabular}
\end{table}

\begin{table}[t]
\centering
\caption{Architecture and the number of trainable parameters of the neural network models for each closure. A preliminary search is done to determine the a good choice for the network architectures. We report the layer output dimension and the number of layers for each model. ICNNs typically require an order of magnitude fewer parameters than ResNets to achieve similar test accuracy.} \label{tab_nn_architecture}
\begin{tabular}{@{\extracolsep{1pt}}lcccccc}
\toprule   
 &\multicolumn{3}{c}{ICNN}  &\multicolumn{3}{c}{ResNet}  \\
 \cmidrule{2-4}
 \cmidrule{5-7}
 closure &  weight matrix dimension & layers & params & weight matrix dimension & layers & params \\
     	\midrule
M$_1$ &$100\times 100$&$2$&$2.1\mathrm{e}{5}$&$200\times 200$&$4$&$3.6\mathrm{e}{6}$\\
M$_2$ &$100\times 100$&$3$&$3.2\mathrm{e}{5}$&$300\times 300$&$6$&$5.4\mathrm{e}{6}$\\
M$_3$ &$300\times 300$&$3$&$2.8\mathrm{e}{6}$&$400\times 400$&$6$&$9.6\mathrm{e}{6}$\\
M$_4$ &$400\times 400$&$3$&$5.4\mathrm{e}{6}$&$600\times 600$&$6$&$2.1\mathrm{e}{7}$\\
\bottomrule
\end{tabular}
\end{table}

\begin{figure}
	\centering
	\begin{subfigure}[t]{0.3\textwidth}
    	\includegraphics[height=4cm]{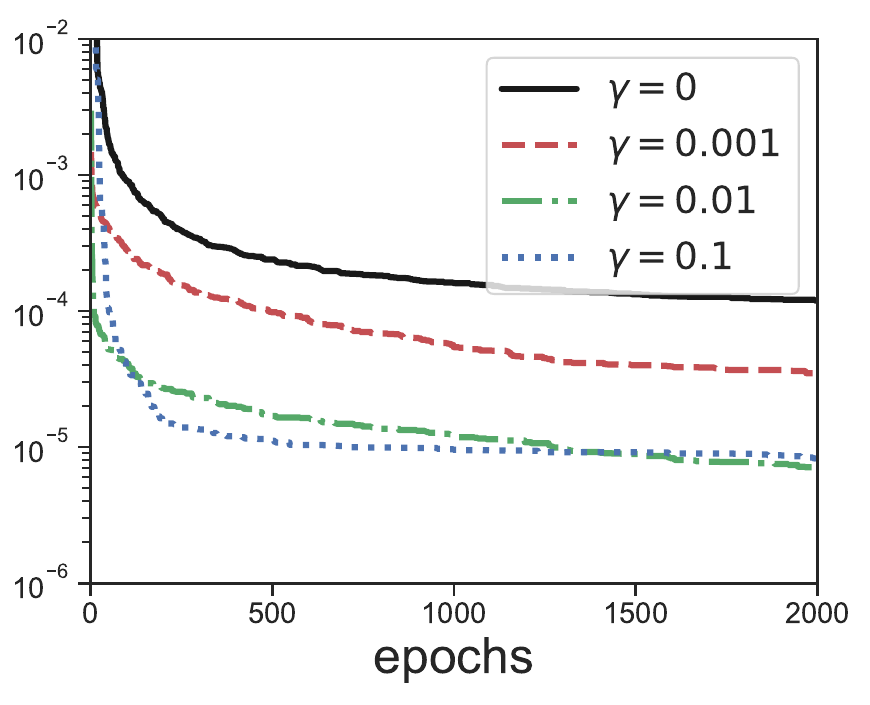}
    	\caption{ M$_3$, ICNN,  $e_{\hat{h}^\gamma}$}
	\end{subfigure}   
	\begin{subfigure}[t]{0.3\textwidth}
    	\includegraphics[width=\textwidth]{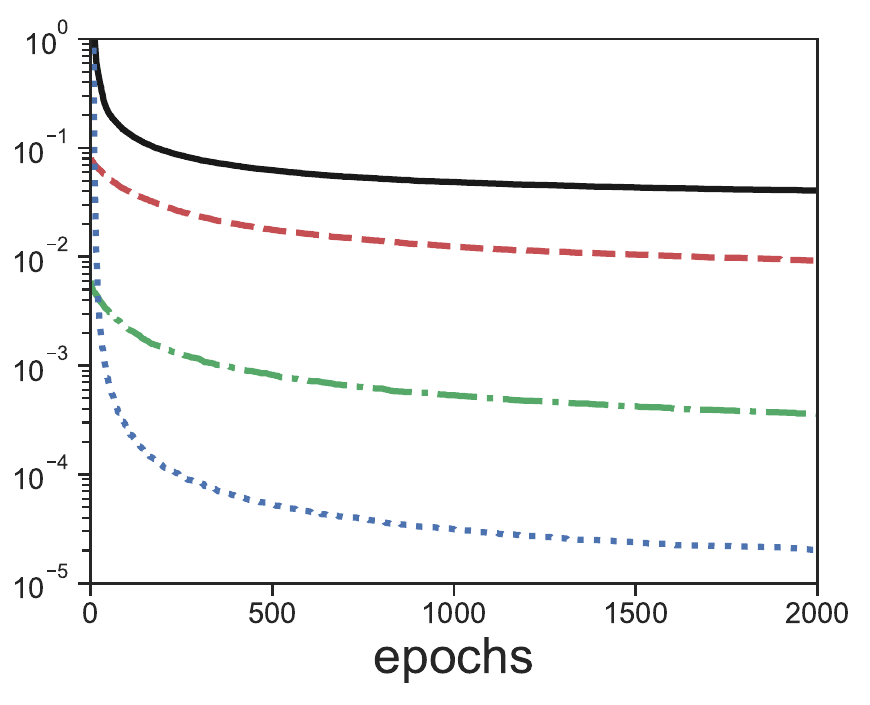}
    	\caption{ M$_3$, ICNN, $e_{\bbeta^\gamma_{\ofus}}$}
	\end{subfigure}  
	\begin{subfigure}[t]{0.3\textwidth}
    	\includegraphics[width=\textwidth]{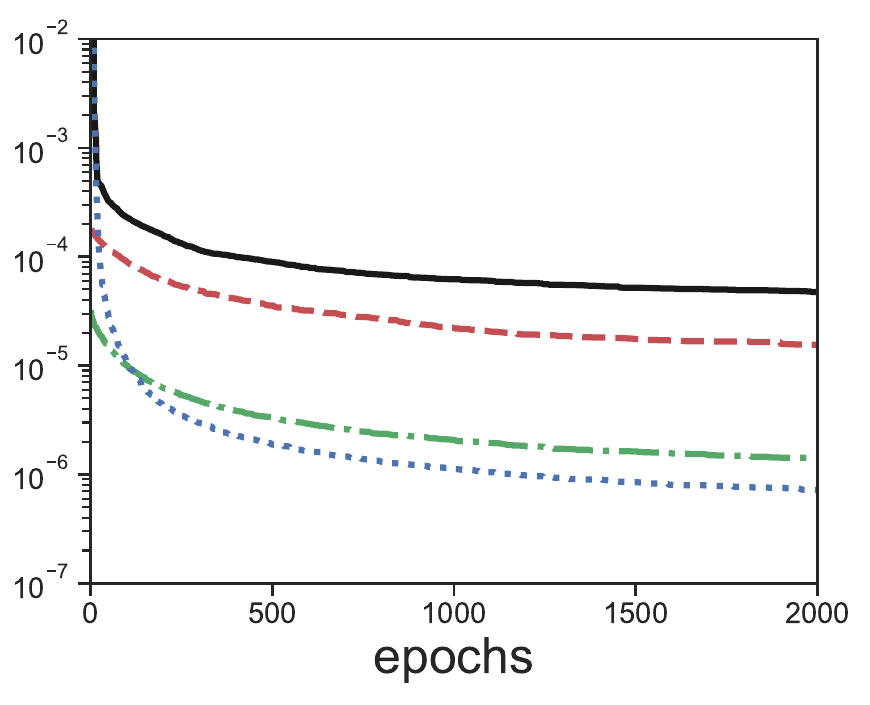}
    	\caption{ M$_3$, ICNN, ${e}_{\ofu}$}
	\end{subfigure}
	\begin{subfigure}[t]{0.3\textwidth}
    	\includegraphics[width=\textwidth]{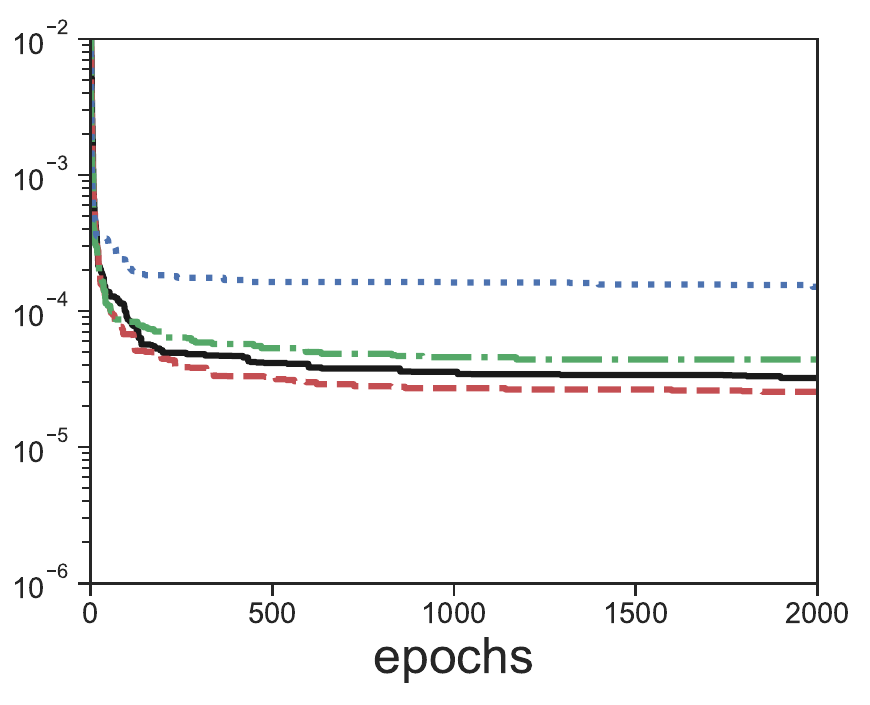}
    	\caption{M$_3$, ResNet,  $e_{\hat{h}^\gamma}$}
	\end{subfigure}   
	\begin{subfigure}[t]{0.3\textwidth}
    	\includegraphics[width=\textwidth]{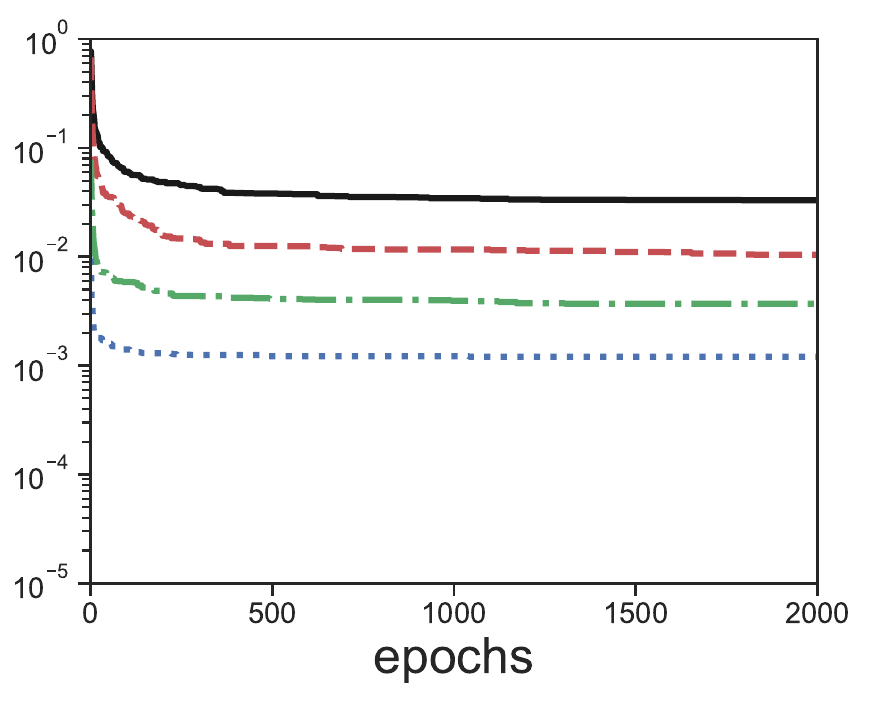}
    	\caption{ M$_3$, ResNet, $e_{\bbeta^\gamma_{\ofus}}$ }
	\end{subfigure}  
	\begin{subfigure}[t]{0.3\textwidth}
    	\includegraphics[width=\textwidth]{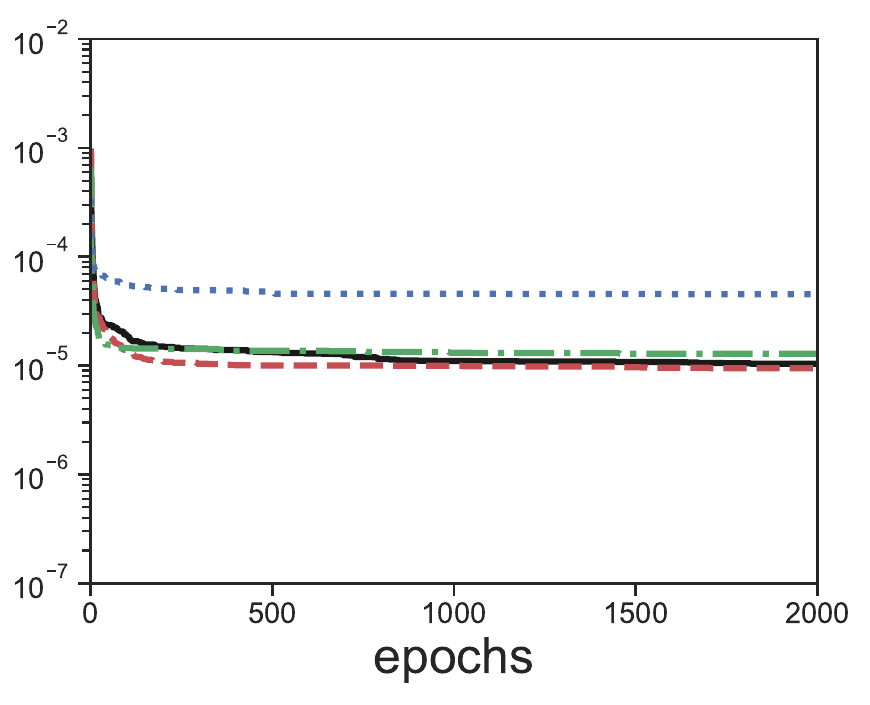}
    	\caption{  M$_3$, ResNet, ${e}_{\ofu}$}
	\end{subfigure}  
	\centering
	\begin{subfigure}[t]{0.3\textwidth}
    	\includegraphics[width=\textwidth]{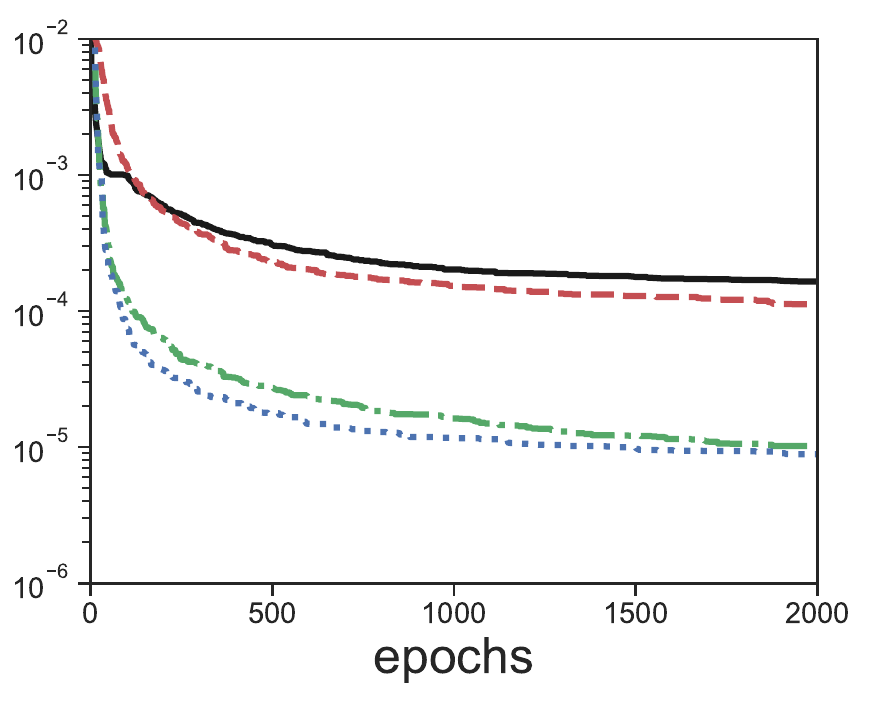}
    	\caption{M$_4$, ICNN,  $e_{\hat{h}^\gamma}$}
	\end{subfigure}   
	\begin{subfigure}[t]{0.3\textwidth}
    	\includegraphics[width=\textwidth]{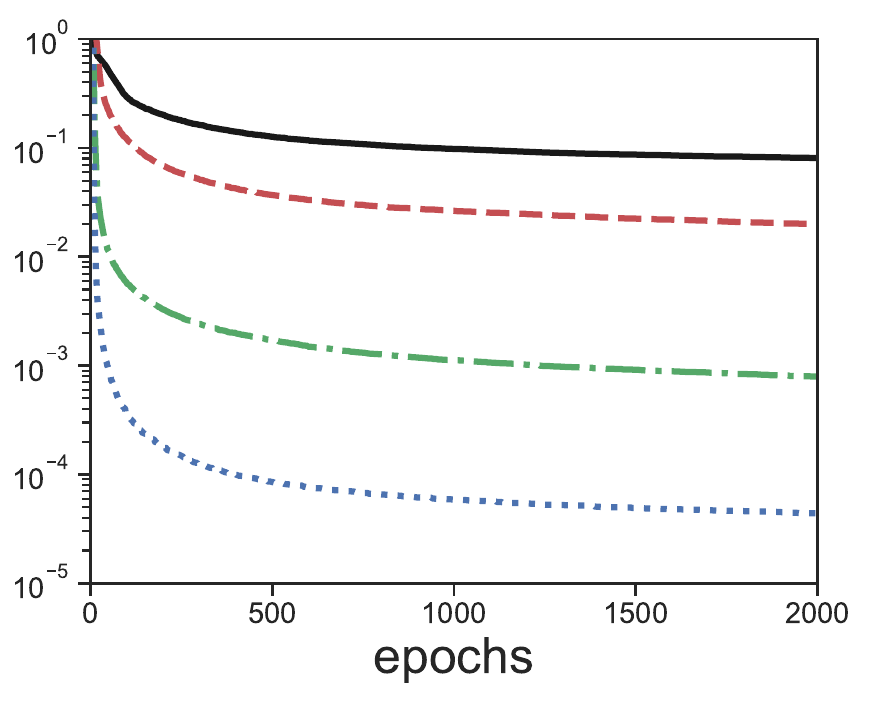}
    	\caption{ M$_4$, ICNN, $e_{\bbeta^\gamma_{\ofus}}$ }
	\end{subfigure}  
	\begin{subfigure}[t]{0.3\textwidth}
    	\includegraphics[width=\textwidth]{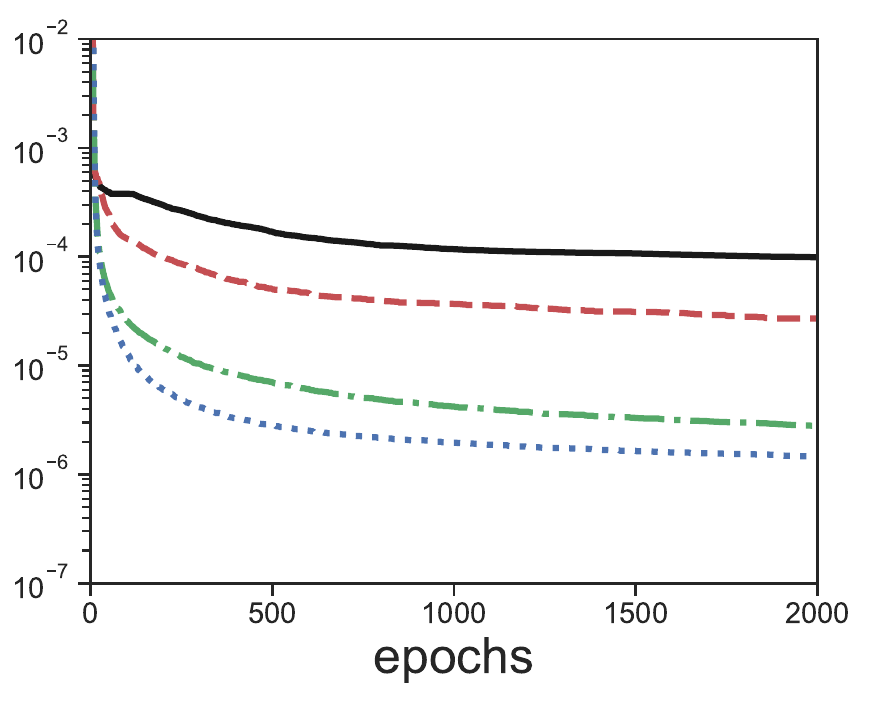}
    	\caption{ M$_4$, ICNN, ${e}_{\ofu}$}
	\end{subfigure}
	\begin{subfigure}[t]{0.3\textwidth}
    	\includegraphics[width=\textwidth]{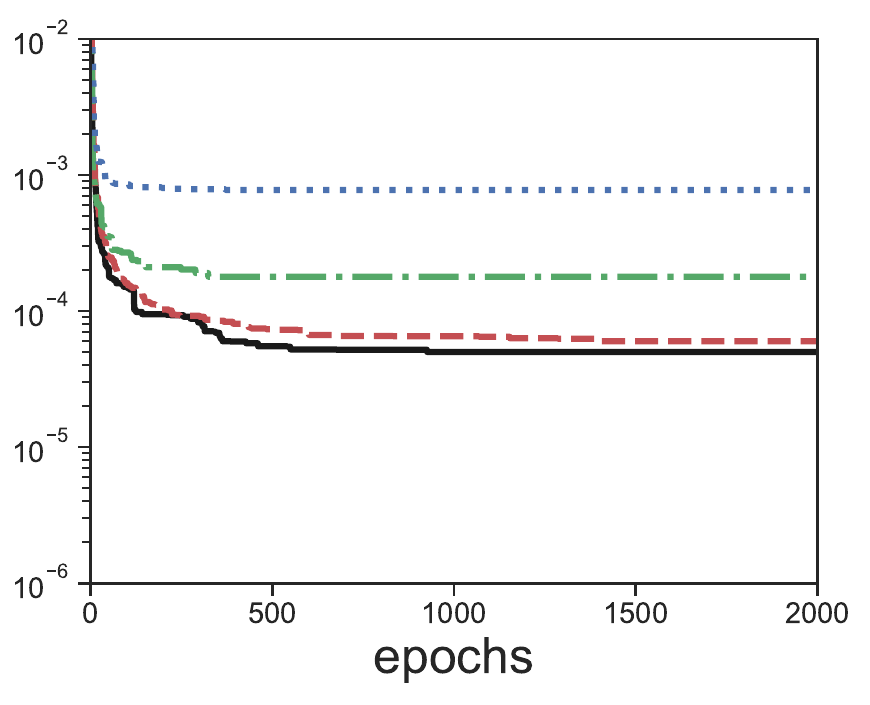}
    	\caption{M$_4$, ResNet,  $e_{\hat{h}^\gamma}$}
	\end{subfigure}   
	\begin{subfigure}[t]{0.3\textwidth}
    	\includegraphics[width=\textwidth]{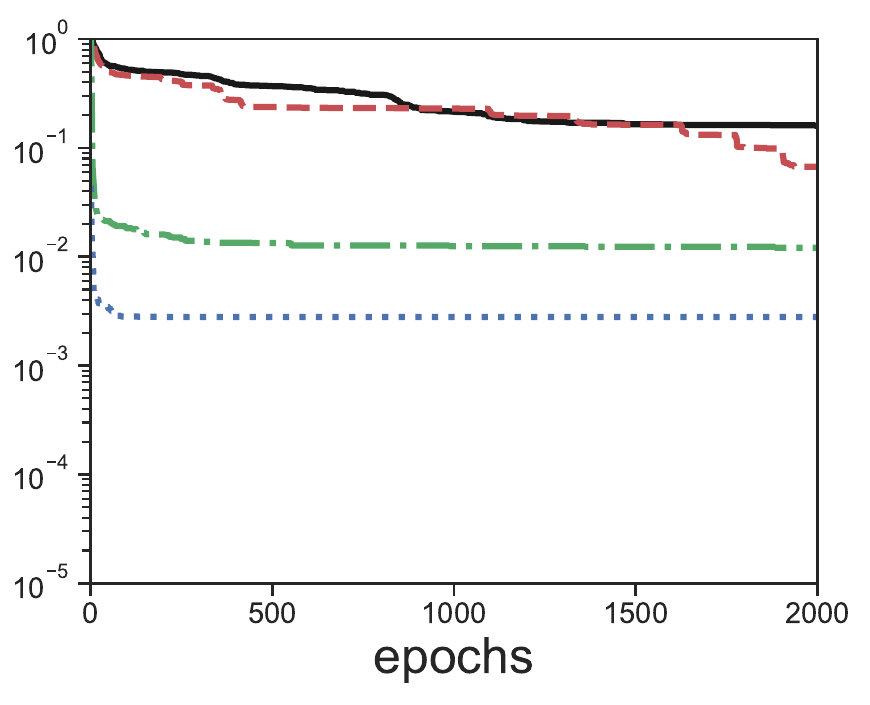}
    	\caption{ M$_4$, ResNet, $e_{\bbeta^\gamma_{\ofus}}$}
	\end{subfigure}  
	\begin{subfigure}[t]{0.3\textwidth}
    	\includegraphics[width=\textwidth]{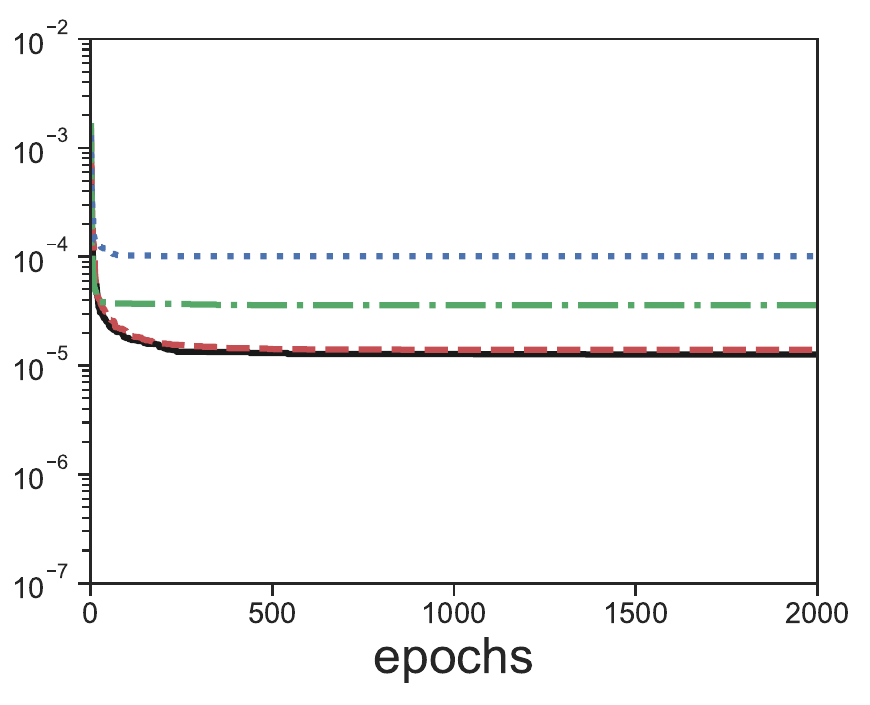}
    	\caption{ M$_4$, ResNet, ${e}_{\ofu}$}
	\end{subfigure}  
	\caption{Comparison of ICNN and ResNet-based test errors for M$_3$ (a)-(f) and M$_4$ (g)-(l) closures, with different regularization levels $\gamma$.%
 The reported error $e_{\hat{h}^\gamma}$, $e_{\bbeta^\gamma_{\ofus}}$ and ${e}_{\ofus}$ are defined inEq.~\eqref{eq_test_error_h},~\eqref{eq_test_error_beta}, and~\eqref{eq_test_error_u}, respectively.%
 The lowest test error up to the current epoch is plotted for each choice of $\gamma$.  The test errors of the ICNN model are reduced heavily by increasing $\gamma$, whereas ResNet test errors reduce only slightly. Results for the M$_2$ closure are comparable.}
	\label{fig_regularized_training_M4}
\end{figure}
The test accuracy of the input convex neural network (ICNN) is compared with a non-convex ResNet architecture that serves as a baseline. The ICNN architectures given by Algorithm~\ref{alg_network_inference_nr} consist of convex layers~\cite{pmlr_v162_schotthofer22a,Amos2017InputCN} build from two weight matrices each, i.e.,
\begin{align}
\begin{aligned}
	\hat{h}^p(\ofus) &= \mathbf{z}_M,\\
	\mathbf{z}_{k} &= \sigma_k(W_k^{\mathbf{z}} \mathbf{z}_{k-1} + W_k^{\fu} \ofus + \mathbf{b}_k), \qquad k=2,\dots, M, \\
	\mathbf{z}_{1} &= \sigma_1(W_1^\fu \ofus + \mathbf{b}_1),
\end{aligned}
\end{align}
where $W_k^{\mathbf{z}}$ has positive entries, $\sigma_k$ is convex and non-decreasing, and $W_k^{\fu}$ may attain negative values.
The ResNet architectures consist of dense layers with skip connections, i.e.,
\begin{align}\label{eq_resnet}
\begin{aligned}
	\hat{h}^p(\ofus) &= \mathbf{z}_M,\\
	\mathbf{z}_{k} &= \sigma_k(W_k \mathbf{z}_{k-1} + \mathbf{b}_k) + \mathbf{z}_{k-1} , \qquad k=1,\dots, M,\\
	\mathbf{z}_0 &= \ofus.
\end{aligned}
\end{align}
We give an overview of the neural network architectures used throughout the numerical results sections in Table~\ref{tab_nn_architecture}. 
For a given closure, the number of layers and neurons are determined in an architecture search at $\gamma=0$ and then fixed for all regularization levels for comparability, and we see in Table~\ref{tab_nn_architecture}, that the ICNN models require fewer parameters for similar or better training performance.
\begin{table}[t!]
\centering
\caption{Mean test errors  ${e}_{\hat{h}^{\gamma}},\,e_{\bbeta^\gamma_{\ofus}},\, e_{\ofus}$ of neural network-based entropy closures.  (See Eq.~\eqref{eq_test_error_h},~\eqref{eq_test_error_beta}, and~\eqref{eq_test_error_u}, respectively, for definitions.) For each architecture, $10$ repetitions are performed and in each case, the standard deviation is less than $10\%$. The smallest errors for each architecture and closure order are marked by bold.} \label{tab_closures}
\begin{tabular}{@{\extracolsep{1pt}}llcccccccc}
\toprule   
 &&\multicolumn{3}{c}{ICNN}  &\multicolumn{3}{c}{ResNet}  \\
 \cmidrule{3-5}
 \cmidrule{6-8}
 closure & $\gamma$ &  $e_{\hat{h}^\gamma}$  &  $e_{\bbeta^\gamma_{\ofus}}$ & $e_{\ofus}$ &	$e_{\hat{h}^\gamma}$  &  $e_{\bbeta^\gamma_{\ofus}}$ & $e_{\ofus}$ \\
     	\midrule
     	M$_2$& $0$  	&  $1.45\mathrm{e}{-5}$ &  $5.23\mathrm{e}{-3}$  & $1.17\mathrm{e}{-5}$  &   
                        	$1.77\mathrm{e}{-5}$ & $4.28\mathrm{e}{-3}$ & $4.37\mathrm{e}{-6}$ \\
     	M$_2$& $1\mathrm{e}{-3}$   &   $1.02\mathrm{e}{-5}$ &  $2.69\mathrm{e}{-3}$  & $8.87\mathrm{e}{-6}$  &  
                                    	$1.44\mathrm{e}{-5}$ &  $2.57\mathrm{e}{-3}$  & $4.05\mathrm{e}{-6}$\\
     	M$_2$& $1\mathrm{e}{-2}$   &   ${1.34\mathrm{e}{-6}}$ &  $9.32\mathrm{e}{-5}$  & $\mathbf{7.81\mathrm{e}{-7}}$ &  
                                    	$\mathbf{1.08\mathrm{e}{-6}}$ &  $7.78\mathrm{e}{-4}$  & $\mathbf{3.92\mathrm{e}{-6}}$ \\
     	M$_2$& $1\mathrm{e}{-1}$	&  $\mathbf{1.24\mathrm{e}{-6}}$ &   $\mathbf{5.12\mathrm{e}{-5}}$  &  $1.71\mathrm{e}{-6}$&
                                    	$3.22\mathrm{e}{-5}$ &  $\mathbf{3.57\mathrm{e}{-4}}$  & $1.65\mathrm{e}{-5}$ \\
     	\midrule
     	M$_3$& $0$  	&  	$1.19\mathrm{e}{-4}$ &   $4.05\mathrm{e}{-2}$  &  $4.73\mathrm{e}{-5}$ &  
                            	$3.92\mathrm{e}{-5}$ &  $3.23\mathrm{e}{-2}$  & $1.04\mathrm{e}{-5}$\\
     	M$_3$& $1\mathrm{e}{-3}$&  $3.52\mathrm{e}{-5}$&   $9.21\mathrm{e}{-3}$  &  $1.55\mathrm{e}{-5}$ &
                                	$\mathbf{2.55\mathrm{e}{-5}}$ &   $1.04\mathrm{e}{-2}$  &  $\mathbf{9.44\mathrm{e}{-6}}$\\
     	M$_3$& $1\mathrm{e}{-2}$&  $\mathbf{7.04\mathrm{e}{-6}}$ & $3.61\mathrm{e}{-4}$  & $1.40\mathrm{e}{-6}$ &  
                                	$4.40\mathrm{e}{-5}$ &   $3.70\mathrm{e}{-3}$  &  $1.28\mathrm{e}{-5}$\\
     	M$_3$& $1\mathrm{e}{-1}$&  ${8.09\mathrm{e}{-6}}$ &  $\mathbf{2.03\mathrm{e}{-5}}$  & $\mathbf{7.14\mathrm{e}{-7}}$ &
                                	$1.50\mathrm{e}{-4}$ &   $\mathbf{1.25\mathrm{e}{-3}}$  &  $4.51\mathrm{e}{-5}$\\
     	\midrule
     	M$_4$& $0$  	&  $1.65\mathrm{e}{-4}$ &  $8.07\mathrm{e}{-2}$  & $9.88\mathrm{e}{-5}$ &  
                        	$\mathbf{5.80\mathrm{e}{-5}}$ &   $1.64\mathrm{e}{-1}$  &  $1.26\mathrm{e}{-5}$ \\
     	M$_4$& $1\mathrm{e}{-3}$  &$1.12\mathrm{e}{-4}$ &  $2.00\mathrm{e}{-2}$  & $2.69\mathrm{e}{-5}$ &  
                                	$6.21\mathrm{e}{-5}$ &   $6.71\mathrm{e}{-2}$  &  $\mathbf{1.39\mathrm{e}{-5}}$\\
     	M$_4$& $1\mathrm{e}{-2}$   &   ${1.01\mathrm{e}{-5}}$ &  $7.94\mathrm{e}{-4}$ & $2.76\mathrm{e}{-6}$ &   
                                    	$1.79\mathrm{e}{-4}$ &   $1.12\mathrm{e}{-2}$  &  $3.56\mathrm{e}{-5}$\\
     	M$_4$& $1\mathrm{e}{-1}$	&  $\mathbf{8.87\mathrm{e}{-6}}$ &  $\mathbf{4.39\mathrm{e}{-5}}$ & $\mathbf{1.46\mathrm{e}{-6}}$&                                             	$7.75\mathrm{e}{-4}$ &   $\mathbf{2.79\mathrm{e}{-3}}$  &  $1.01\mathrm{e}{-4}$\\
\bottomrule
\end{tabular}

\end{table}
\begin{table}[t!]
\centering
 \caption{Mean test errors ${e}_{\hat{h}^{0,p}},\,{e}_{\bbeta^{0,p}_{\ofus}},\, {e}_{\ofus^{0,p}}$ of neural network-based entropy closures  against the \textit{non-regularized reference solution}.   (See Eq.~\eqref{eq_combined_test_error_h},~\eqref{eq_combined_test_error_beta}, and~\eqref{eq_combined_test_error_u}, respectively, for definitions.). For each architecture, $10$ repetitions are performed and in each case, the standard deviation is less than $10\%$. The smallest errors for each architecture and closure order are marked in bold.
 Remark that we have by definition ${e}_{\hat{h}^{0,p}}={e}_{\hat{h}^\gamma}$ for $\gamma=0$, and correspondingly  ${e}_{\bbeta^{0,p}_{\ofus}}={e}_{\bbeta^{\gamma}_{\ofus}}$, and ${e}_{\ofus^{0,p}}={e}_{\ofus^{\gamma}}$.} \label{tab_closures_reg}
\begin{tabular}{@{\extracolsep{1pt}}llcccccccc}
\toprule   
 &&\multicolumn{3}{c}{ICNN}  &\multicolumn{3}{c}{ResNet}  \\
 \cmidrule{3-5}
 \cmidrule{6-8}
 closure & $\gamma$ &  $ {e}_{\hat{h}^{0,p}}$  &  ${e}_{\bbeta^{0,p}_{\ofus}}$ & ${e}_{\ofus^{0,p}}$ &	${e}_{\hat{h}^{0,p}}$  &  ${e}_{\bbeta^{0,p}_{\ofus}}$ & ${e}_{\ofus^{0,p}}$ \\
     	\midrule
M$_2$& $0$  	&   $1.45\mathrm{e}{-5}$ &  $5.23\mathrm{e}{-3}$  & $1.17\mathrm{e}{-5}$  &   $1.77\mathrm{e}{-5}$ & $4.28\mathrm{e}{-3}$ & $4.37\mathrm{e}{-6}$ \\    
     	\cmidrule{3-5} \cmidrule{6-8}
M$_2$& $1\mathrm{e}{-3}$ &$\mathbf{1.78\mathrm{e}{-4}}$ &  $\mathbf{1.68\mathrm{e}{-2}}$  & $\mathbf{1.34\mathrm{e}{-5}}$  
                     	&$\mathbf{9.33\mathrm{e}{-5}}$ &  $\mathbf{1.12\mathrm{e}{-2}}$  & $\mathbf{7.94\mathrm{e}{-5}}$\\
M$_2$& $1\mathrm{e}{-2}$ &${3.39\mathrm{e}{-3}}$ &  ${2.76\mathrm{e}{-1}}$  & ${1.69\mathrm{e}{-4}}$    
                     	&${3.44\mathrm{e}{-3}}$ &  ${2.92\mathrm{e}{-1}}$  & ${1.09\mathrm{e}{-4}}$ \\
M$_2$& $1\mathrm{e}{-1}$ &$7.97\mathrm{e}{-2}$ &   ${1.13\mathrm{e}{-0}}$  &  $5.5\mathrm{e}{-3}$
                     	&$7.7\mathrm{e}{-2}$ &  ${1.12\mathrm{e}{-0}}$  & $8.83\mathrm{e}{-3}$ \\
     	\midrule
M$_3$& $0$  	&	$1.19\mathrm{e}{-4}$ &   $4.05\mathrm{e}{-2}$  &  $4.73\mathrm{e}{-5}$ &  $3.92\mathrm{e}{-5}$ &  $3.23\mathrm{e}{-2}$  & $1.04\mathrm{e}{-5}$\\           	\cmidrule{3-5} \cmidrule{6-8}
M$_3$& $1\mathrm{e}{-3}$ &$\mathbf{2.01\mathrm{e}{-4}}$  &   $\mathbf{1.10\mathrm{e}{-1}}$  &  $\mathbf{2.01\mathrm{e}{-5}}$   
                     	&$\mathbf{6.58\mathrm{e}{-5}}$ &   $\mathbf{1.32\mathrm{e}{-1}}$  &  $\mathbf{7.49\mathrm{e}{-5}}$\\
M$_3$& $1\mathrm{e}{-2}$ &${3.84\mathrm{e}{-3}}$  & ${5.82\mathrm{e}{-1}}$	& ${9.45\mathrm{e}{-5}}$    
                     	&${3.86\mathrm{e}{-3}}$ &   ${5.7\mathrm{e}{-1}}$  &  ${1.12\mathrm{e}{-4}}$\\
M$_3$& $1\mathrm{e}{-1}$&${7.47\mathrm{e}{-2}}$ &  ${1.32\mathrm{e}{-0}}$  & ${3.06\mathrm{e}{-3}}$
                    	&$8.39\mathrm{e}{-2}$ &   ${1.32\mathrm{e}{-0}}$  &  $3.13\mathrm{e}{-3}$\\
     	\midrule
M$_4$& $0$  	&  $1.65\mathrm{e}{-4}$ &  $8.07\mathrm{e}{-2}$  & $9.88\mathrm{e}{-5}$ &  ${5.80\mathrm{e}{-5}}$ &   $1.64\mathrm{e}{-1}$  &  $1.26\mathrm{e}{-5}$  \\  \cmidrule{3-5} \cmidrule{6-8}
M$_4$& $1\mathrm{e}{-3}$ &$\mathbf{2.69\mathrm{e}{-4}}$ &  $\mathbf{4.07\mathrm{e}{-1}}$  & $\mathbf{3.1\mathrm{e}{-5}}$
                     	&$\mathbf{2.12\mathrm{e}{-4}}$ &   $\mathbf{4.10\mathrm{e}{-1}}$  &  $\mathbf{2.48\mathrm{e}{-5}}$\\
M$_4$& $1\mathrm{e}{-2}$ &${4.75\mathrm{e}{-3}}$ &  ${1.29\mathrm{e}{0}}$ & ${6.11\mathrm{e}{-5}}$
                     	&${2.78\mathrm{e}{-3}}$ &   ${1.15\mathrm{e}{0}}$  &  ${1.04\mathrm{e}{-4}}$\\
M$_4$& $1\mathrm{e}{-1}$ &$6.71\mathrm{e}{-2}$ &  ${1.85\mathrm{e}{0}}$&	${1.73\mathrm{e}{-3}}$
                     	&$7.35\mathrm{e}{-2}$ &   ${1.83\mathrm{e}{0}}$&  $3.06\mathrm{e}{-3}$\\
\bottomrule
\end{tabular}

\end{table}
Each model for each closure and regularization level is initialized with normally distributed weights and then trained in single precision on an RTX3090 GPU for a total of $2000$ epochs with a batch size of $256$ on the loss function function of Eq.~\eqref{eq_loss_icnn}. The experiment is repeated $10$ times for each closure, regularization, and model until the results have a standard deviation of less than $10\%$. In what follows we report average results over the 10 trials. Table~\ref{tab_closures} displays the mean test errors of the training runs, i.e.,
\begin{align}
e_{\hat{h}^\gamma} &= \frac{1}{T_{\rm{test}}}\sum_{i=1 }^{T_{\rm{test}}}\norm{\hat{h}^\gamma(\ofu_{\#,i}) -\hat{h}^p(\ofu_{\#,i})}^2, \label{eq_test_error_h}\\
e_{\bbeta^\gamma_{\ofus}}&=\frac{1}{T_{\rm{test}}}\sum_{i=1 }^{T_{\rm{test}}}\norm{ {\bbeta}_{\overline{\fu}_{\#,i}}^\gamma-{\bbeta}_{\overline{\fu}_{\#,i}}^p}^2,\label{eq_test_error_beta} \\
e_{\ofus}&=\frac{1}{T_{\rm{test}}}\sum_{i=1 }^{T_{\rm{test}}}\norm{{\ofus}_{,i} -  \boldsymbol\psi^\gamma({\bbeta}_{\overline{\fu}_{\#,i}}^p)}^2,\label{eq_test_error_u}
\end{align}
where $T_{\rm{test}}$ is the size of the test data-set.
Figure~\ref{fig_regularized_training_M4} shows the test accuracy for ICNN and ResNet approximations of entropy $e_{\hat{h}^\gamma}$, fruncated Lagrange multiplier  $e_{\bbeta^\gamma_{\ofus}}$, and reconstructed moment $e_{\ofus}$ of the M$_3$ and M$_4$ closure for different regularization parameters. The plots display the best test error until the current epoch.
We can see that the test error reduces by several orders of magnitude for larger values of $\gamma$ for $h^\gamma$, $\bbeta^\gamma_{\ofus}$ as well as $\ofus$ for ICNNs for the M$_2$, M$_3$, and M$_4$ closure. We believe that the steeper slope of the entropy function of the non-regularized problem is harder to approximate; see Fig.~\ref{fig_entropy_gammas}.  This trend is not as pronounced for the ResNet approximations, where models exhibit inconsistent approximation performance for $\hat{h}^\gamma$, $\bbeta^\gamma_{\ofus}$ and $\overline{\fu}$. Furthermore,  highly regularized ICNN approximations outperform ResNet approximations by an order of magnitude, especially for M$_3$ and M$_4$ closures.

Finally, we measure the combined neural network approximation and regularization error in Table~\ref{tab_closures_reg}, i.e.,
\begin{align}
{e}_{\hat{h}^{0,p}} &= \frac{1}{T_{\rm{test}}}\sum_{i=1 }^{T_{\rm{test}}}\norm{\hat{h}^{\gamma=0}(\ofu_{\#,i}) -\hat{h}^p(\ofu_{\#,i})}^2,\label{eq_combined_test_error_h} \\
{e}_{\bbeta^{0,p}_{\ofus}}&=\frac{1}{T_{\rm{test}}}\sum_{i=1 }^{T_{\rm{test}}}\norm{ {\bbeta}_{\overline{\fu}_{\#,i}}^{\gamma=0}-{\bbeta}_{\overline{\fu}_{\#,i}}^p}^2,\label{eq_combined_test_error_beta} \\
{e}_{\ofus^{0,p}} &=\frac{1}{T_{\rm{test}}}\sum_{i=1 }^{T_{\rm{test}}}\norm{{\ofus}_{,i} -  \boldsymbol\psi^{\gamma=0}({\bbeta}_{\overline{\fu}_{\#,i}}^p)}^2.\label{eq_combined_test_error_u}
\end{align}
Using Theorem~\ref{theo_reg_ansatz_error} for the fruncated moments, the approximation error  in $\ofus$ can be reformulated as
\begin{align}
  \norm{{\ofus -   \boldsymbol\psi^{\gamma=0}({\bbeta}_{\overline{\fu}_{\#,i}}^p)}}&\leq \norm{\ofus -{\ofu}^\gamma_{\#}} +\norm{{\ofu}^\gamma_{\#} -  \boldsymbol\psi^{\gamma=0}({\bbeta}_{\overline{\fu}_{\#,i}}^p)}
 \leq  \gamma M + n\left(1-\exp\left(-\frac{\gamma}{2}M^2\right)\right)+\norm{\overline{\fu}^\gamma_\# -  \boldsymbol\psi^{\gamma=0}({\bbeta}_{\overline{\fu}_{\#,i}}^p)}.
\end{align}
{Considering Table~\ref{tab_closures_reg} and Table~\ref{tab_closures}, we observe that $\norm{\overline{\fu}^\gamma_\# -  \boldsymbol\psi^{\gamma=0}({\bbeta}_{\overline{\fu}_{\#,i}}^p)}\leq\gamma M$. Therefore we expect the regularization error to have a larger influence on the simulation than the neural network approximation error for $\gamma>1\rm{e}{-3}$}. A trade-off between neural network approximation and regularization errors is required.

\subsection{Line source test case}
\label{sec_linesource}
\begin{table}[t!]
\caption{Setup for the numerical test cases} \label{tab_linesource_setup}

\centering
\begin{tabular}{@{\extracolsep{2pt}}lccccc}
\toprule   
    Test-Case &  CFL & $t_f$&$\Delta_t$ & $\Delta_x$   \\
             	\midrule
     	 Line source &$0.3$ & $0.75$ &$1.68\mathrm{e}{-3}$ & $1\mathrm{e}{-2}$ \\ 
     	Hohlraum &$0.2$ & $2$ &$9.5\mathrm{e}{-4}$  & $5\mathrm{e}{-3}$ \\  
\bottomrule
\end{tabular}
\end{table}

\begin{figure}[t!]
\begin{subfigure}[t]{0.31\textwidth}	\centering
    \includegraphics[height=5cm]{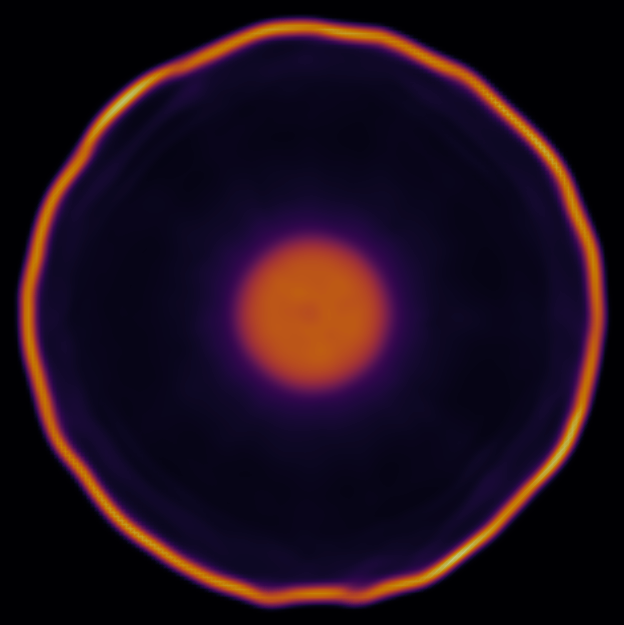}
        \caption{ICNN-based M$_2$}
\end{subfigure}
\begin{subfigure}[t]{0.31\textwidth}	\centering
    \includegraphics[height=5cm]{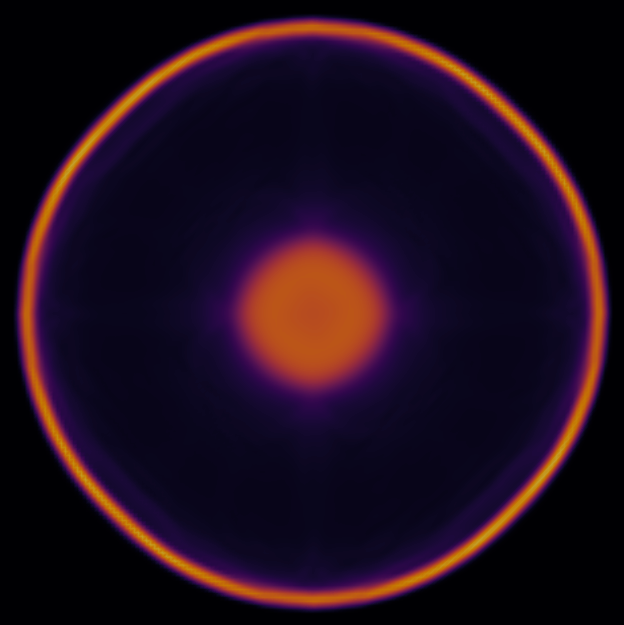}
        \caption{Rotated, ICNN-based M$_2$}
\end{subfigure}
\begin{subfigure}[t]{0.31\textwidth}	\centering
    \includegraphics[height=5cm]{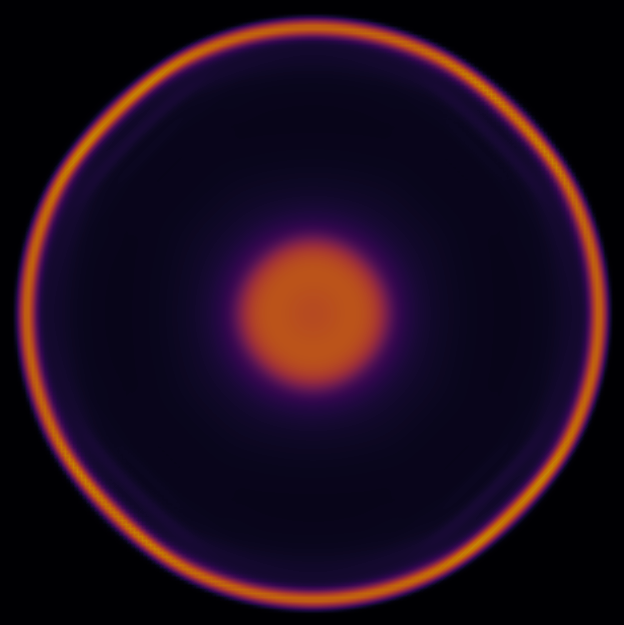}
        \caption{Newton-based M$_2$, reference}
\end{subfigure}
\begin{subfigure}[t]{0.03\textwidth}
    \includegraphics[height=5cm]{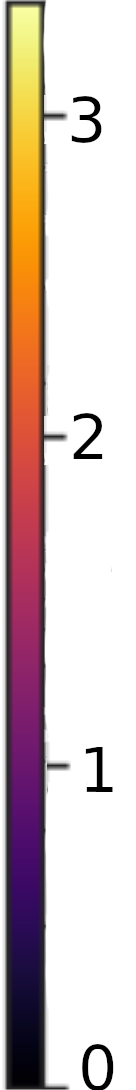}
    {\vspace{.02cm}}
\end{subfigure}
 \centering
       	\caption{Line source test case. (a) ICNN-based M$_2$ simulation with $\gamma=1e-3$ of the  line source test case. The test case dynamic at $t_f$ is captured well by the ICNN-simulation. At the wavefront, slight asymmetries are observed since rotational invariance is not enforced by the neural network-based simulation. (b) Rotatationally invariant ICNN-based simulation with $\gamma=1e-3$ of the line source test case. (c) Reference solution computed with a Newton optimizer.}
        \label{fig_dynamical_ansatz_linesource}
\end{figure}
The line source benchmark~\cite{Ganapol} is a torture test for numerical methods for kinetic equations and exposes the advantages and disadvantages of different velocity space discretizations~\cite{GarretHauck,osti_800993, PPNHauck, MCCLARREN20105597, RADICE2013648, Kitrt_paper}. 
We test the proposed neural network-based entropy closures on this problem.
The physical setup is given by an initial pulse of particles distributed isotropically along an infinite line in three-dimensional space. The particles travel through a homogeneous material medium with a constant scattering kernel $\sigma_{\textup{s}}=1$. The collision operator is given by
\begin{align}\label{eq_collision_mod_line}
    	Q(f)(\mathbf{v})=\int_{\mathbb{S}^2} [ f({\fv}_*) -f({\fv}) ] d{\fv}_*.
\end{align}
When the $x_3$ axis is aligned with the infinite line, the flux in the $x_3$ direction is zero and thus the resulting particle distribution is constant in the $x_3$ direction. Therefore, we consider a reduced moment system in the remainder of this section, in which the spatial domain is reduced to two dimensions on the $x_1$-$x_2$ plane.
The reduced moment system is given by
\begin{align}\label{eq_momentRT_linesource}
	\partial_t \fu + \nabla_\mathbf{x}\cdot\inner{\mathbf{v} \otimes{\fm}f_{\fu}^p}= \inner{{\fm}Q(f_{\fu}^p)},
\end{align}
where the analytic initial condition is the moment $\fu_0 = \inner{\fm f_0} $ of an isotropic Dirac pulse $f_0$. For numerical simulations, $f_0$ is approximated by a steep Gaussian 
\begin{align}\label{eq_linesource_ic}
	f_0&\approx\max \menge{\epsilon,\frac{1}{4\pi c}\exp \left(\frac{-\norm{\mathbf x}^2}{4\pi\epsilon}\right)}
 ,\qquad\forall \mathbf x\in \mathbf X,{\fv}\in \mathbb{S}^2,
\end{align}
with $c=0.0032$ and floor value $\epsilon=1\rm{e}-4$. 
The computational domain is chosen as $\mathbf X=[-1,1]^2$, such that the wavefront does not reach the domain boundary at the final time $t_f$. We use zero-value Dirichlet boundary conditions. 
The solver settings are displayed in Table~\ref{tab_linesource_setup}.
In general, the M$_N$ methods perform comparatively well in the  line source test case~\cite{GarretHauck} and accurately track the wavefront, whereas linear Galerkin-type closures as P$_N$ tend to oscillate and nodal methods as S$_N$ exhibit ray effects.

We compare the M$_2$ entropy closure computed with a Newton optimizer with $\gamma=1\rm{e}-3$ to an ICNN-based approximation in Fig.~\ref{fig_dynamical_ansatz_linesource}(c) and (a).  The test case dynamic at $t_f$ is captured well by the ICNN-based solution. However, slight asymmetries are observed near the wavefront since rotational invariance is not enforced by the ICNN-based simulation.
In Fig.~\ref{fig_dynamical_ansatz_linesource}(b), we show the result from an ICNN-based M$_2$ entropy closure with rotational invariance enforced via post-processing. When compared to the standard ICNN-based solution in Fig.~\ref{fig_dynamical_ansatz_linesource}(a), the rotationally invariant solution is qualitatively closer to the reference solution in Fig.~\ref{fig_dynamical_ansatz_linesource}(c), which illustrates the importance of preserving structural properties of the underlying system in the closure procedure. The current post-processing technique we used to generate the rotationally invariant solution in Fig.~\ref{fig_dynamical_ansatz_linesource}(b) does not maintain the entropy dissipation property and does not extend to general three-dimensional problems. The development of an efficient, rotationally-invarant, entropy-based closure is in the scope of the future work.



\subsection{Hohlraum test case}

\begin{figure}[th!]
	\centering
	\begin{subfigure}{0.3\textwidth}\centering
    	\includegraphics[width=0.6\textwidth]{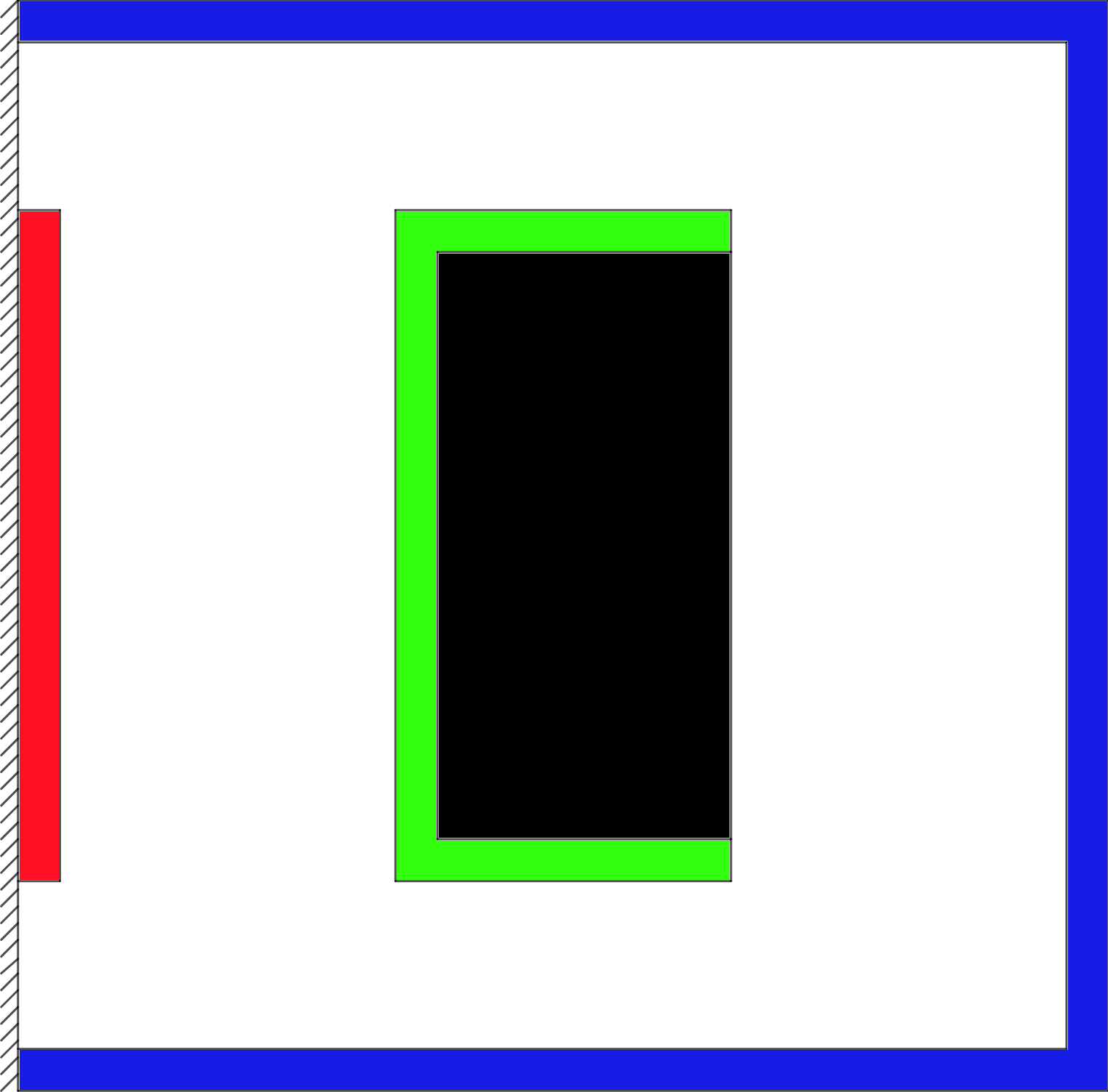}
   	\caption{Hohlraum Geometry}	\label{fig_hohlraum_geometry}
	\end{subfigure}  
 	\begin{subfigure}{0.3\textwidth}
	\centering
	\begin{tabular}{@{\extracolsep{1pt}}cccc}
	\toprule   
   	   Region & $\sigma_{\textup{t}}$ & $\sigma_{\textup{s}}$ &$\sigma_{\textup{a}}$   \\
	\midrule
    	white & $0.1$  & $0.1$ & $0.0$ \\
    	red & $100.0$  & $95.0$ & $5.0$ \\
    	green & $100.0$  & $90.0$ & $10.0$ \\
    	blue & $100.0$  & $0.0$ & $100.0$ \\
    	black & $100.0$  & $50.0$ & $50.0$ \\
	\bottomrule
	\end{tabular}
    	\caption{Material properties}    \label{table_hohlraum_color_codes}
	\end{subfigure}  
	\begin{subfigure}{0.3\textwidth}\centering
    	\includegraphics[width=0.7\textwidth]{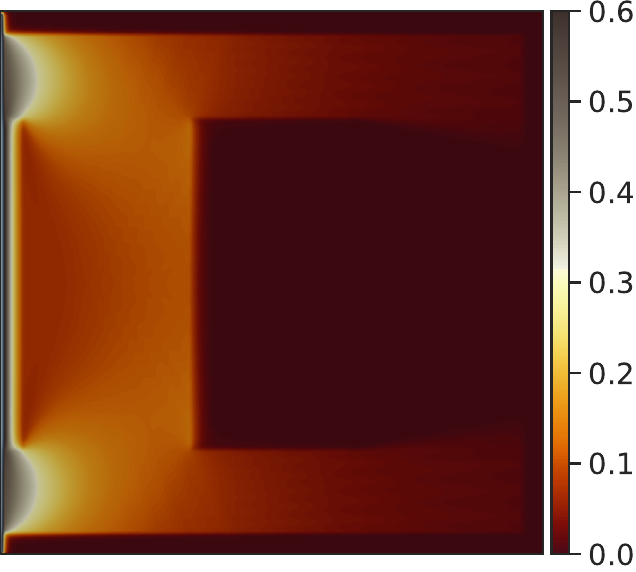}
    	\caption{S$_{60}$ reference solution}\label{fig_reference_sol_S50}
	\end{subfigure}
	\caption{Setup and reference solution for the hohlraum test case.  }\label{sec_num_hohlraum}
\end{figure}

The hohlraum test case~\cite{crockatt2020improvements} is a simplified version of hohlraum configurations used in nuclear fusion devices. Originally proposed in~\cite{osti_800993} was a coupled system of radiative transfer and an energy equation for the background material. A subsequently simplified version is described in~\cite{hohlraumHauckMclarren,osti_1550352,osti_1642263}, where the nonlinear thermal absorption and re-emission of radiation are replaced by particle scattering. The corresponding test-case design is displayed in Fig.~\ref{sec_num_hohlraum}. The collision operator is given by
\begin{align}\label{eq_collision_mod}
    	Q(f)(\mathbf{v})=\int_{\mathbb{S}^2}\sigma_{\textup{s}}(\mathbf{x}) \left[f({\fv}_*) -f({\fv})\right] d{\fv}_*,
\end{align}
where the isotropic collision kernel  $k(\mathbf{x},\fv\cdot\fv_*)=\sigma_{\textup{s}}(\mathbf{x})$ has a spatial dependence as specified in  the color-coded areas described in Fig.~\ref{fig_hohlraum_geometry} and Table~\ref{table_hohlraum_color_codes}.
Together with the space-dependent absorption $\sigma_{\textup{a}}({\mathbf{x}})$ this model yields the moment system
\begin{align}\label{eq_momentRT_hohlraum}
	\partial_t \fu + \nabla_\mathbf{x}\cdot\inner{\mathbf{v} \otimes{\fm}f_{\fu}^p}= \inner{{\fm}Q(f_{\fu}^p)} - \sigma_{\textup{a}} \fu,
\end{align}
with the collision operator $Q$ given by Eq.~\eqref{eq_collision_mod}. We have $\sigma_{\textup{t}}=\sigma_{\textup{a}}+\sigma_{\textup{s}}$. The spatial domain $\mathbf{X}=[-0.65,0.65]^2$ is two-dimensional. The initial condition is the moment $\fu_0 = \inner{\fm f_0} $ with $f_0=\epsilon=1\rm{e}-4$. The boundary conditions are zero valued Dirichlet conditions on all boundaries marked in blue. On the left hand side boundary, the incoming data is described by $ f_{\textup{BC}}(\fx,\fv,t)= 1.0$ for $\fx=[x_1=-0.65, -0.65\leq x_2\leq -0.65]$, $\fv\cdot [0,1]>0$ and $t\leq 0$. 

The reference solution for the hohlraum test case is a discrete ordinates simulation of order $60$, i.e. S$_{60}$,~\cite{lewis1984computational} and is displayed in Fig.~\ref{fig_reference_sol_S50}.

\subsubsection{Simulation quality of the regularized, ICNN-based M$_N$ method}
\begin{figure}
	\centering
	\begin{subfigure}[t]{0.226\textwidth}\centering
    	\includegraphics[height=3.7cm]{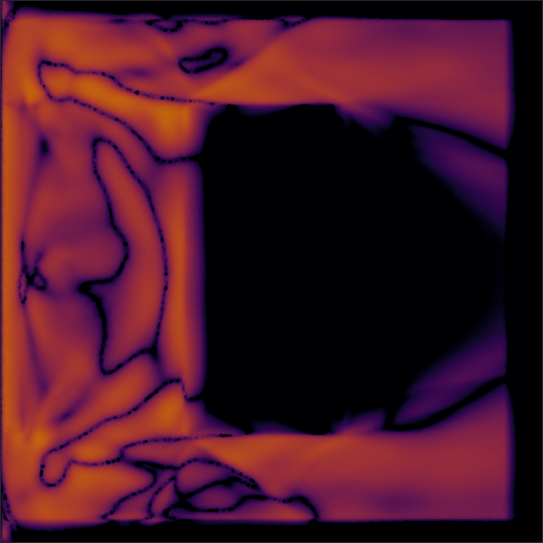}
    	\caption{$\abs{u_0^{\rm{ICNN}}-u_0^{\rm{Newton}}}$}
	\end{subfigure}
 	\begin{subfigure}[t]{0.226\textwidth}\centering
    	\includegraphics[height=3.7cm]{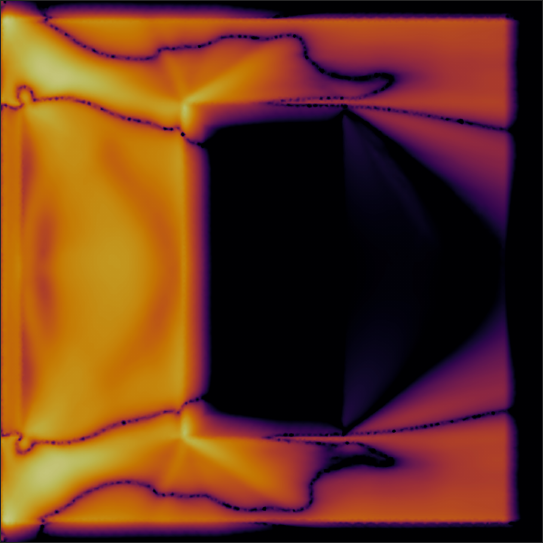}
    	\caption{$\abs{u_0^{\rm{ref}}-u_0^{\rm{Newton}}}$}
	\end{subfigure}
	\begin{subfigure}[t]{0.226\textwidth}\centering
    	\includegraphics[height=3.7cm]{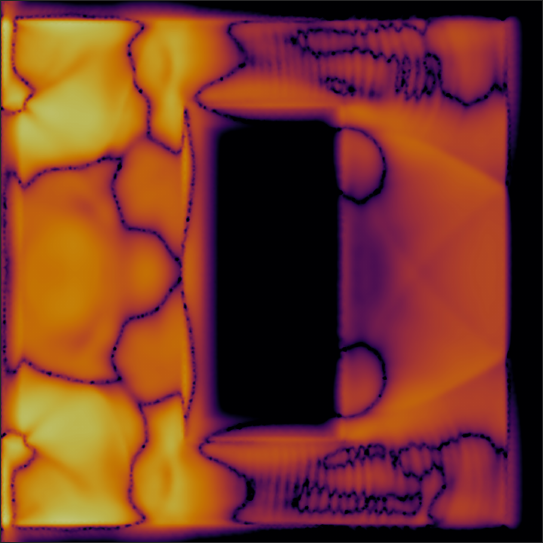}
    	\caption{$\abs{u_0^{\rm{ICNN}}-u_0^{\rm{Newton}}}$}
	\end{subfigure}
 	\begin{subfigure}[t]{0.265\textwidth}\centering
    	\includegraphics[height=3.7cm]{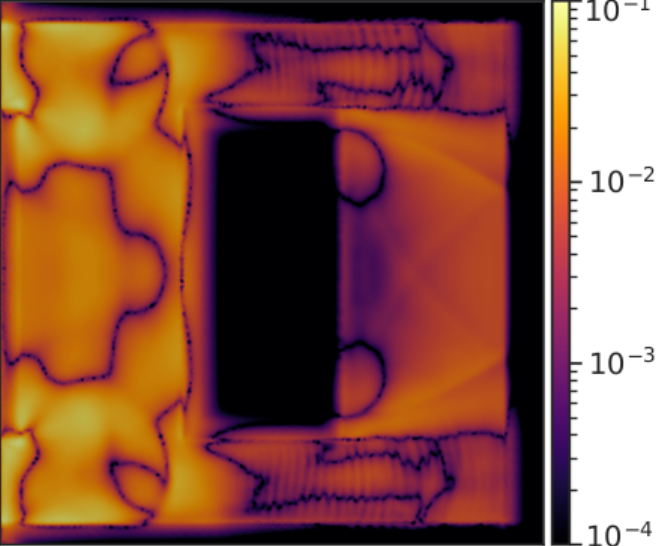}
    	\caption{$\abs{u_0^{\rm{ref}}-u_0^{\rm{Newton}}}$}
	\end{subfigure}
	\caption{Regularized entropy closures with ${\gamma=1\rm{e}{-3}}$. M$_2$ closure is displayed in a) and b), and M$_3$ in c) and d). Fig. a) and c) display the absolute difference between Newton-based and ICNN-based simulation, and Fig. b) and d) display the absolute difference between Newton-based simulation and the S$_{60}$ reference solution. The network approximation error (a) and (c) is much smaller than the difference between M$_N$ and S$_N$ models (b) and (d).}
	\label{fig_hohlraum_Newton_ICNN_S50}
\end{figure}
First, we evaluate the difference between the order zero moment of an ICNN-based entropy closure, Newton-based entropy closure with regularization parameter $\gamma=1\rm{e}-3$ and the S$_{60}$ reference solution in Fig.~\ref{fig_hohlraum_Newton_ICNN_S50}. For the ICNN-based simulation, Algorithm~\ref{alg_network_inference_nr} is used. One sees that the difference between ICNN and Newton-based regularized M$_2$ simulation is smaller than the absolute difference between the Newton-based solution and the S$_{60}$ reference solution. In the case of the regularized M$_3$ closure, both errors are in the same order of magnitude.
\begin{figure}
	\centering
	\begin{subfigure}[t]{0.226\textwidth}\centering
    	\includegraphics[height=3.7cm]{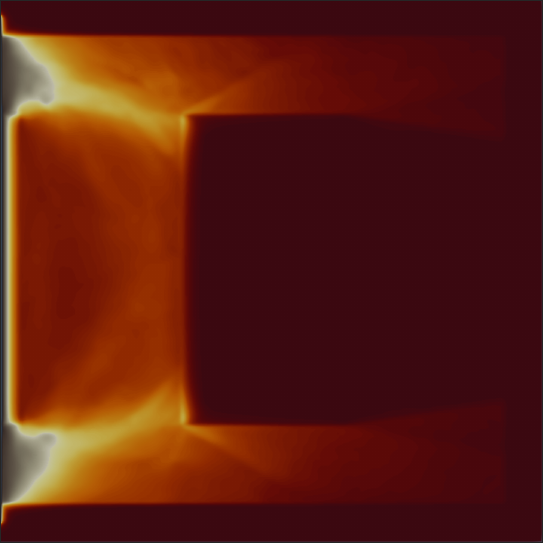}
    	\caption{${\gamma=0}$}
	\end{subfigure}
 	\begin{subfigure}[t]{0.226\textwidth}\centering
    	\includegraphics[height=3.7cm]{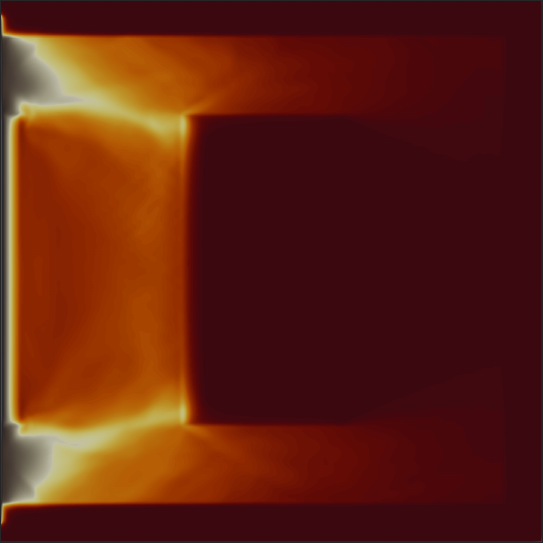}
    	\caption{${\gamma=1\mathrm{e}{-3}}$}
	\end{subfigure}
    \begin{subfigure}[t]{0.226\textwidth}\centering
    	\includegraphics[height=3.7cm]{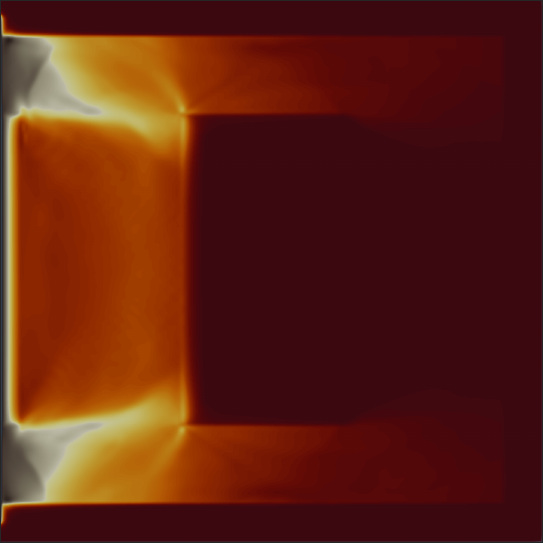}
    	\caption{${\gamma=1\mathrm{e}{-2}}$}
	\end{subfigure}
    \begin{subfigure}[t]{0.266\textwidth}\centering
    	\includegraphics[height=3.7cm]{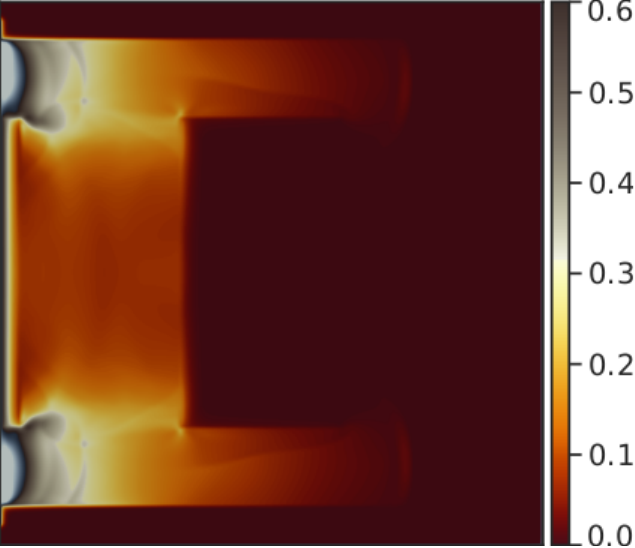}
    	\caption{${\gamma=1\mathrm{e}{-1}}$}
	\end{subfigure}
    \begin{subfigure}[t]{0.226\textwidth}\centering
    	\includegraphics[height=3.7cm]{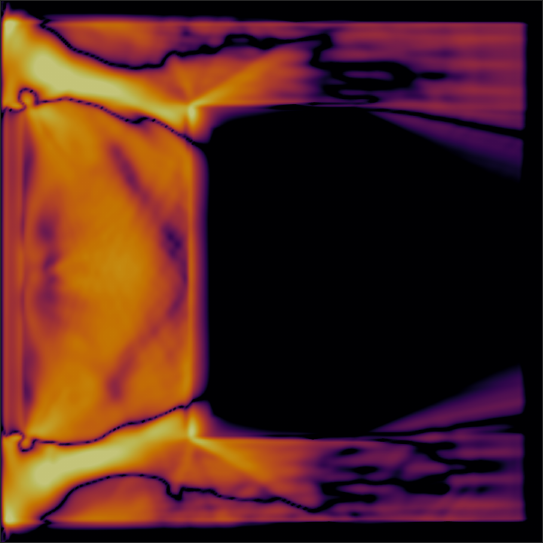}
    	\caption{${\gamma=0}$}
	\end{subfigure}
 	\begin{subfigure}[t]{0.226\textwidth}\centering
    	\includegraphics[height=3.7cm]{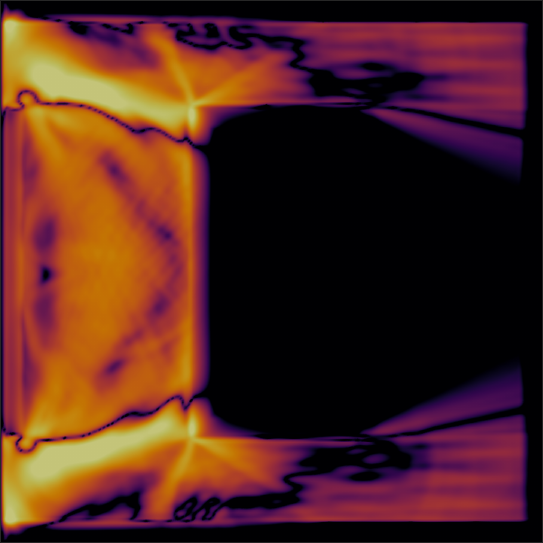}
    	\caption{${\gamma=1\mathrm{e}{-3}}$}
	\end{subfigure}
    \begin{subfigure}[t]{0.226\textwidth}\centering
    	\includegraphics[height=3.7cm]{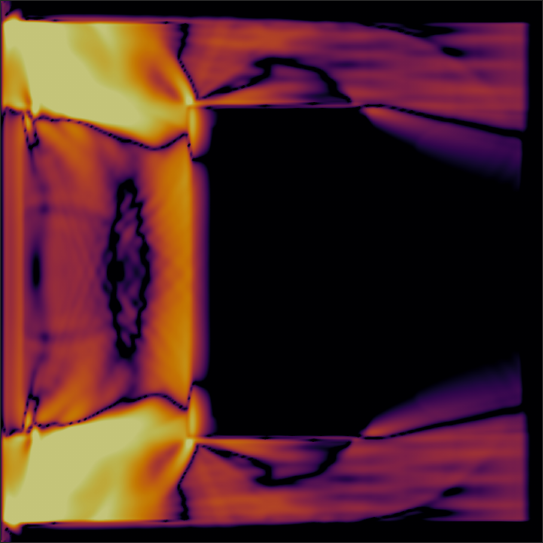}
    	\caption{${\gamma=1\mathrm{e}{-2}}$}
	\end{subfigure}
    \begin{subfigure}[t]{0.266\textwidth}\centering
    	\includegraphics[height=3.7cm]{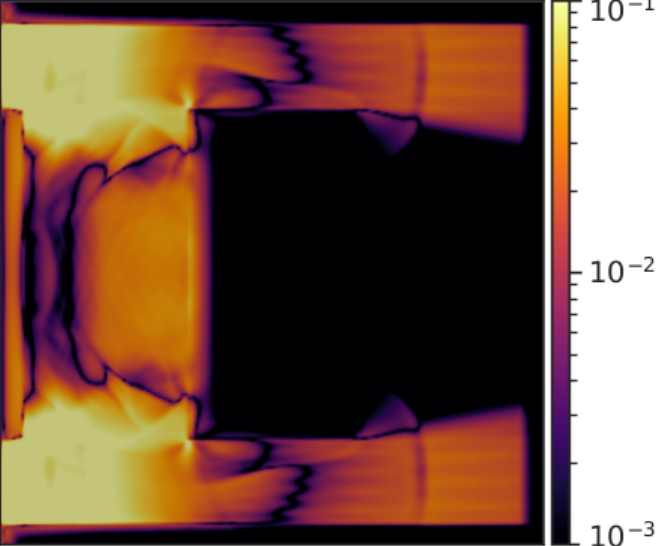}
    	\caption{${\gamma=1\mathrm{e}{-1}}$}
	\end{subfigure}
	\caption{ICNN-based M$_2$ closures with different regularization (top row) and their absolute difference to the S$_{60}$ solution (bottom row). The $\gamma=1\mathrm{e}{-3}$ model has the smallest error followed by the $\gamma=0$ model, where regularization and neural network approximation errors have the best trade-off. The heavily regularized model $\gamma=1\mathrm{e}{-1}$ experiences larger errors due to regularization in the absorption region on the right side of the domain increasing the model deviation to the S$_{60}$ solution.}
	\label{fig_hohlraum_reflection_M2}
\end{figure}
\begin{figure}
	\centering
	\begin{subfigure}[t]{0.226\textwidth}\centering
    	\includegraphics[height=3.7cm]{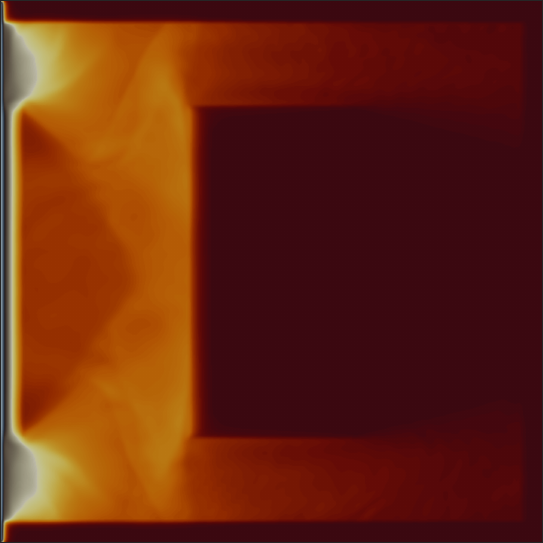}
    	\caption{${\gamma=0}$}
	\end{subfigure}  
	\begin{subfigure}[t]{0.226\textwidth}\centering
    	\includegraphics[height=3.7cm]{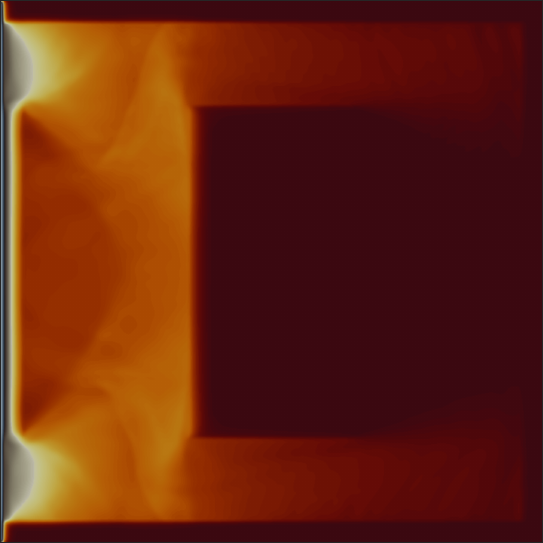}
    	\caption{${\gamma=1\mathrm{e}{-3}}$}
	\end{subfigure}   
	\begin{subfigure}[t]{0.226\textwidth}\centering
    	\includegraphics[height=3.7cm]{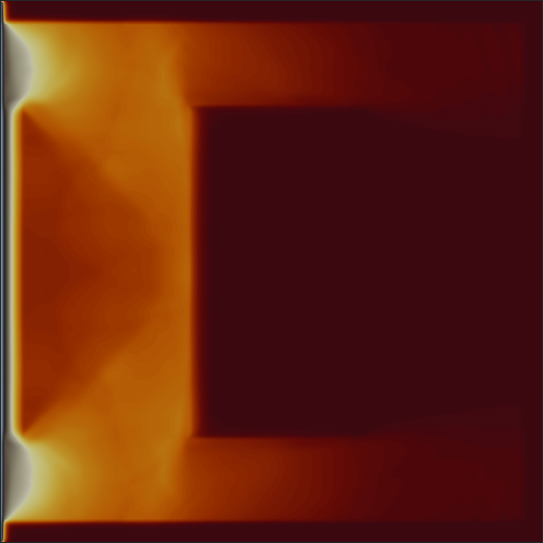}
    	\caption{${\gamma=1\mathrm{e}{-2}}$}
	\end{subfigure}  
 	\begin{subfigure}[t]{0.266\textwidth}\centering
    	\includegraphics[height=3.7cm]{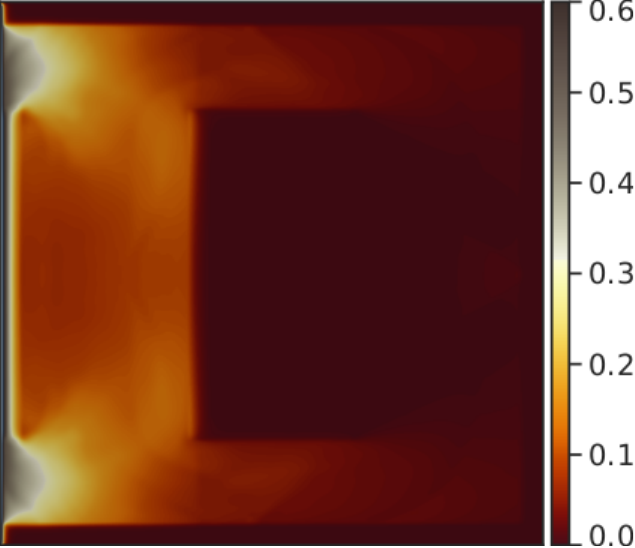}
    	\caption{${\gamma=1\mathrm{e}{-1}}$}
	\end{subfigure}   
	\begin{subfigure}[t]{0.226\textwidth}\centering
    	\includegraphics[height=3.7cm]{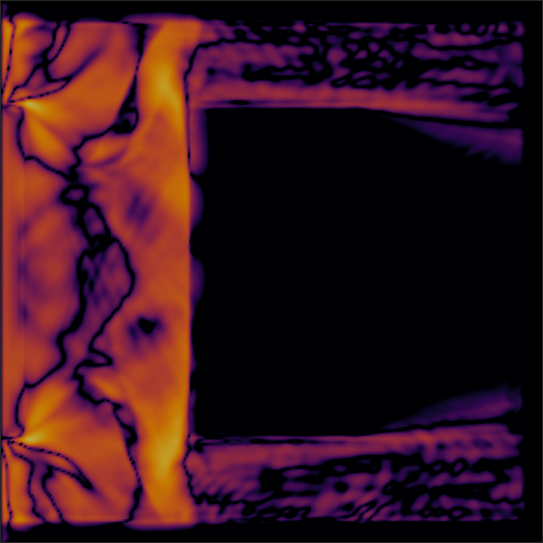}
    	\caption{${\gamma=0}$}
	\end{subfigure}
 	\begin{subfigure}[t]{0.226\textwidth}\centering
    	\includegraphics[height=3.7cm]{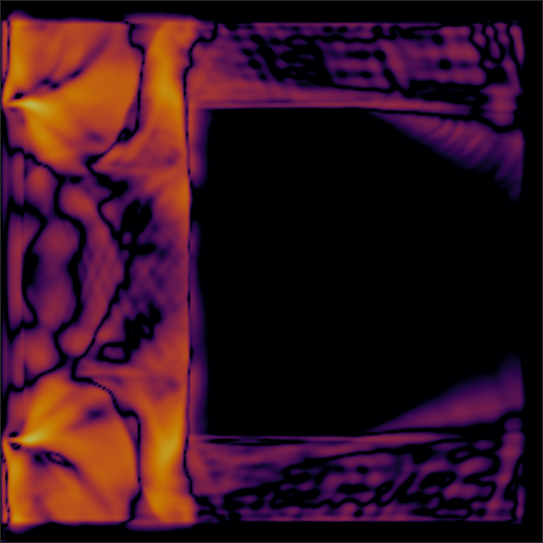}
    	\caption{${\gamma=1\mathrm{e}{-3}}$}
	\end{subfigure}  
 	\begin{subfigure}[t]{0.226\textwidth}\centering
    	\includegraphics[height=3.7cm]{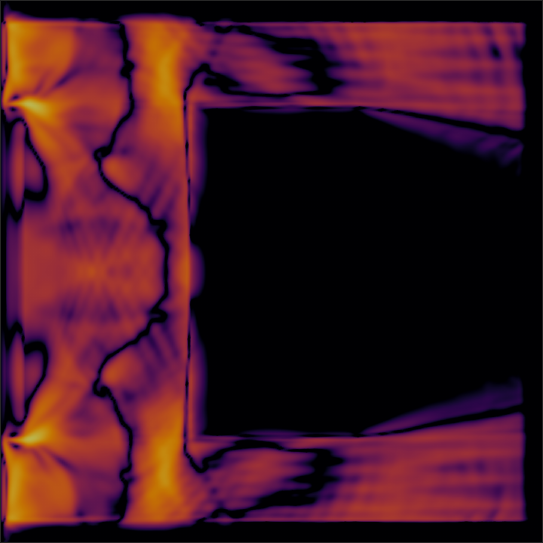}
    	\caption{${\gamma=1\mathrm{e}{-2}}$}
	\end{subfigure}   
 	\begin{subfigure}[t]{0.266\textwidth}\centering
    	\includegraphics[height=3.7cm]{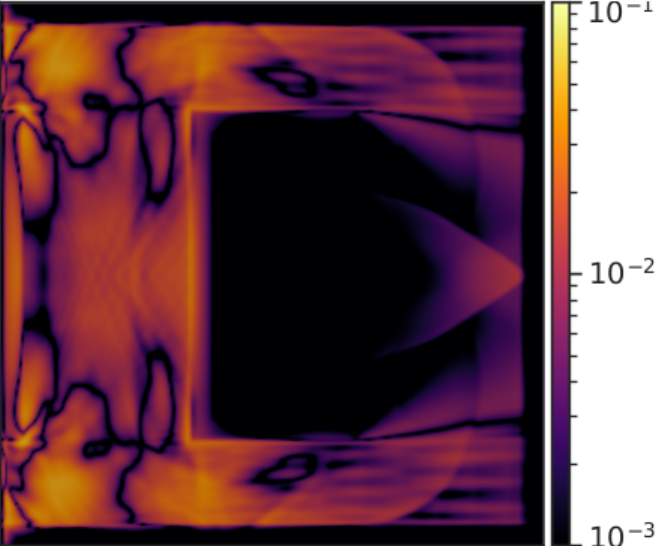}
    	\caption{${\gamma=1\mathrm{e}{-1}}$}
	\end{subfigure}    

	\caption{ICNN-based M$_3$ closures with different regularization (top row) and their absolute difference to the S$_{60}$ solution (bottom row). The $\gamma=1\mathrm{e}{-3}$ model has the smallest error followed by the $\gamma=1\mathrm{e}{-2}$ model, where regularization and neural network approximation errors have the best trade-off. The heavily regularized model $\gamma=1\mathrm{e}{-1}$ experiences larger errors due to regularization in the absorption region on the right side of the domain.}
	\label{fig_hohlraum_N3}
\end{figure}
\begin{figure}
	\centering
	\begin{subfigure}[t]{0.226\textwidth}\centering
    	\includegraphics[height=3.7cm]{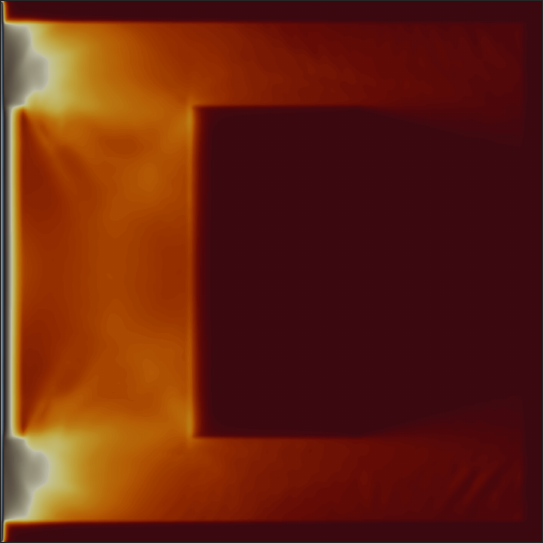}
    	\caption{${\gamma=0}$}
	\end{subfigure}  
	\begin{subfigure}[t]{0.226\textwidth}\centering
    	\includegraphics[height=3.7cm]{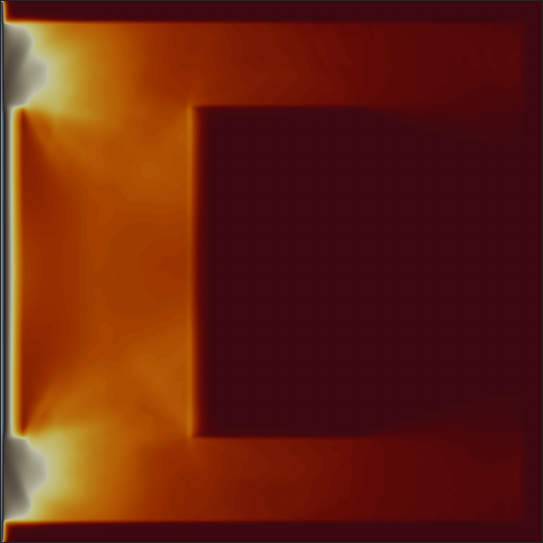}
    	\caption{${\gamma=1\mathrm{e}{-3}}$}
	\end{subfigure}   
	\begin{subfigure}[t]{0.226\textwidth}\centering
    	\includegraphics[height=3.7cm]{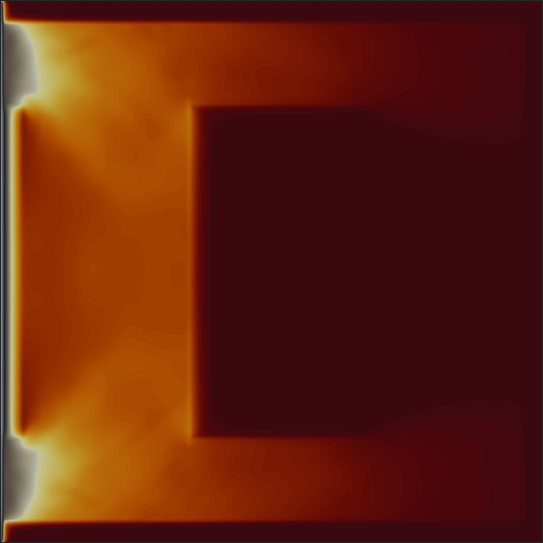}
    	\caption{${\gamma=1\mathrm{e}{-2}}$}
	\end{subfigure}  
 	\begin{subfigure}[t]{0.266\textwidth}\centering
    	\includegraphics[height=3.7cm]{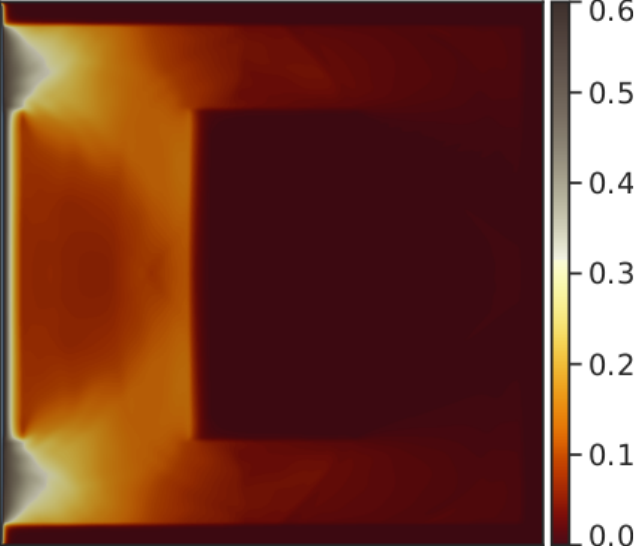}
    	\caption{${\gamma=1\mathrm{e}{-1}}$}
	\end{subfigure}   
	\begin{subfigure}[t]{0.226\textwidth}\centering
    	\includegraphics[height=3.7cm]{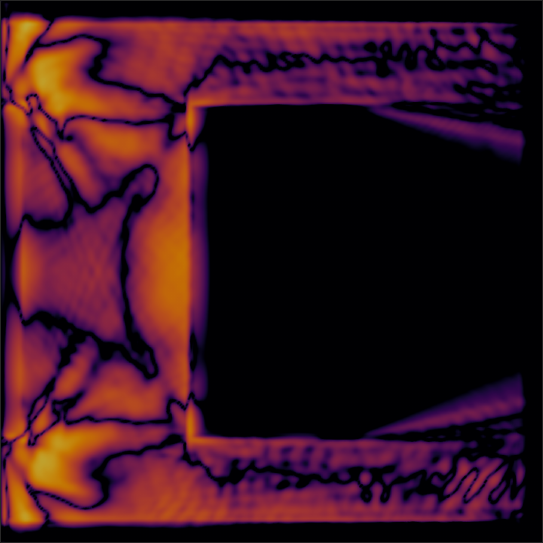}
    	\caption{${\gamma=0}$}
	\end{subfigure}
 	\begin{subfigure}[t]{0.226\textwidth}\centering
    	\includegraphics[height=3.7cm]{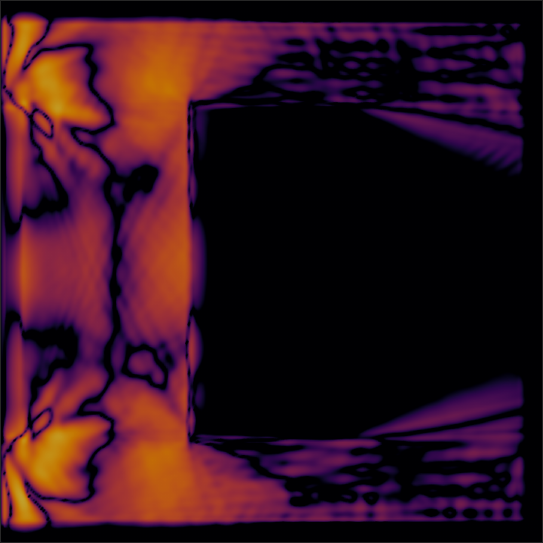}
    	\caption{${\gamma=1\mathrm{e}{-3}}$}
	\end{subfigure}  
 	\begin{subfigure}[t]{0.226\textwidth}\centering
    	\includegraphics[height=3.7cm]{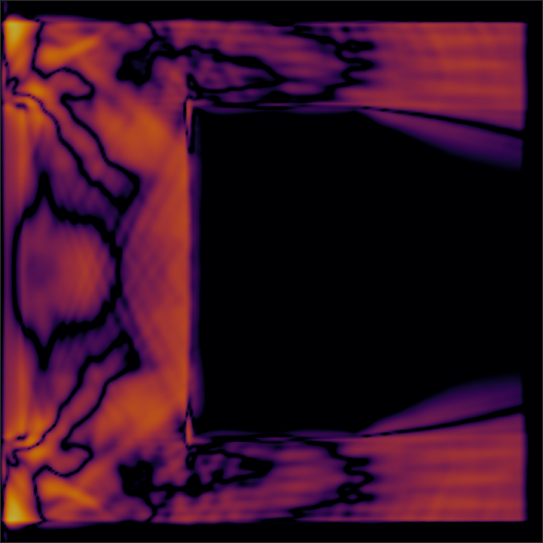}
    	\caption{${\gamma=1\mathrm{e}{-2}}$}
	\end{subfigure}   
 	\begin{subfigure}[t]{0.266\textwidth}\centering
    	\includegraphics[height=3.7cm]{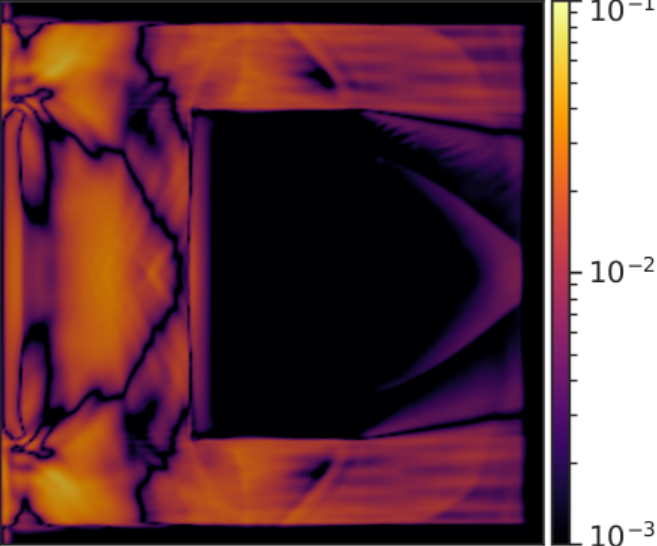}
    	\caption{${\gamma=1\mathrm{e}{-1}}$}
	\end{subfigure}    
	\caption{ICNN-based M$_4$ closures with different regularization. The absolute difference to the S$_{60}$ solution between ICNN-based M$_4$ closure is displayed in the bottom row. The $\gamma=1\mathrm{e}{-2}$ model has the smallest error followed by the $\gamma=1\mathrm{e}{-3}$ model, where regularization and neural network approximation errors have the best trade-off. The heavily regularized model $\gamma=1\mathrm{e}{-1}$ experiences regularization errors in the absorption region on the right side of the domain and the non-regularized model experiences artifacts in the inflow region on the left.}
	\label{fig_hohlraum_N4}
\end{figure}

Next, we compare the simulated order zero moment $u_0$ for various closure orders and regularization levels to the S$_{60}$ reference solution $u_0^{\text{ref}}$. We consider the pointwise absolute difference
\begin{align}
	e(t_f,\mathbf{x})= \abs{u_0^{\text{ref}}(t_f,\mathbf x)-u_0(t_f,\mathbf x)}
\end{align}
at final time $t_f$ of the given numerical method.
Figure~\ref{fig_hohlraum_reflection_M2},~\ref{fig_hohlraum_N3} and~\ref{fig_hohlraum_N4} show the corresponding evaluations of $u_0$ of the ICNN-based  M$_2$, M$_3$ and M$_4$ closures in the first row and additionally $e(t_f,\mathbf{x})$ in the second row.
We see that with increased closure order, the errors in the streaming regions near the left side inflows are significantly reduced and the solution quality on the right half of the simulation increases by almost an order of magnitude.
Considering Fig.~\ref{fig_hohlraum_N3} and~\ref{fig_hohlraum_N4}, we further see that the difference to the S$_{60}$ solution increases for $\gamma=0.1$ in the M$_3$ and M$_4$ simulations compared to smaller $\gamma$, which is an effect of the regularization error of the entropy closure, see Theorem~\ref{theo_reg_ansatz_error}. 
Consequently, a trade-off between training performance and regularization error has to be made when configuring a simulation: Higher order entropy closures yield lower differences to the reference solutions but are harder to train. Higher regularization increases training performance but introduces the regularization error if $\gamma$ is significantly bigger than the grid resolution. Table~\ref{tab_ICNN_hohlraum} shows the best regularization level for each closure order and we see, that $\gamma=1\mathrm{e}{-3}$ yields the best results for $\Delta x=7.5\mathrm{e}{-3}$.
The displayed errors are computed as
\begin{align}\label{eq_rel_integrated_error}
	e_{\rm{rel},u_0} = \frac{\int_{\mathbf X} \abs{u_0^{\text{ref}}(t_f,\mathbf x)-u_0(t_f,\mathbf x)} \intD \mathbf x}{\int_{\mathbf X} \abs{u_0^{\text{ref}}(t_f,\mathbf x))} \intD \mathbf x}.
\end{align}
Note that the neural network validation results in Table~\ref{tab_closures_reg} further support this regularization range.

\begin{table}[t!]
\centering
\caption{Relative integrated errors $e_{\rm{rel},u_0}$ of ICNN-based entropy closures corresponding to the S$_{60}$ hohlraum simulation. Particularly higher-order closures benefit from regularized surrogate models. {A large regularization, i.e. $\gamma=1\mathrm{e}{-1}$, dominates the neural network and method errors. The best trade-off between the numerical errors is $\gamma=1\mathrm{e}{-3}$}.} \label{tab_ICNN_hohlraum}
\begin{tabular}{@{\extracolsep{2pt}}lccc}
\toprule   
    $\gamma$ &  M$_2$   & M$_3$ & M$_4$   \\             	\midrule 
	  	$0$ &${7.76\mathrm{e}{-2}}$ 	                  	& $3.60\mathrm{e}{-2} $      	& $4.07\mathrm{e}{-2} $ \\ 
     	$1\mathrm{e}{-3}$ &$\mathbf{7.1\mathrm{e}{-2}}$ 	& $\mathbf{3.34\mathrm{e}{-2}} $ & $\mathbf{2.01\mathrm{e}{-2}} $\\ 
     	$1\mathrm{e}{-2}$ &$1.06\mathrm{e}{-1}$ 	             	& $3.86\mathrm{e}{-2} $      	& ${6.33\mathrm{e}{-2}} $\\ 
     	$1\mathrm{e}{-1}$ &$1.88\mathrm{e}{-1}$ 	          	& $1.66\mathrm{e}{-1} $      	& $1.57\mathrm{e}{-1} $ \\ 
\bottomrule
\end{tabular}
\end{table}

\subsubsection{Performance comparison to the P$_N$ and S$_N$ methods}

\begin{figure}[t!]
	\centering
	\begin{subfigure}[t]{0.226\textwidth}\centering
    	\includegraphics[height=3.7cm]{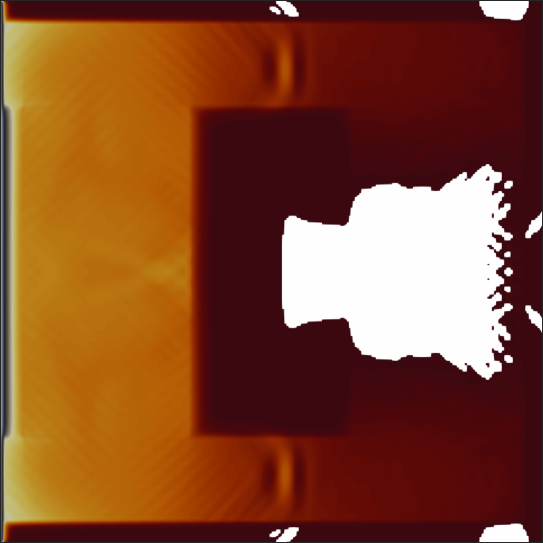}
    	\caption{P$_3$}
	\end{subfigure}  
	\begin{subfigure}[t]{0.226\textwidth}\centering
    	\includegraphics[height=3.7cm]{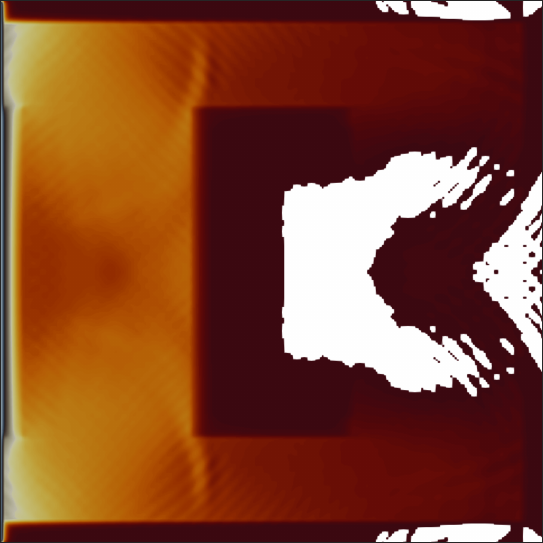}
    	\caption{P$_5$}
	\end{subfigure}   
	\begin{subfigure}[t]{0.226\textwidth}\centering
    	\includegraphics[height=3.7cm]{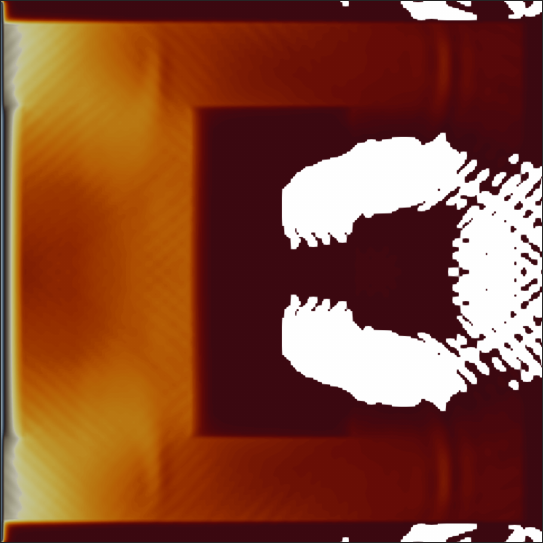}
    	\caption{P$_7$}
	\end{subfigure}  
 	\begin{subfigure}[t]{0.266\textwidth}\centering
    	\includegraphics[height=3.7cm]{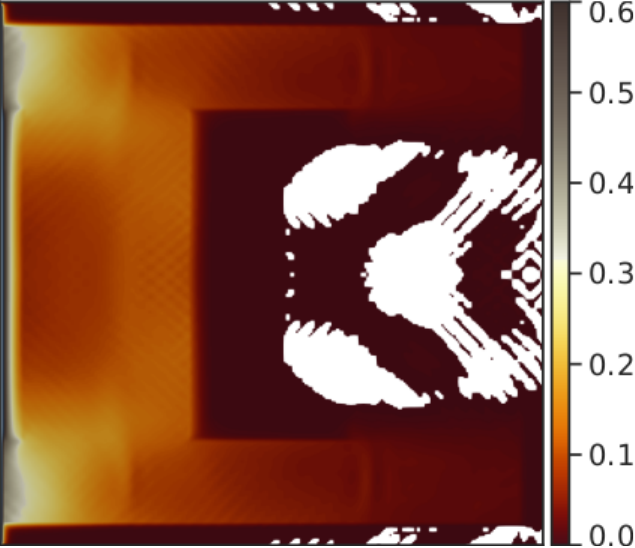}
    	\caption{P$_9$}
	\end{subfigure}   
 
	\begin{subfigure}[t]{0.226\textwidth}\centering
    	\includegraphics[height=3.7cm]{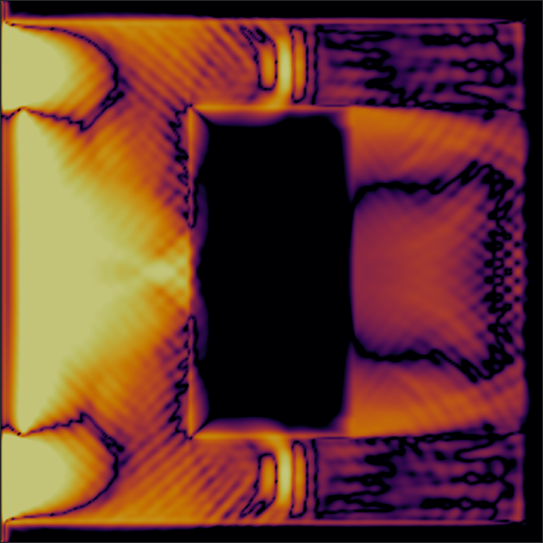}
    	\caption{P$_3$}
	\end{subfigure}
 	\begin{subfigure}[t]{0.226\textwidth}\centering
    	\includegraphics[height=3.7cm]{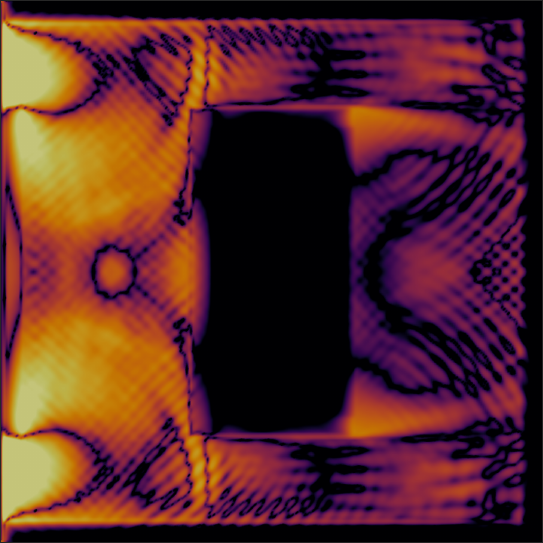}
    	\caption{P$_5$}
	\end{subfigure}  
 	\begin{subfigure}[t]{0.226\textwidth}\centering
    	\includegraphics[height=3.7cm]{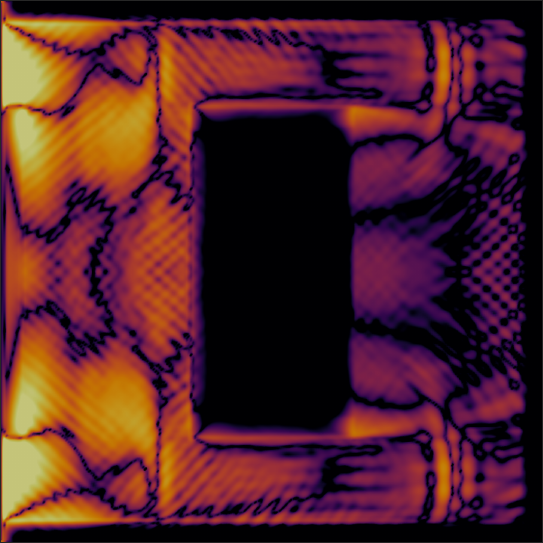}
    	\caption{P$_7$}
	\end{subfigure}   
 	\begin{subfigure}[t]{0.266\textwidth}\centering
    	\includegraphics[height=3.7cm]{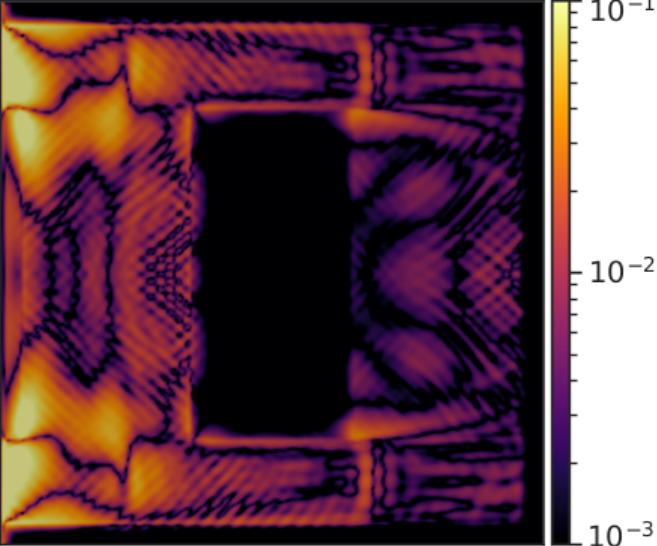}
    	\caption{P$_9$}
	\end{subfigure}    

	\caption{P$_N$ closures (top row) and their absolute difference to the S$_{60}$ solution (bottom row). The white areas in the solution plots denote unphysical negative solution values. The P$_N$ closure displays higher numerical errors compared to S$_N$ and M$_N$ simulations with similar degrees of freedom.}
	\label{fig_hohlraum_PN}
\end{figure}
\begin{figure}
	\centering
	\begin{subfigure}[t]{0.226\textwidth}\centering
    	\includegraphics[height=3.7cm]{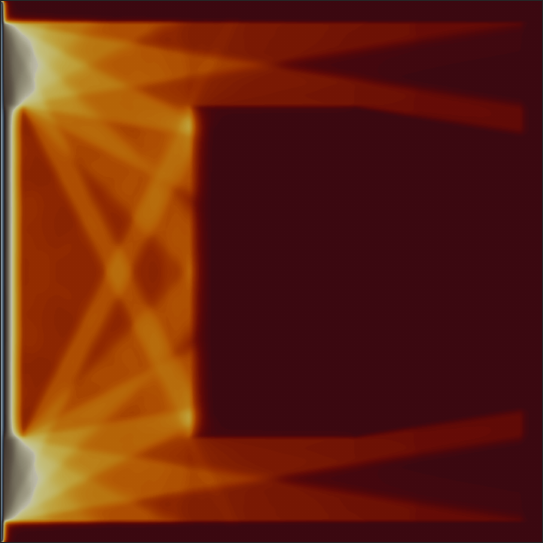}
    	\caption{S$_{11}$}
     \label{fig_hohlraum_S10}
	\end{subfigure}  
	\begin{subfigure}[t]{0.226\textwidth}\centering
    	\includegraphics[height=3.7cm]{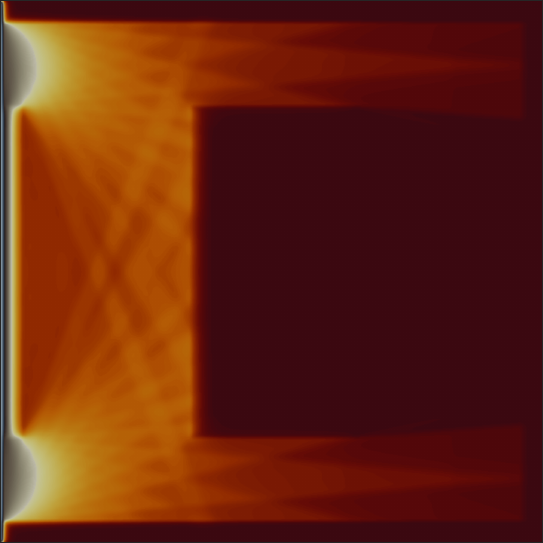}
    	\caption{S$_{21}$}
	\end{subfigure}   
	\begin{subfigure}[t]{0.226\textwidth}\centering
    	\includegraphics[height=3.7cm]{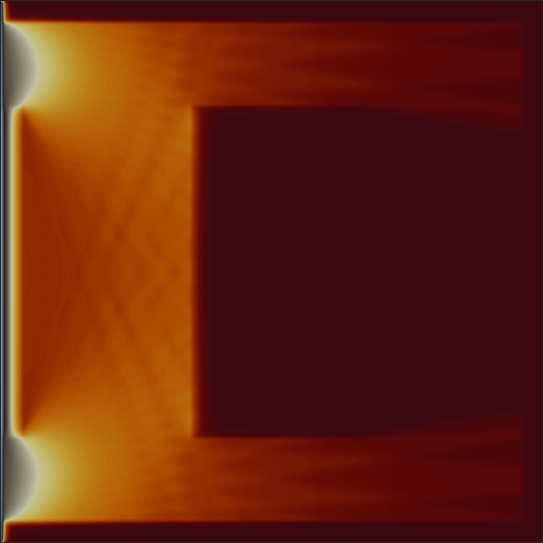}
    	\caption{S$_{31}$}
	\end{subfigure}  
 	\begin{subfigure}[t]{0.266\textwidth}\centering
    	\includegraphics[height=3.7cm]{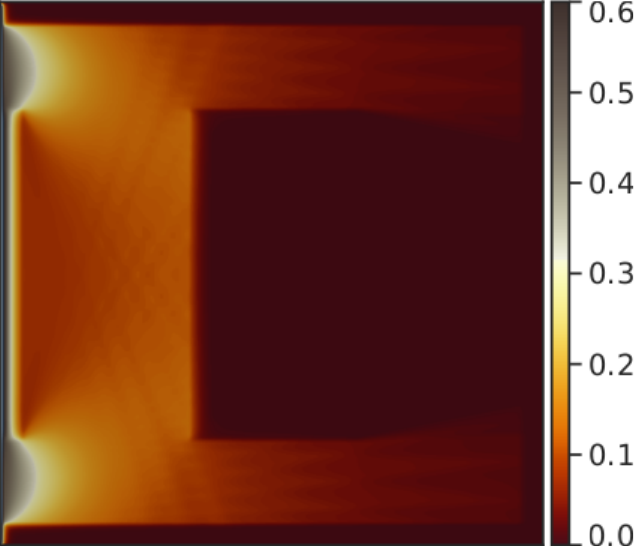}
    	\caption{S$_{41}$}
	\end{subfigure}   
	\begin{subfigure}[t]{0.226\textwidth}\centering
    	\includegraphics[height=3.7cm]{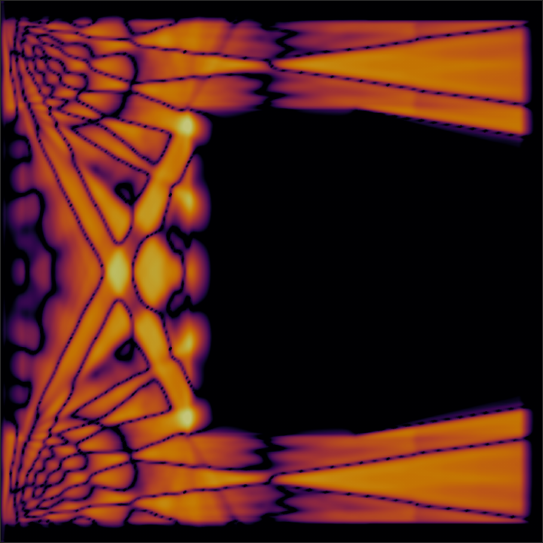}
    	\caption{S$_{11}$}
	\end{subfigure}
 	\begin{subfigure}[t]{0.226\textwidth}\centering
    	\includegraphics[height=3.7cm]{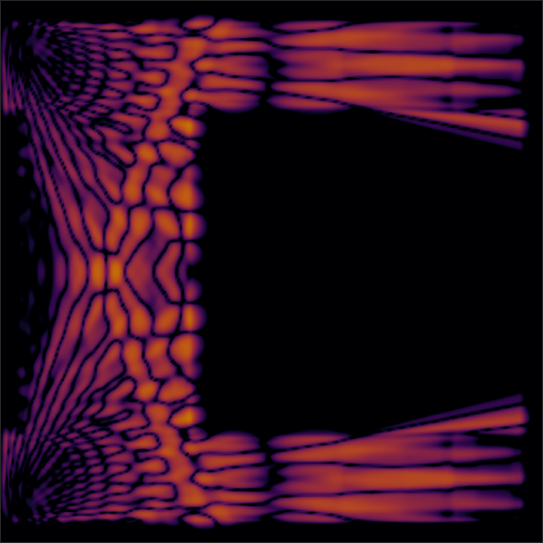}
    	\caption{S$_{21}$}
	\end{subfigure}  
 	\begin{subfigure}[t]{0.226\textwidth}\centering
    	\includegraphics[height=3.7cm]{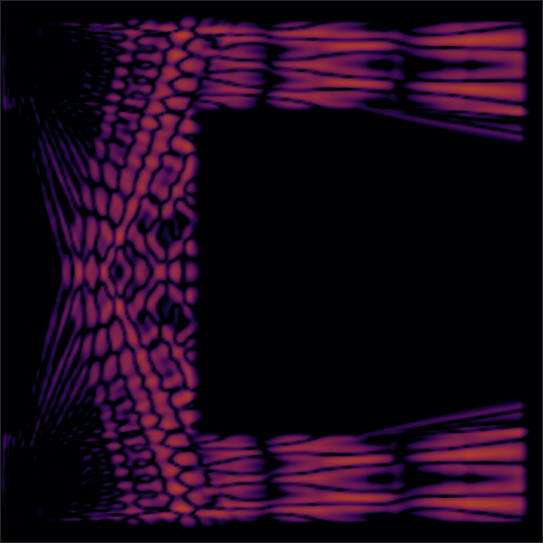}
    	\caption{S$_{31}$}
	\end{subfigure}   
 	\begin{subfigure}[t]{0.266\textwidth}\centering
    	\includegraphics[height=3.7cm]{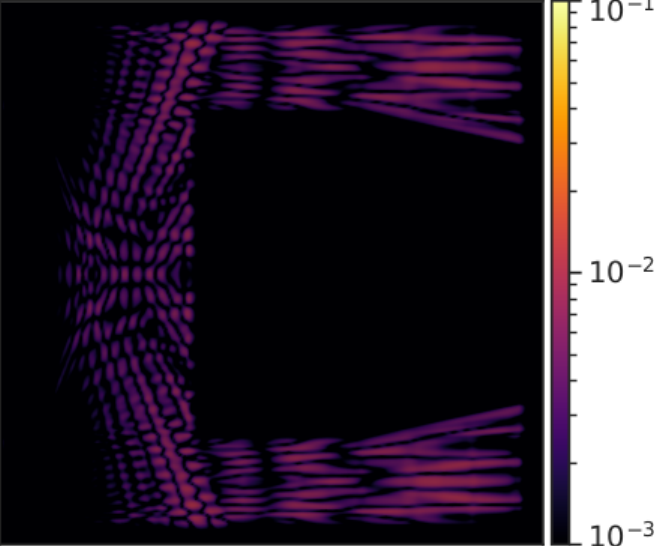}
    	\caption{S$_{41}$}
	\end{subfigure}   

	\caption{S$_N$ simulations (top row) and their absolute differences to the S$_{60}$ solution (bottom row). The numerical artifacts are ray effects, which can be mitigated by very high order S$_N$ schemes or diffusive rotation methods~\cite{hauck2019filtered,frank2020ray}.}
	\label{fig_hohlraum_SN}
\end{figure}
Common other methods for radiative transport simulations are the P$_N$ and S$_N$ methods. The former is a moment method with a different closure. Instead of using the entropy-based closure, a simple truncation closure is used, which yields a linear reconstruction of the kinetic density $f_\fu$ from the moment basis,
\begin{align}
	f_\fu({\fv}) = \fu\cdot{\fm}({\fv}).
\end{align}
The P$_N$ closure can also be viewed as the special case of a quadratic entropy density $\eta(g)=g^2$. Although computationally highly efficient, the pitfalls of the P$_N$ method are oscillations and negative solutions, which can be seen in Fig.~\ref{fig_hohlraum_PN}. Furthermore, higher moment orders are required to ensure the same order of accuracy as M$_N$ closures. Figure~\ref{fig_hohlraum_PN} shows that especially in the shielded regions on the right-hand side of the computational domain, the P$_N$ method experiences higher numerical errors than the M$_N$ simulation.

The S$_N$ method discretizes the velocity space directly using a tensorized Gauss-Legendre quadrature rule of order $N$. The resulting transport system is of size equal to the number of quadrature points and the system's equations are coupled only by the collision operator. Consequently, the S$_N$ method has the lowest computational cost in comparison to P$_N$ and M$_N$ systems of similar size. However, the required quadrature order for a high-quality simulation is much higher than the moment order of P$_N$ and M$_N$ methods since low-order S$_N$ methods typically exhibit ray-effect artifacts. The number of quadrature points scales quadratically with the quadrature order, which results in a large system of equations for high-order S$_N$ methods. On modern high-performance clusters, the performance bottleneck is typically the memory footprint of the simulation. An overview of S$_N$ simulations of the hohlraum test case for different $N$ is given in Fig.~\ref{fig_hohlraum_SN}.

Thus a comparison between P$_N$, S$_N$, M$_N$, and neural network-based M$_N$ methods needs to consider computational time and memory footprint besides the difference to the reference solution. To this end, we measure the wall-time of a hohlraum simulation with the setup given by Table~\ref{tab_linesource_setup}  for all previously discussed non-regularized, regularized, and neural network-based M$_N$, as well as the P$_N$ and S$_N$ methods. Each simulation is computed on the same, isolated hardware, using a $16$-core CPU with $64$ GB Memory, such that each test case fits entirely into the system's memory. Figure~\ref{fig_hohlraum_performance} compares the computational performance of all methods with different moment, respectively quadrature orders. 
\begin{figure}[tb!]
	\centering
 	\begin{subfigure}[t]{0.32\textwidth}\centering
    	\includegraphics[width=\textwidth]{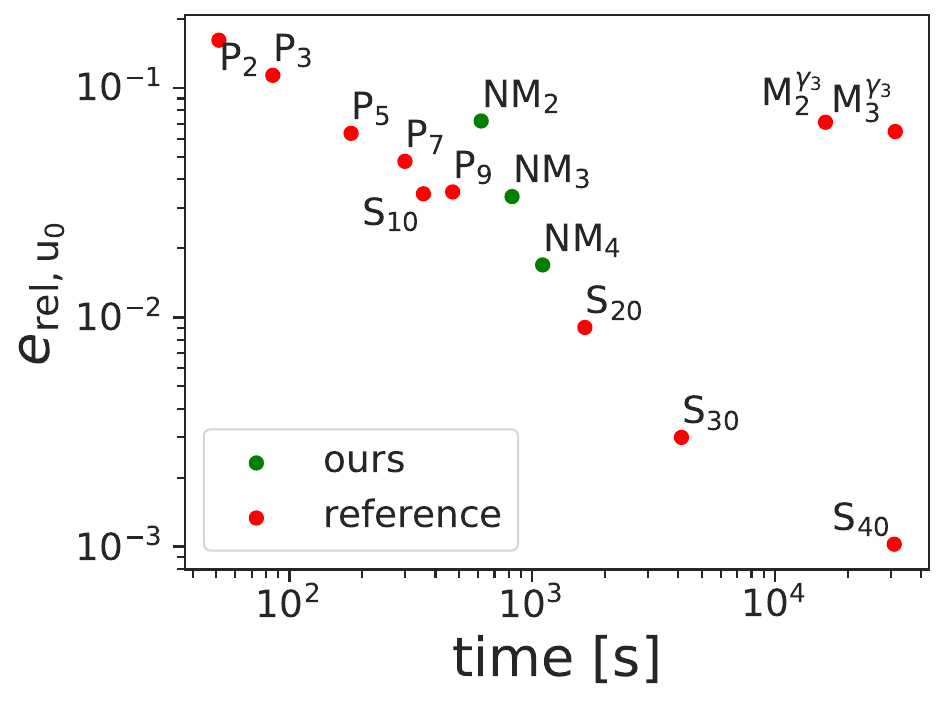}
    	\caption{$e_{\rm{rel},u_0}$ vs simulation time }	\label{fig_hohlraum_performancea}

	\end{subfigure}
 	\begin{subfigure}[t]{0.32\textwidth}\centering
    	\includegraphics[width=\textwidth]{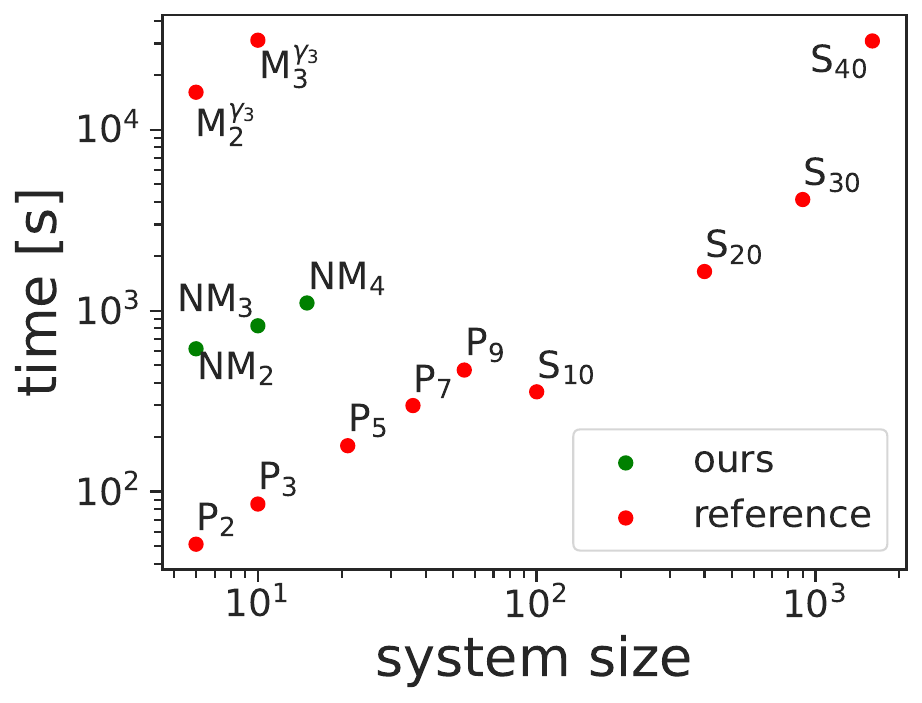}
    	\caption{Simulation time vs system size }	\label{fig_hohlraum_performanceb}

	\end{subfigure}
	\begin{subfigure}[t]{0.32\textwidth}\centering
    	\includegraphics[width=\textwidth]{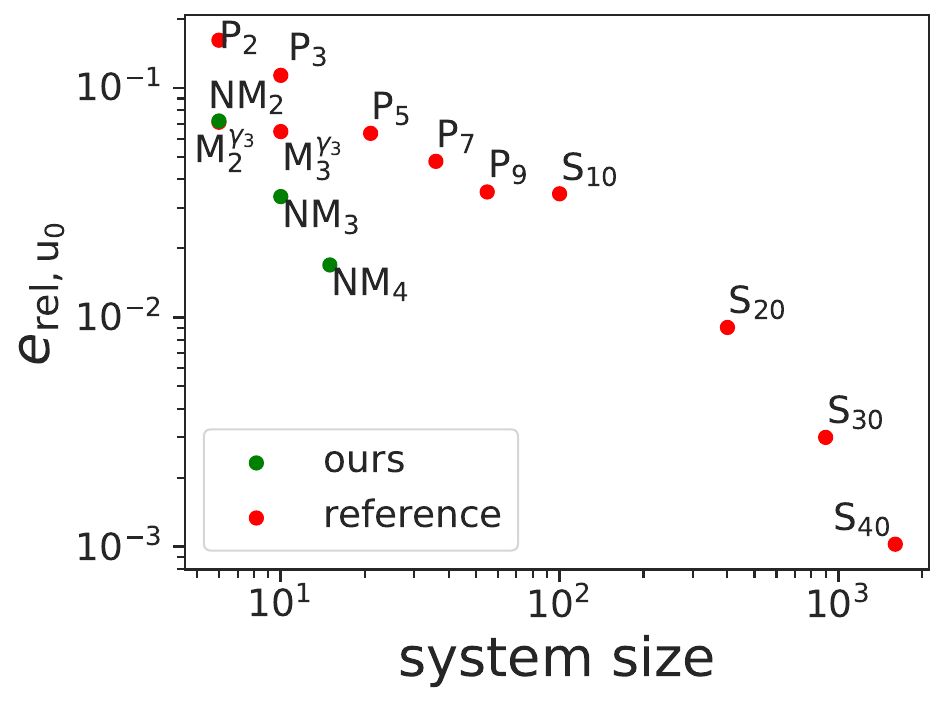}
    	\caption{$e_{\rm{rel},u_0}$ vs system size}	\label{fig_hohlraum_performancec}

	\end{subfigure}
   
    	\caption{Comparison of memory footprint, simulation wall-time, and relative simulation error of P$_N$, S$_N$, M$_N$ and ICNN-based M$_N$ (denoted by NM$_N$) solutions. Less is better. The neural network-based M$_N$ method has a particularly small memory footprint, paired with a competitive wall-time efficiency relative to the numerical error.}
	\label{fig_hohlraum_performance}
\end{figure}
Note that the computation time and memory footprint of neural network-based M$_N$ methods only depends on the size of the neural network, thus we only report the performance of one neural network-based simulation per moment order, i.e. the best neural network-based run of the reported regularization levels. The displayed errors are computed with Eq.~\eqref{eq_rel_integrated_error}. We denote the neural network-based methods by NM$_N$ in the illustration.
We see in Fig.~\ref{fig_hohlraum_performancea} that the neural network-based entropy closure accelerates the M$_N$ method to be computationally competitive, compared to the P$_N$ and S$_N$ methods.
Note, that simulation time is a function of the code. The KiT-RT framework is constructed such that spatial and temporal discretizations use the same implementation across all macroscopic methods. However, remark that an M$_N$ implementation leaves plenty of opportunities~\cite{KRISTOPHERGARRETT2015573, GarretHauck} for advanced code optimization that may improve the methods timings.

Figure~\ref{fig_hohlraum_performanceb} compares the size of the transport system and computational time of the different methods showing that even for higher order closures, neural network-based M$_N$ methods have the same simulation time as S$_{20}$ simulations, whereas the memory footprint of the neural network-based M$_N$ method is smaller by almost two orders of magnitude. Neural network-based entropy closure accelerates the Newton-based M$_N$ method by more than an order of magnitude in terms of computational time, while keeping the memory footprint the same.
Figure~\ref{fig_hohlraum_performancec} shows the simulation error over the system size of the different methods. Here, the neural network-based M$_3$ and M$_4$ methods have the best trade-off between memory footprint and simulation accuracy.

\section{Summary and conclusion}
We have presented a framework for regularized, structure-preserving neural network surrogate models for the entropy-based closure of the moment system of linear kinetic equations.
Regularization addresses the challenge of training and inference of neural network-based surrogates near the boundary of the realizable set, and thus enables the creation of robust surrogates for high-order closures in two spatial dimensions. We have provided an error analysis for regularized network approximations and put it into the context of numerical errors of commonly used  schemes for the kinetic equation.
The presented methods are tested on  synthetic and simulation test cases, with applications to radiation transport. The neural network-based entropy closure leads to a computationally competitive simulation method with an advantageous trade-off between memory footprint and numerical error, compared to the P$_N$, the traditional M$_N$, and the S$_N$ method.

\section*{Acknowledgements}
The work of Steffen Schotthöfer and Martin Frank has been funded by the Priority Programme ``Theoretical Foundations of Deep Learning (SPP2298)''  by the Deutsche Forschungsgemeinschaft. 
The work of Steffen Schotthöfer, Cory Hauck and Paul Laiu is sponsored by the Applied Mathematics Progrm at the Office of Advanced Scientific Computing Research, U.S. Department of Energy, and performed at the Oak Ridge National Laboratory, which is managed by UT-Battelle, LLC under Contract No. DE-AC05-00OR22725 with the U.S. Department of Energy. The United States Government retains and the publisher, by accepting the article for publication, acknowledges that the United States Government retains a non-exclusive, paid-up, irrevocable, world-wide license to publish or reproduce the published form of this manuscript, or allow others to do so, for United States Government purposes. The Department of Energy will provide public access to these results of federally sponsored research in accordance with the DOE Public Access Plan (http://energy.gov/downloads/doe-public-access-plan).

\bibliographystyle{abbrv}
\bibliography{main} 

\newpage
\appendix
\section{Additional material to Section~\ref{sec_kt}}\label{sec_appendix_kt}

\subsection{Proof of Theorem~\ref{theo_diverging_entropy}}

As a remineder we redefine weak convergence and a Dirac sequence as a prerequisite for the proof of Theorem~\ref{theo_diverging_entropy}.
Let  $\mu_n$, $n\in\mathbb{N}$ and  $\mu$ be measures on the interval $[-1,1]$.
The sequence of measures $\mu_n$ converges weakly to $\mu$, if $\int_{[-1,1]} \phi \intD\mu_n \xrightarrow[]{n\rightarrow\infty}\int_{[-1,1]} \phi \intD\mu\,\forall \phi\in C_b$ where $C_b$ denotes the set of bounded and continuous functions.

For the purpose of proving Theorem~\ref{theo_diverging_entropy}, we consider the sequence of (kinetic) densities  $f_n:[-1,1,]\mapsto \mathbb{R}_+$ under the Lebesque measure on $[-1,1]$, i.e. the velocity space in slab geometry, as the sequence of measures $\mu_n$. The target measure $\mu$ is here a Dirac distribution at $v^*\in [-1,1]$, which we write as $\delta_{{\fv}^*}\intD \fv$. 

Further, a sequence $f_n$ is called Dirac sequence, if 
\begin{enumerate}
    \item $\int_{[-1,1]} \phi f_n\intD\fv \xrightarrow[]{n\rightarrow\infty}\int_{[-1,1]} \phi \delta_{{\fv}^*}\intD\fv\quad\forall \phi\in C_b$
    \item  $\int_{[-1,1]} \phi f_n\intD\fv =1\quad\forall n\in\mathbb{N}$
\end{enumerate}

\begin{proof}[Proof of Theorem~\ref{theo_diverging_entropy}]
Consider the fact that $f_{\ofu}$ for $\overline{\fu}\in\partial\overline{\mathcal{R}}$ consists of a positive, linear combination of Dirac distributions for $D=[0,\infty)$ and $[-1,1]$  as the velocity space\cite{Monreal_210538,Curto_recursiveness}.

First, consider weak convergence to a single Dirac distribution $\delta_{{\fv}^*}$ at an arbitrary point $\mathbf  v^*\in [-1,1]$ .
We show that the Maxwell-Boltzmann entropy of a Dirac sequence $f_n$ converging weakly to $\delta_{{\fv}^*}$ diverges to infinity.
Consider
\begin{align}
	f_n({\fv}) =n\mathds{1}_{B_{n}^{{\fv}^*}}({\fv}), \qquad n\in\mathbb{N}
\end{align}
where $B_{n}^{{\fv}^*}\subset [-1,1]$ is a subdomain of measure $\frac{1}{n}$ and $\mathds{1}({\fv})$ is the indicator function. Then $f_n$ is a Dirac sequence.
Moreover,
\begin{align}
	\int_{[-1,1]}{f_n({\fv})\log{f_n({\fv})}-f_n({\fv})}\intD {\fv} &=\int_{[-1,1]}n\mathds{1}_{B_{n}^{{\fv}^*}}({\fv})\log{n\mathds{1}_{B_{n}^{{\fv}^*}}({\fv})}-n\mathds{1}_{B_{n}^{{\fv}^*}}({\fv})\intD {\fv} \\
	&=\int_{B_{n}^{{\fv}^*}}n\log{n}\intD {\fv} -n\int_{B_{n}^{{\fv}^*}}1\intD {\fv} \label{eq_helper1}\\
	&=\log{n}-1.
\end{align}
This term diverges to infinity as $n\rightarrow\infty$, which concludes the proof.

Now consider the case, where $f_n$ converges to a  positive combination of $p$ Dirac distributions at points $v_i^*$,  $i=1,\dots,l$. Then, the proof works analogously up to Eq.~\eqref{eq_helper1}, where we consider the sum of the integrals of the components of the linear combination. Since $v_i^*$ are finitely many distinct points, there is an $n\in\mathbb{N}$ such that $B_{n}^{{\fv}_i^*}\cup B_{n}^{{\fv}_j^*}=\emptyset$ for all $1\leq i\neq j\leq l$. Then the intregrals over the corresponding sets
 can be considered individually and the remainder of the proof follows in analogy to the case $p=1$.
\end{proof}

\section{Additional material to Section~\ref{sec_reg_moment}}\label{sec_appendix_reg_moment}

\subsection{Proof of Theorem~\ref{thm_grad_h}}
\label{sec_thm_grad_h_pf}

\begin{proof}
From the fact that ${h}^\gamma(\fu) =-\phi^\gamma(\balpha_{\fu}^\gamma;\fu)$ and the definition of $\phi^\gamma$ in \eqref{eq_entropyDualOCP_part_reg_objectF}, the gradient of $h^\gamma$ can be expressed as
\begin{equation}\label{eq_grad_h}
\nabla_{\fu} {h}^\gamma(\fu) 
= \balpha_{\fu}^\gamma 
- \frac{\gamma}{2}\big[\|(\bm\alpha_{\fu}^\gamma)_{\#}\|^2,\mathbf{0}^\Tr\big]^\Tr 
- 
\nabla_{\fu} \balpha_{\fu}^\gamma \left(\nabla_{\balpha_{\fu}^\gamma}{\phi}^\gamma(\balpha_{\fu}^\gamma; \fu )\right).
\end{equation}
Since $\balpha_{\fu}^\gamma$ is defined as the minimizer of $\phi^\gamma$ in \eqref{eq_entropyDualOCP_part_reg}, $\nabla_{\balpha_{\fu}^\gamma}{\phi}^\gamma(\balpha_{\fu}^\gamma; \fu )=0$.  Thus the last term in \eqref{eq_grad_h} vanishes, which proves the claim.
\end{proof}

\subsection{Proof of Lemma~\ref{lem_alpha_0}}
\label{sec_lem_alpha_0_pf}

\begin{proof}
Recall that, by definition, $\overline{u}_0=1$. The first order optimality condition of 
\eqref{eq_entropyDualOCP_part_reg} leads to
\begin{align}
	\nabla_{\bm\alpha} \phi^\gamma\left(\balpha^\gamma_{\ofu};\overline{\fu}\right)=\inner{{\fm}\exp\left( \balpha^\gamma_{\ofu}\cdot {\fm}\right)}-\overline{\fu}+{\gamma}\big[0,\left(\balpha^\gamma_{\ofu}\right)_\#^\Tr\big]^\Tr = 0\:.
\end{align}
The first element of this vector equation implies that
$
  \inner{m_0\exp\left( \balpha^\gamma_{\ofu}\cdot {\fm}\right)}=\overline{u}_0 =1
$.
Since $\balpha_{\ofus}^\gamma= \left[\alpha^\gamma_{\overline{\fu},0},\left(\balpha^\gamma_{\ofu}\right)^\top \right]^\top$, 
\begin{equation}
	1=\inner{m_0\exp\left( \balpha^\gamma_{\ofu}\cdot {\fm}\right)} = \exp\left(m_0 \alpha^\gamma_{\overline{\fu},0}\right)\inner{m_0\exp\left(\left(\balpha^\gamma_{\ofu}\right)_\#\cdot \mathbf{m}_\#\right)},
\end{equation}
which, together with the definition of $\vartheta$ given in Eq.~\eqref{eq_ref_alpha_reduction}, proves the claim.
\end{proof}
\subsection{Proof of Lemma~\ref{lem_reduced_dual_obj}}
\label{sec_lem_reduced_dual_obj_pf}

\begin{proof}
The first claim that
\begin{equation}\label{eq_hatphi}
	\hat{\phi}^\gamma(\bbeta;\mathbf w)=  \frac{1}{m_0} +\frac{1}{m_0}\big(\log(m_0) +\log( \inner{\exp\left({\bbeta}\cdot{\fm}_\#\right)})\big) -{\bbeta}\cdot \mathbf w + \frac{\gamma}{2} \norm{\bbeta}^2,
\end{equation}
follows directly from the definitions of $\hat{\phi}^\gamma$, $\phi^\gamma$, and $\vartheta$ given in Eqs.~\eqref{eq_def_reduced_obj_func}, \eqref{eq_entropyDualOCP_part_reg_objectF}, and \eqref{eq_ref_alpha_reduction}, respectively.
From Eq.~\eqref{eq_hatphi}, it is clear that $\hat{\phi}^\gamma$ is twice differentiable with respect to $\bbeta$. To show that $\hat{\phi}^\gamma$ is convex, we use the fact that sums of convex functions are convex and prove that $\log\left(\inner{\exp(\bbeta\cdot {\fm}_\#}\right)$ satisfies Jensen's inequality, thus is convex.
We prove Jensen's inequality by using the monotonicity of the logarithm and Hölder's inequality with $1/p=t$ and $1/q=(1-t)$. Specifically, for $t\in(0,1)$ and ${\bbeta}_1,{\bbeta}_2\in\mathbb{R}^{n}$ with $\bbeta_1\not= \bbeta_2$, we have
\begin{align*}
	\log(\inner{\exp((t{\bbeta}_1 + (1-t){\bbeta}_2)\cdot\mathbf  m_\#)})
	=& \log(\inner{\exp(t{\bbeta}_1\cdot {\fm}_\#)\exp((1-t){\bbeta}_2\cdot{\fm}_\#)}) &&\\
	\leq& \log(\inner{\exp({\bbeta}_1\cdot{\fm}_\#)}^t\inner{\exp({\bbeta}_2\cdot{\fm}_\#)}^{(1-t)}) && \text{(Hölder's ineq.)}\\
	=& \log(\inner{\exp({\bbeta}_1\cdot {\fm}_\#)}^t)+\log(\inner{\exp({\bbeta}_2\cdot{\fm}_\#)}^{(1-t)}) && \\
	=& t\log(\inner{\exp({\bbeta}_1\cdot\mathbf  m_\#)})+(1-t)\log(\inner{\exp({\bbeta}_2\cdot{\fm}_\#)}). &&
\end{align*}
Further, for $\gamma>0$, $\hat{\phi}^\gamma$ is strictly convex since $ \frac{\gamma}{2} \norm{\bbeta}^2$ is strictly convex.%
\footnote{It turns out that the function $\hat{\phi}^\gamma(\cdot;\mathbf w)$ is strictly convex even when $\gamma=0$.  This can easily be shown by verifying that H\"olders inequality is strict unless $\bbeta_1 = \bbeta_2$.}
\end{proof}

\subsection{Proof of Lemma~\ref{lem_reduced_dual_min}}
\label{sec_lem_reduced_dual_min_pf}
\begin{proof}
Since $\hat{\phi}^\gamma$ is strictly convex, there exists a unique minimizer $\bbeta_{\ofus}^\gamma$ in \eqref{eq_reduced_reg_closure} that satisfies the first order optimality condition, i.e.,
\begin{equation}\label{eq_reg_reduced_FOOC}
	\begin{alignedat}{2}
    0=\nabla_{\bbeta} \hat{\phi}^\gamma(\bbeta_{\ofus}^\gamma; \ofus) &=  \frac{1}{m_0}{\big\langle{\exp(\bbeta_{\ofus}^\gamma\cdot {\fm}_\#)}\big\rangle}^{-1}
    \big\langle{{\fm}_\#\exp(\bbeta_{\ofus}^\gamma\cdot {\fm}_\#)}\big\rangle  - \ofus + \gamma\bbeta_{\ofus}^\gamma\\
    &=  \big\langle{{\fm}_\#\exp([\vartheta(\bbeta_{\ofus}^\gamma),(\bbeta_{\ofus}^\gamma)^\Tr]^\Tr\cdot {\fm})}\big\rangle  - \ofus + \gamma{\bbeta}_{\ofus}^\gamma.
	\end{alignedat}
\end{equation}
On the other hand, the first order optimality condition of \eqref{eq_entropyDualOCP_part_reg} is given by
\begin{align}\label{eq_reg_FOOC}
	0=\nabla_{\balpha} {\phi}^\gamma(\balpha_{\ofu}^\gamma; {\ofu}) 
 = \inner{{\fm}\exp( \balpha_{\ofu}^\gamma \cdot {\fm})} - {{\ofu}} + \gamma\left[0,(\balpha_{\ofu}^\gamma)_\#^\Tr\right]^\Tr.
\end{align}
Using Eq.~\eqref{eq_reg_reduced_FOOC} and Lemma~\ref{lem_alpha_0}, it is straightforward to verify that $[\vartheta(\bbeta_{\ofus}^\gamma),(\bbeta_{\ofus}^\gamma)^\Tr]^\Tr$ satisfies the optimality condition in Eq.~\eqref{eq_reg_FOOC}.  The equivalence of $\balpha_{\ofu}^\gamma$ and $[\vartheta(\bbeta_{\ofus}^\gamma),(\bbeta_{\ofus}^\gamma)^\Tr]^\Tr$ then follows directly from the strict convexity of ${\phi}^\gamma$, which implies uniqueness of the minimizer.
Finally, the fact that $\phi^\gamma(\balpha_{\ofu}^\gamma;{\ofu})= \hat{\phi}^\gamma({\bbeta}_{\ofus}^\gamma;\ofus)$ is a direct consequence of the definition of $\hat{\phi}^\gamma$ in Eq.~\eqref{eq_def_reduced_obj_func}.
\end{proof}

\subsection{Proof of Lemma~\ref{lem_hessian_cond}}
\label{sec_lem_hessian_cond_pf}
\begin{proof}
It follows from the formulation of $\hat{\phi}^\gamma$ in \eqref{eq_reduced_obj_reformulation} that 
$
	\hat{H}^\gamma(\bbeta) = \hat{H}^{\gamma=0}(\bbeta) + \gamma I.
$
Thus, the condition number of $\hat{H}^\gamma(\bbeta)$ can be written as 
\begin{equation}
    \text{cond}(\hat{H}^\gamma(\bbeta)) = 
    \frac{\hat{\lambda}_{\max} + \gamma }{\hat{\lambda}_{\min} + \gamma },
\end{equation} 
where $\hat{\lambda}_{\max}$ and $\hat{\lambda}_{\min}$ are the maximum and minimum eigenvalues of $\hat{H}^{\gamma=0}(\bbeta)$, respectively. 
According to Lemma~\ref{lem_reduced_dual_obj}, $\hat{\phi}^\gamma$ is strictly convex even when $\gamma=0$. 
Hence, $\hat{\lambda}_{\min}\geq0$, and $\text{cond}(\hat{H}^\gamma(\bbeta))$ is upper bounded by $1+\gamma^{-1}\hat{\lambda}_{\max}$.
\end{proof}

\subsection{Proof of Theorem~\ref{thm_reg_rescale}}
\label{sec_thm_reg_rescale_pf}

\begin{proof}
According to Lemma~\ref{lem_reduced_dual_min}, $\bm\alpha_{\ofu}^\gamma = [\vartheta({\bbeta}_{\ofus}^\gamma),({\bbeta}_{\ofus}^\gamma)^\Tr]^\Tr$. 
Thus, we first show that $\bm\alpha_{{\fu}}^\gamma = [\alpha_{\overline{\fu},0}^\gamma + \frac{\log \,u_0}{m_0}, (\bm\alpha_{\ofu}^\gamma)_{\#}^\Tr]^\Tr$.
The first order optimality condition of \eqref{eq_entropyDualOCP_part_reg} at $\fu$ and $\overline{\fu}$ (see, e.g., \eqref{eq_reg_FOOC}) gives
\begin{align}\label{eq_reg_FOOC_2}
{{{\fu}}} = 
\inner{{\fm}\exp( \balpha_{{\fu}}^\gamma \cdot {\fm})} + u_0\gamma\left[0,(\balpha_{{\fu}}^\gamma)_\#^\Tr\right]^\Tr
\quad\text{and}\quad
{{\ofu}} = 
\inner{{\fm}\exp( \balpha_{\ofu}^\gamma \cdot {\fm})} + \gamma\left[0,(\balpha_{\ofu}^\gamma)_\#^\Tr\right]^\Tr,
\end{align}
respectively.
Multiplying the second equation above by $u_0$ yields
\begin{equation}
\begin{alignedat}{2}
\fu = u_0 \, \overline{\fu} &= 
\inner{u_0\,{\fm}\exp( \balpha_{\ofu}^\gamma \cdot {\fm})} + u_0\gamma\left[0,(\balpha_{\ofu}^\gamma)_\#^\Tr\right]^\Tr\\
&= 
\inner{{\fm}\exp\left( \left[\alpha_{\overline{\fu},0}^\gamma + \frac{\log \,u_0}{m_0}, (\balpha_{\ofu}^\gamma)_{\#}^\Tr \right]^\Tr \cdot {\fm}\right)} + u_0\gamma\left[0,(\balpha_{\ofu}^\gamma)_\#^\Tr\right]^\Tr.
\end{alignedat}
\end{equation}
Therefore, $[\alpha_{\overline{\fu},0}^\gamma + \frac{\log \,u_0}{m_0}, (\bm\alpha_{\ofu}^\gamma)_{\#}^\Tr]^\Tr$ satisfies the optimality condition at $\fu$ given in \eqref{eq_reg_FOOC_2}, we have $\bm\alpha_{{\fu}}^\gamma = [\alpha_{\overline{\fu},0}^\gamma + \frac{\log \,u_0}{m_0}, (\bm\alpha_{\ofu}^\gamma)_{\#}^\Tr]^\Tr$.

We next show that $\hat{h}^\gamma\colon \mathbb{R}^{n}\to\mathbb{R}$ is strictly convex. By the definition of $\hat{h}^\gamma$ and the formulation of $\hat{\phi}^\gamma$ in \eqref{eq_reduced_obj_reformulation},
\begin{equation}\label{eq_hat_h}
\begin{alignedat}{2}    \hat{h}^\gamma(\overline{\fu}_{\#}) 
    &= \sup_{{\bbeta}\in\mathbb{R}^{n}} 
    \left[
      -\hat{\phi}^\gamma(\bbeta;\ofus) \right]
      =\sup_{{\bbeta}\in\mathbb{R}^{n}} 
    \left[
     {\bbeta} \cdot \overline{\fu}_{\#} - \bigg(\hat{\phi}^\gamma(\bbeta;\ofus)  +{\bbeta} \cdot \overline{\fu}_{\#}\bigg)\right] \:.
\end{alignedat}
\end{equation}
Thus $\hat{h}^\gamma$ is the Legendre dual of the function $\bbeta \mapsto \hat{\phi}^\gamma(\bbeta;\ofus)  +{\bbeta} \cdot \overline{\fu}_{\#}$,
which is strictly convex by Lemma~\ref{lem_reduced_dual_obj}. Therefore, $\hat{h}^\gamma$ is also strictly convex.
As for the gradient of $\hat{h}^\gamma$, since $\hat{h}^\gamma(\fw):= -\hat{\phi}^\gamma({\bbeta}_{\fw}^\gamma;\fw)$ and $-\nabla_{\fw} \hat{\phi}^\gamma( \bbeta;\fw) = \bbeta$,
\begin{equation}\label{eq_grad_h_hat}
\nabla_{\fw}\hat{h}^\gamma(\fw) 
={\bbeta}_{\fw}^\gamma 
-\nabla_{\fw}{\bbeta}_{\fw}^\gamma \left(\nabla_{\bbeta_{\fw}^\gamma}\hat{\phi}^\gamma(\bbeta_{\fw}^\gamma;\fw)\right).
\end{equation}
Since $\bbeta_{\fw}^\gamma$ minimizes $\hat{\phi}^\gamma$, the last term in \eqref{eq_grad_h_hat} vanishes, which proves the claim.

To derive the relation between $h^\gamma$ and $\hat{h}^\gamma$, see Definition~\ref{def_entropy} and~\ref{def_reduced_values}, we start with
\begin{equation}\label{eq_h}
    h^\gamma(\fu) = - \phi^{\gamma}(\bm\alpha_{{\fu}}^\gamma; \fu) 
    = - \inner{\exp(\bm\alpha_{{\fu}}^\gamma \cdot \mathbf{m})} + \bm\alpha_{{\fu}}^\gamma\cdot \fu - \frac{u_0\gamma}{2} \|(\bm\alpha_{{\fu}}^\gamma)_{\#}\|^2,
\end{equation}
see Eq.~\eqref{eq_entropy_phi_def}. Plugging $\bm\alpha_{{\fu}}^\gamma = [\alpha_{\overline{\fu},0}^\gamma + \frac{\log \,u_0}{m_0}, (\bm\alpha_{\ofu}^\gamma)_{\#}^\Tr]^\Tr$, where $\bm\alpha_{\ofu}^\gamma = [\vartheta({\bbeta}_{\ofus}^\gamma),({\bbeta}_{\ofus}^\gamma)^\Tr]^\Tr$, into the above equation then gives
\begin{equation}\label{eq_h_extension}
\begin{alignedat}{2}
    h^\gamma(\fu)  &= - \inner{\exp\left(\left[\alpha_{\overline{\fu},0}^\gamma + \frac{\log \,u_0}{m_0}, (\bm\alpha_{\ofu}^\gamma)_{\#}^\Tr\right]^\Tr \cdot \mathbf{m}\right)} + \,\left[\alpha_{\overline{\fu},0}^\gamma + \frac{\log \,u_0}{m_0}, (\bm\alpha_{\ofu}^\gamma)_{\#}^\Tr\right]^\Tr\cdot \overline{\fu} u_0 - \frac{u_0\gamma}{2} \|(\bm\alpha_{\ofu}^\gamma)_{\#}\|^2\\
    &= - u_0 \inner{\exp(\bm\alpha_{\ofu}^\gamma \cdot \mathbf{m})} + \frac{u_0}{m_0}\log \,u_0 + u_0\,\bm\alpha_{\ofu}^\gamma\cdot \overline{\fu} - \frac{u_0\gamma}{2} \|(\bm\alpha_{\ofu}^\gamma)_{\#}\|^2\\
    &= - \frac{u_0}{m_0} + \frac{u_0}{m_0}\log \,u_0 + u_0 \vartheta({\bbeta}_{\ofus}^\gamma) + u_0\,{\bbeta}_{\ofus}^\gamma \cdot \overline{\fu}_{\#} - \frac{u_0\gamma}{2} \|{\bbeta}_{\ofus}^\gamma\|^2\\
    &= u_0 \hat{h}^\gamma(\overline{\fu}_{\#}) + \frac{u_0}{m_0}\log \,u_0\:,
\end{alignedat}
\end{equation}
where the last equality uses the first equality in \eqref{eq_hat_h} and the definition of $\vartheta$ in \eqref{eq_ref_alpha_reduction}.
Finally, $h^\gamma(\overline{\fu})= \hat{h}^\gamma(\overline{\fu}_{\#})$ is a direct consequence of the relation between $h^\gamma$ and $\hat{h}^\gamma$ given in \eqref{eq_h_extension}. 
\end{proof}

\subsection{Proof of Theorem~~\ref{theo_reg_ansatz_error}}
\label{sec_theo_reg_ansatz_error_pf}
\begin{proof}


First, it follows from \eqref{eq_alpha_reconstructor} and \eqref{eq_scaled_alpha} that 
$(\balpha_\fu^\gamma)_{\#} = (\balpha_{\ofu}^\gamma )_{\#} = {\bbeta}_{\ofus}^\gamma
$,
which when combined with the first equation in \eqref{eq_reg_FOOC_2} implies that
%
\begin{equation}\label{eq_part_reg_moment_mult}
    \widetilde{\fu}^\gamma = \fu - u_0 \gamma \, [0,({\bbeta}^{\gamma}_{\ofus})^\Tr]^\Tr.
\end{equation}
Since ${\bbeta}_{\ofus}^\gamma\in B_M$, \eqref{eq_regularization_error} immediately holds.
In addition, since $\fg^\gamma_{\fu}= \balpha_{\fu}^\gamma - \frac{\gamma}{2}\big[\|(\balpha_{\fu}^\gamma)_{\#}\|^2,\mathbf{0}^{\Tr}\big]^{\Tr}$ (see \eqref{eq_grad_h_thm} in Theorem~\ref{thm_grad_h}), 
\begin{equation}\label{eq_part_reg_moment_grad}
    \fu^\gamma 
    =\langle \mathbf{m} f_{\fu}^\gamma \rangle 
    = \langle \mathbf{m} \exp( \fg^\gamma_{\fu} \cdot \fm) \rangle 
    = \langle \mathbf{m} \exp( \balpha^\gamma_{\fu} \cdot \fm) \rangle\exp\big(-\frac{\gamma}{2}\|{\bbeta}_{\ofus}^\gamma\|^2 m_0\big)
    = \exp\big(-\frac{\gamma}{2}\|{\bbeta}_{\ofus}^\gamma\|^2 m_0 \big) \widetilde{\fu}^\gamma.
\end{equation}
Further, \eqref{eq_part_reg_moment_grad} implies
\begin{align}\label{eq_helperxyz}
	\|\fu^\gamma-\widetilde{\fu}^\gamma\|=\big(1-\exp(-\frac{\gamma}{2}\|{\bbeta}_{\ofus}^\gamma\|^2m_0)\big)|\|\widetilde{\fu}^\gamma\|
	\leq \big(1-\exp(-\frac{\gamma}{2}M^2m_0)\big)\|\widetilde{\fu}^\gamma\|
 \leq  \big(1-\exp(-\frac{\gamma}{2}M^2m_0)\big) \,u_0\,(n+1).
\end{align}
The fact that
\begin{equation}
    \max\menge{\abs{m_i(\fv)} : \fv\in\mathbb{S}^2, i=0,\dots, n}\leq \left(\frac{n+1}{\abs{\mathbb{S}^2}}\right)^{1/2},
\end{equation}
see~\cite[Section 2.1.3]{atkinson2012spherical}, where $\abs{\mathbb{S}^2}$ is the measure of the unit sphere. Normalization of the moment vector and the triangle inequality gives
\begin{equation}
    \norm{\frac{\widetilde{\fu}^\gamma}{u_0}}\leq \frac{(n+1)^{3/2}}{\abs{\mathbb{S}^2}^{1/2}}
\end{equation}
which yields the last inequality of Eq.~\eqref{eq_helperxyz}.
The bound in \eqref{eq:ugamma-diffs} a direct consequence of the triangle inequality. 
\end{proof}

\subsection{Proof of Theorem~\ref{theo_entropy_dissipation}}
\label{sec_theo_entropy_dissipation_pf}
\begin{proof} 
For entropy dissipation, we first show that the integrability condition holds for the entropy/entropy-flux pair $h^\gamma$ and $\mathbf{j}^\gamma$ with the flux term 
$\mathbf{F}(\fu)=\inner{\mathbf{v}\otimes\mathbf{m}f^\gamma_{\fu}}$.
For $i=1,2,3$, let $j_{i}^\gamma:=\inner{v_i \eta(f^\gamma_{\fu})}$ and $F_i(\fu):=\inner{v_i\,\mathbf{m}f^\gamma_{\fu}}$, then 
\begin{equation}\label{eq_entropy_entropyflux_pair}
	\nabla_\fu j^\gamma_{i}(\fu)= \inner{v_i\nabla_\fu \eta(f^\gamma_{\fu})}
	= \inner{v_i\eta^\prime(f^\gamma_{\fu})\nabla_\fu f^\gamma_{\fu}} 
	= \inner{v_i\eta^\prime(\eta_*^\prime(\mathbf{g}^\gamma_{\fu}\cdot\mathbf{m}))\nabla_\fu f^\gamma_{\fu}} 
	= \mathbf{g}^\gamma_{\fu}\cdot\inner{v_i\mathbf{m}\,\nabla_\fu f^\gamma_{\fu}}
=\nabla_\fu h^\gamma(\fu)\cdot \nabla_\fu F_i(\fu),
\end{equation}
where we have used the chain rule, the definition of $f^\gamma_{\fu}$ in Eq.~\eqref{eq_ansatzes}, the Legendre duality, and the fact that $\mathbf{g}^\gamma_{\fu}= \nabla_\fu h^\gamma(\fu)$.
To derive the entropy dissipation law, we multiply the moment system from the left with $\nabla_\fu h^\gamma(\fu)$ and obtain
\begin{equation}
	\nabla_\fu h^\gamma(\fu)\cdot\partial_t \fu + \nabla_\fu h^\gamma(\fu)\cdot\sum_{i=1}^3\partial_{{x}_i}F_i(\fu)= \nabla_\fu h^\gamma(\fu)\cdot\inner{\mathbf{m} Q( {f}_{\fu}^\gamma )}\:.
\end{equation}
From the reverse chain rule and Eq.~\eqref{eq_entropy_entropyflux_pair}, this leads to 
\begin{equation}
	\partial_t h^\gamma(\fu) + \nabla_\mathbf{x}\cdot \bm j^\gamma(\fu)= \nabla_\fu h^\gamma(\fu)\cdot\inner{\mathbf{m} Q( {f}_{\fu}^\gamma)}.
\end{equation}
Entropy dissipation requires non-positivity of right-hand side term, which is given by
\begin{equation}
\nabla_\fu h^\gamma(\fu)\cdot\inner{\mathbf{m} Q( {f}_{\fu}^\gamma )}
= \inner{\mathbf{g}^\gamma_{\fu}\cdot\mathbf{m} Q( {f}_{\fu}^\gamma )}
= \inner{\eta^\prime(\eta_*^\prime(\mathbf{g}^\gamma_{\fu}\cdot\mathbf{m})) Q( {f}_{\fu}^\gamma )}
=\inner{\eta^\prime({f}_{\fu}^\gamma) Q( {f}_{\fu}^\gamma )}\leq0,
\end{equation}
where the equalities follow from similar properties used in Eq.~\eqref{eq_entropy_entropyflux_pair} in the reversed order, and the inequality is from the entropy dissipation law in Eq.~\eqref{eq_kinetic_entropy_dissipation}.

We next show that the closed moment system is hyperbolic. Note that the hyperbolicity follows from the entropy dissipation property of the moment system.
Consider again the flux function of the moment system
\begin{align}
	\mathbf{F}(\fu)=\inner{\mathbf{v}\otimes\mathbf{m}f^\gamma_{\fu}} =\inner{\mathbf{v}\otimes\mathbf{m}\,\eta'_*\left(\mathbf{g}_{\fu}^\gamma\cdot\mathbf{m}\right)}.
\end{align}
We define $j_{i,*}(\mathbf{g})=\inner{v_i\eta_*(\mathbf{g}\cdot \fm)}$, $i=1, 2, 3$, such that
$\nabla_{\mathbf{g}}j_{i,*}(\mathbf{g}^\gamma_{\fu})=F_i(\fu)$, and denote
\begin{align}
	J_i =  {\nabla_\fg^2}j_{i,*}(\mathbf{g}^\gamma_{\fu}),\quad i = 1,2,3,\quad\text{and}\quad
    K = \nabla_\fu^2 h^\gamma(\fu) = \nabla_\fu \mathbf{g}^\gamma_{\fu},
\end{align}
where $K$ is symmetric positive definite by the strict convexity of $h^\gamma(\fu)$.
By using chain rule, the flux term in Eq.~\eqref{eq_hyper_moment_sys} can be written as
\begin{equation}
    \nabla_{\mathbf{x}}\inner{\mathbf{v}\otimes\mathbf{m}f^\gamma_{\fu}} = \sum_{i=1}^3 \partial_{x_i} F_i(\fu) = \sum_{i=1}^3 \partial_{x_i} (\nabla_{\mathbf{g}}j_{i,*}(\mathbf{g}^\gamma_{\fu})) = \sum_{i=1}^3 {\nabla_\fg^2}j_{i,*}(\mathbf{g}^\gamma_{\fu}) \cdot \nabla_\fu \mathbf{g}^\gamma_{\fu} \cdot \partial_{x_i}\fu
    = \sum_{i=1}^3  J_i K \partial_{x_i}\fu\:.
\end{equation}
The symmetric positive definite matrix $K$ then symmetrizes \cite{friedrichs1971systems} the closed moment system \eqref{eq_hyper_moment_sys}; that is, after multiplication on the left by $K$, 
\begin{equation}
    K\partial_t\fu +  \sum_{i=1}^3 K J_i K \partial_{x_i}\fu\: = K \inner{\mathbf{m} Q( {f}_{\fu}^\gamma)},
\end{equation}
where the matrix $K J_i K$ is symmetric for each $i \in \{1,2,3\}$.

\end{proof}

\end{document}